\documentclass[11pt]{amsart}

\usepackage{amsfonts} 
\usepackage{amsmath} 
\usepackage{amssymb} 
\usepackage{amsthm} 
\usepackage{mathtools}
\usepackage{amscd}
\usepackage{booktabs}
\usepackage{color}
\usepackage{import}
\usepackage{wasysym} 

\usepackage[margin=3.5cm]{geometry} 

\usepackage[hidelinks,pagebackref,pdftex]{hyperref} 
\usepackage{makecell} 
\usepackage{multicol}
\usepackage{graphicx}

\usepackage{kbordermatrix} 

\renewcommand*{\backref}[1]{}
\renewcommand*{\backrefalt}[4]{
  \ifcase #1
  [No citations.]
  \or [#2]
  \else [#2]
  \fi }

\AtBeginDocument{%
   \def\MR#1{}
}

\numberwithin{equation}{section}
\theoremstyle{plain}
\newtheorem{theorem}[equation]{Theorem}
\newtheorem{thm}[equation]{Theorem}

\newtheorem{lemma}[equation]{Lemma}
\newtheorem{lem}[equation]{Lemma}

\newtheorem{prop}[equation]{Proposition}

\newtheorem*{namedtheorem}{\theoremname}
\newcommand{\theoremname}{testing}
\newenvironment{named}[1]{\renewcommand{\theoremname}{#1}\begin{namedtheorem}}{\end{namedtheorem}}

\theoremstyle{definition}
\newtheorem{definition}[equation]{Definition}
\newtheorem{remark}[equation]{Remark}

\newtheorem{example}[equation]{Example}

\newcommand{\0}{{\bf 0}}

\renewcommand{\a}{\mathbf{a}}

\renewcommand{\b}{\mathbf{b}}

\newcommand{\cc}{\mathbf{c}}

\newcommand{\CC}{\mathbb{C}}

\newcommand{\calE}{\mathcal{E}}

\newcommand{\mfc}{\mathfrak{c}}
\newcommand{\mfl}{\mathfrak{l}}
\newcommand{\mfm}{\mathfrak{m}}

\newcommand{\p}{\mathbf{p}}

\newcommand{\tet}{\mathbf{t}}

\newcommand{\TT}{\mathcal{T}}
\newcommand{\calT}{\mathcal{T}}

\newcommand{\V}{\mathcal{V}}

\DeclareMathOperator{\Id}{Id}

\newcommand{\In}{\mathrm{In}}
\newcommand{\NZ}{\mathrm{NZ}}
\newcommand{\SY}{\mathrm{SY}}
\newcommand{\PSL}{{\mathrm{PSL}}}
\newcommand{\PGL}{{\mathrm{PGL}}}

\newcommand{\bdy}{\partial}

\newcommand{\refsec}[1]{Section~\ref{Sec:#1}}
\newcommand{\refdef}[1]{Definition~\ref{Def:#1}}
\newcommand{\reffig}[1]{Figure~\ref{Fig:#1}}

\newcommand{\refeqn}[1]{\eqref{Eqn:#1}}
\newcommand{\reflem}[1]{Lemma~\ref{Lem:#1}}
\newcommand{\refprop}[1]{Proposition~\ref{Prop:#1}}
\newcommand{\refthm}[1]{Theorem~\ref{Thm:#1}}

\begin{document}

\title[A symplectic basis]{A symplectic basis for 3-manifold triangulations} 

\author{Daniel V. Mathews} 
\address{School of Mathematics,
Monash University,
VIC 3800, Australia}
\email{Daniel.Mathews@monash.edu}

\author{Jessica S. Purcell}
\address{School of Mathematics,
Monash University,
VIC 3800, Australia}
\email{jessica.purcell@monash.edu}


\begin{abstract}
In the 1980s, Neumann and Zagier introduced a symplectic vector space associated to an ideal triangulation of a cusped 3-manifold, such as a knot complement. We give a geometric interpretation for this symplectic structure in terms of the topology of the 3-manifold, via intersections of certain curves on a Heegaard surface. We also give an algorithm to construct curves forming a symplectic basis.
\end{abstract}

\maketitle


\section{Introduction}

Triangulated 3-manifolds that admit a complete hyperbolic structure must satisfy certain equations, called gluing and cusp equations, first investigated by Thurston~\cite{thurston}. 
In the 1980s, Neumann and Zagier observed that the gluing and cusp equations satisfy certain interesting symplectic relations~\cite{NeumannZagier}. 
This was made geometric: the symplectic form arises from properties of the underlying triangulation; see Choi~\cite{Choi:NZ} and Neumann~\cite{Neumann}. 

This result has had numerous applications; we name only a few here. Neumann and Zagier originally used the symplectic form to show that the solution space to gluing and cusp equations has positive rank, and obtained a bound on volumes of Dehn fillings 
by interpolating a bound on volume before Dehn filling~\cite{NeumannZagier}. Choi used the symplectic form to prove that the geometric solution space to gluing and cusp equations is a smooth complex manifold~\cite{Choi:PosOr}. The form may be reinterpreted as acting on a vector space of normal quadrilaterals embedded in 
tetrahedra, for example in work of Luo~\cite{Luo:Triangulated3Mflds} and~\cite{Luo:VolNormSfces}. Thus it has relations with normal surface theory. 
Integer solutions to gluing equations pick out solutions to matching equations of quadrilaterals, forming normal surfaces and spun normal surfaces; see Garoufalidis, Hodgson, Hoffman, and Rubinstein~\cite{GHHR:3DIndex}. 
Such solutions exist due to work of Neumann~\cite{Neumann}, reinterpreted in \cite{GHHR:3DIndex}. 
Non-degenerate solutions to the gluing equations give representations of the fundamental group into $\PSL(2,\CC)$. This has been extended to representations into $\PGL(n,\CC)$, with a symplectic form there, e.g.\ by Garoufalidis and Zickert~\cite{GZ:Symplectic2016}. There are additionally many applications in quantum topology and perturbative Chern-Simons theory; see, for example~\cite{Neumann:BlochCS, DimofteGaroufalidis, DGLZ:ChernSimons, GHRS:Index}.

Neumann and Zagier in~\cite{NeumannZagier} showed that, under their symplectic form $\omega$, the gluing relations give rise to vectors that are symplectically orthogonal. For each cusp, the cusp relations give rise to a portion of a symplectic basis: two curves $\mu$, $\lambda$ forming a homology basis on each cusp give rise to vectors $h(\mu)$, $h(\lambda)$ that are essentially (up to a factor of 2) symplectic duals: $\omega(h(\lambda),h(\mu)) = \pm 2$.
Here, the function $h$ is the \emph{combinatorial holonomy} function of Thurston, Neumann and Zagier, giving a sum of tetrahedron parameters.  What is missing, and has been missing since the form's introduction, is a dual analogue to the gluing vectors. That is, there has been no geometric description of vectors that are symplectic duals to those vectors arising from gluing relations. 
Such vectors can be determined algebraically, using a symplectic version of the Gram-Schmidt algorithm. However, such a process has no real connection to the 3-manifold triangulation, and it is not clear whether or how the resulting vectors interact with the underlying geometry and topology of the manifold.

In this paper, we give a geometric algorithm to determine symplectic dual vectors to the gluing equations, completing the cusp and gluing relations to a full symplectic basis. Our algorithm is determined by the triangulation of the 3-manifold alone. We find a collection of \emph{oscillating curves} that are geometric duals to curves encoding gluing equations. 
We define oscillating curves in \refsec{train_tracks_stations}; roughly, they are curves that proceed along $\partial M$ but may ``dive through" the interior of $M$ along edges of the triangulation, oscillating their orientation in the process.
More generally, we define \emph{abstract oscillating curves} using the geometry of triangulations. We show that the Neumann--Zagier symplectic form extends to a symplectic form on abstract oscillating curves, picking out their intersection properties. A main result is the following.

\begin{named}{\refthm{MainThm}}
Let $M$ be an open 3-manifold, homeomorphic to the interior of a compact 3-manifold $\overline{M}$,
such that $M$ is equipped with an ideal triangulation. Let $\omega$ denote the Neumann--Zagier symplectic form. For any abstract oscillating curves $\zeta$, $\zeta'$ in $M$,
\[
\omega( h(\zeta), h(\zeta')) = 2 \zeta\cdot \zeta',
\]
where $\zeta\cdot\zeta'$ denotes intersection number. 
\end{named}

We will define all the requisite notions in due course.

We note that this theorem applies more generally than to the case of finite volume hyperbolic 3-manifolds, which is the classical setting for the Neumann--Zagier symplectic form. In that setting, $\bdy\overline{M}$ must consist of tori, but \refthm{MainThm} applies to any ideal triangulation, meaning any triangulation of $M$ with vertices removed, which can be truncated to give a decomposition of $\overline{M}$ into truncated tetrahedra. Thus the frontier of $M$ may consist of any finite collection of oriented closed compact surfaces, and there may be components of any finite genus. 
Indeed, while Thurston and Neumann--Zagier were initially interested in hyperbolic structures, some methods work equally well for any triangulation pairing faces to faces, as in \cite{Neumann}.  For example, these results hold even when an ideal edge is homotopic into the boundary.

In the work of Neumann and Zagier, as interpreted by Choi, the gluing and cusp equations are represented by curves on the boundary of $M$. These are oscillating curves in the framework of \refthm{MainThm}.
Indeed, the proof of \refthm{MainThm} gives a new proof of symplectic properties of gluing and cusp relations, and extends them to larger classes of curves. It additionally gives a new proof of several results observed by Neumann and Zagier, and reframes them in a more geometric context.
For example, if the frontier of $M$ consists of $n_\mfc$ torus cusps and the triangulation consists of $n$ ideal tetrahedra, Neumann and Zagier showed that $n-n_\mfc$ gluing equations are linearly independent. Here, we pick out the linearly independent vectors explicitly and constructively.

\begin{thm}
\label{Thm:OscAlgorithmTori}
When the frontier of $M$ consists of $n_\mfc$ tori,
there exists a constructive algorithm to determine oscillating curves $C_1, \dots, C_{n-n_\mfc}$ associated with edge gluings, and dual oscillating curves $\Gamma_1, \dots, \Gamma_{n-n_\mfc}$, so that these curves, along with a standard basis for $H_1(\partial M)$, $\mfm_1$, $\mfl_1$, $\ldots$, $\mfm_{n_\mfc}$, $\mfl_{n_\mfc}$ form a symplectic basis, in the sense that
\begin{align*}
\omega(h(C_i), h(G_j)) &= 2 C_i \cdot G_j = 2 \delta_{ij} \\
\omega(h(\mfm_i), h(\mfl_j)) &= 2 \mfm_i \cdot \mfl_j = 2 \delta_{ij}
\end{align*}
and $\omega$, evaluated on the combinatorial holonomy of any other pair of curves, is zero.
\end{thm}

This construction also works in the general case, where the frontier may consist of surfaces of any genus. Again let $n_\mfc$ denote the number of boundary components and $n$ the number of ideal tetrahedra; also let $n_E$ denote the number of edges. Then the frontier of $M$ consists of $n_\mfc$ closed oriented surfaces $S_1, \ldots, S_{n_\mfc}$. Let $S_i$ have genus $g_i$ and let $g = \sum_{i=1}^{n_\mfc} g_i$ denote the total genus.

\begin{named}{\refthm{OscAlgorithm}}
There exists a constructive algorithm to determine oscillating curves $C_1, \dots, C_{n_E-n_\mfc}$ associated with edge gluings, and dual oscillating curves $\Gamma_1, \dots, \Gamma_{n_E-n_\mfc}$, so that these curves along with a standard basis for $H_1(\partial M)$, $\mfm_1$, $\mfl_1$, $\ldots$, $\mfm_{g}$, $\mfl_{g}$ form a symplectic basis, in the sense that
\begin{align*}
\omega(h(C_i), h(G_j)) &= 2 C_i \cdot G_j = 2 \delta_{ij} \\
\omega(h(\mfm_i), h(\mfl_j)) &= 2 \mfm_i \cdot \mfl_j = 2 \delta_{ij}
\end{align*}
and $\omega$, evaluated on the combinatorial holonomy of any other pair of curves, is zero.
\end{named}  

The $\mfm_j$ and $\mfl_j$ above consist of $2g_i$ curves on each boundary surface $S_i$, forming a standard symplectic basis for $H_1(S_i)$. This includes the case of genus zero: when $S_i$ is a sphere, none of the $\mfm_j$ or $\mfl_j$ lie on $S_i$. 

When all boundary surfaces are tori we have all $g_i = 1$ so $g = n_\mfc$, and $n_E = n$, so \refthm{OscAlgorithm} reduces to \refthm{OscAlgorithmTori}.

\subsection{Three-dimensional train tracks}
For the proof of the main theorems, we organise the curves in our triangulation into a 3-dimensional variant of \emph{train tracks} on a 3-dimensional triangulation, and show that Neumann--Zagier's symplectic form is equivalent to an intersection form on these 3-dimensional train tracks.
Oscillating curves can be equivalently defined without train tracks --- they can be regarded more geometrically as ``curves on the boundary which may dive through the manifold along ideal edges of a triangulation", and this can be made precise --- but train tracks provide a convenient device for defining and organising them.

Train tracks were introduced by Thurston to study simple closed curves and measured laminations on surfaces~\cite{thurston}. Penner, with Harer, gives a comprehensive treatment of the combinatorics of train tracks on surfaces~\cite{PennerHarer:TrainTracks}, including discussion of a symplectic form on closed curves and measured laminations on surfaces. This symplectic form is often called Thurston's anti-symmetric intersection form on the space of measured laminations. It is known to be closely related to the Weil-Petersson symplectic form on Teichm\"uller space~\cite{BonahonSozen}, and to the Atiyah-Bott-Goldman symplectic form on the character variety~\cite{Zeybek}. Indeed, Luo notes in \cite{Luo:VolNormSfces} that the Neumann--Zagier symplectic form is a 3-dimensional counterpart to Thurston's intersection form. This paper makes the analogy more explicit.

Our 3-dimensional train tracks live on the Heegaard surface associated to a 3-manifold triangulation. That is, associated to a triangulation is a Heegaard splitting with a handlebody on one side corresponding to a neighbourhood of the dual graph of the triangulation, and a compression body on the other side. The train track is on the boundary of the handlebody. 
Notions of branches, switches, carrying curves, and measured laminations exist on the Heegaard surface, and they are identical to their 2-dimensional analogues. The key difference is that we must introduce a new type of switch that we call a ``station''. This appears exactly in neighbourhoods of ideal edges of the triangulation, or equivalently in an annular neighbourhood of a meridian curve for the compression body in the Heegaard splitting. At a station, oriented arcs must stop and change direction. Oscillating curves are exactly those curves that run into an even number of stations, so that an oscillating orientation on the entire curve is well defined.

\subsection{Applications}

Our geometric interpretation of a symplectic dual to the gluing equations can be read immediately off of a triangulation of the 3-manifold, similar to the gluing and cusp equations themselves. They are obtained algorithmically. Gluing and cusp vectors are implemented by SnapPy~\cite{SnapPy}. We expect it should be straightforward to add to this computation of symplectic duals of gluing equations, although we have not implemented this computationally as of the writing of this paper. 

Solutions to the gluing and cusp equations may be encoded in vectors $Z$ that satisfy $\NZ\cdot Z = i\pi C$, where $\NZ$ is an integer matrix encoding gluing and cusp equations, and $C$ is an integer vector with values determined by the triangulation. Using the methods of this paper, particularly \refthm{OscAlgorithm}, we may add rows to $\NZ$ to extend it into a full $2n\times 2n$ matrix $\SY$, which is essentially symplectic: namely symplectic, aside from some factors of $2$. Solutions to gluing and cusp equations then satisfy $\SY\cdot Z = \overline{C}$, where $\overline{C}$ is obtained from $C$ by adding some zeroes.
Since $\SY$ is essentially symplectic, it is very easily invertible.

\begin{named}{\refthm{SolvingNZ}}
Let $\SY$ be the $2n\times 2n$ 
matrix whose rows correspond to gluing equations and their symplectic duals arising from oscillating curves, and curves forming a symplectic basis for $H_1(\partial M)$.
Then a solution to the gluing and cusp equation satisfies
\[ 2Z = i\pi (-J(\SY)^T J)\overline{C} \]
Here $\overline{C}$, $\SY$, $J$ are all integral matrices, so our solutions are integer or half integer multiples of $i\pi$. 
\end{named}

If we remove the factor $i\pi$ from the right hand side of the equation of \refthm{SolvingNZ}, we are considering integer and half-integer solutions to an equation considered by Neumann, who showed that in fact, integer solutions exist \cite{Neumann}. 
\refthm{SolvingNZ} gives an explicit way to compute these solutions, modulo a factor of two. It would be very satisfying to develop an algorithmic way to determine symplectic dual vectors that are all even, so that our matrix $\SY$ yields integer solutions as predicted by \cite{Neumann}. In work of the authors with Howie, \cite{HowieMathewsPurcell} we show such solutions are obtained by adding multiples of explicit vectors, but the geometry of the process is unclear.

\subsection{Organisation}

In \refsec{TrainTracks} we discuss train tracks and our enhancements of them, including stations. We also introduce the notions of oscillating curve and their intersection numbers. 

In \refsec{Tetrahedra} we discuss the Heegaard splitting obtained from an ideal triangulation, and a cell decomposition of the associated handlebody which is crucial for our constructions. Then in \refsec{polyhedra_abcs} we systematically construct train tracks on the Heegaard surface, and consider oscillating curves on these train tracks and their intersection properties.

In \refsec{CombinatorialHolonomy} we define the symplectic vector space associated to the triangulation and the notion of combinatorial holonomy, allowing us to state the Main \refthm{MainThm}. In the subsequent \refsec{omega_on_oscillating} we prove it.

In \refsec{SymplecticBasis} we describe how to construct a symplectic basis of oscillating curves, stating and proving the second Main \refthm{OscAlgorithm}.

Finally, in \refsec{Matrices} we consider the examples of the figure-8 knot and Whitehead link complements, explicitly calculating symplectic bases and matrices. We also consider how our results can be used to solve gluing and cusp equations.

\subsection{Acknowledgements}
This work was partially funded by Australian Research Council grant DP210103136.

\section{Enhanced train tracks}\label{Sec:TrainTracks}

We will be considering paths on a Heegaard surface. Not all paths are permitted, at least without adjusting by isotopy. 
It is convenient to describe the permitted paths by constraining them
to lie on an object which is similar to a train track, with a few enhancements. We now discuss tracks and curves on them, in a more or less conventional sense, before proceeding to describe the required enhancements. See Penner and Harer \cite{PennerHarer:TrainTracks} for more background on conventional train tracks.

\subsection{Conventional train tracks}
A \emph{train track} on a surface $S$ is a graph $\tau$ smoothly embedded in $S$ such that edges at each vertex are all mutually tangent, and there is at least one edge in each of the two possible directed tangent directions. 
Because of the resemblance to railroads, the vertices of $\tau$ are called \emph{switches} and the edges are called \emph{branches}. Thus all switches have valence at least two; in the examples we construct, they will have valence two, three, and four, with one edge on one side and three on the other. Figure~\ref{Fig:Switch} shows a trivalent switch.

\begin{figure}
\begingroup%
  \makeatletter%
  \providecommand\color[2][]{%
    \errmessage{(Inkscape) Color is used for the text in Inkscape, but the package 'color.sty' is not loaded}%
    \renewcommand\color[2][]{}%
  }%
  \providecommand\transparent[1]{%
    \errmessage{(Inkscape) Transparency is used (non-zero) for the text in Inkscape, but the package 'transparent.sty' is not loaded}%
    \renewcommand\transparent[1]{}%
  }%
  \providecommand\rotatebox[2]{#2}%
  \newcommand*\fsize{\dimexpr\f@size pt\relax}%
  \newcommand*\lineheight[1]{\fontsize{\fsize}{#1\fsize}\selectfont}%
  \ifx\svgwidth\undefined%
    \setlength{\unitlength}{144bp}%
    \ifx\svgscale\undefined%
      \relax%
    \else%
      \setlength{\unitlength}{\unitlength * \real{\svgscale}}%
    \fi%
  \else%
    \setlength{\unitlength}{\svgwidth}%
  \fi%
  \global\let\svgwidth\undefined%
  \global\let\svgscale\undefined%
  \makeatother%
  \begin{picture}(1,0.5)%
    \lineheight{1}%
    \setlength\tabcolsep{0pt}%
    \put(0,0){\includegraphics[width=\unitlength,page=1]{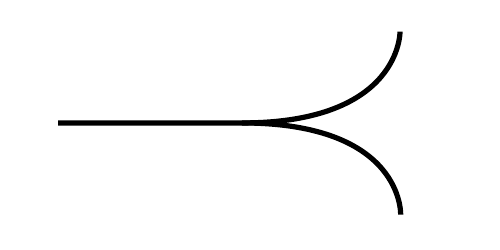}}%
    \put(0.02326245,0.24411309){\color[rgb]{0,0,0}\makebox(0,0)[lt]{\lineheight{1.25}\smash{\begin{tabular}[t]{l}$\gamma_0$\end{tabular}}}}%
    \put(0.820005,0.43849827){\color[rgb]{0,0,0}\makebox(0,0)[lt]{\lineheight{1.25}\smash{\begin{tabular}[t]{l}$\gamma_2$\end{tabular}}}}%
    \put(0.81958951,0.07060927){\color[rgb]{0,0,0}\makebox(0,0)[lt]{\lineheight{1.25}\smash{\begin{tabular}[t]{l}$\gamma_1$\end{tabular}}}}%
  \end{picture}%
\endgroup%

  \caption{The switch of a train track.}
  \label{Fig:Switch}
\end{figure}
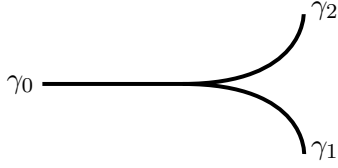

An \emph{orientation} on a train track is an orientation of each of its branches. We allow each branch to be oriented arbitrarily.
Note in the literature, a train track is orientable if there exists an orientation with branches running into a switch on one side, and branches running out on the other. The train tracks we construct will never be orientable in this sense. Instead, we will orient branches in our constructed train tracks below in a way that makes our computations more straightforward. 

Finally, the complementary regions of a conventional train track typically cannot include any nullgons, monogons, bigons, or annuli. We will be constructing specific train tracks that satisfy this condition, and so we will not mention it further.

A train track $\tau$ has a natural thickening, called its fibred neighbourhood and denoted $N(\tau)$, as follows. Each branch $e$ is thickened to $e\times[-1,1]$, and intervals $p\times[-1,1]$ are called fibres, for points $p\in e$. At a trivalent switch $v$, branches $e_1, e_2$ approach from one direction and $e_3$ from another. The fibred neighbourhood identifies $v\times [-1,1]$ in $e_1\times[-1,1]$ to $v\times[-1,0]$ in $e_3\times[-1,1]$, and $v\times[-1,1]$ in $e_2\times[-1,1]$ to $v\times[0,1]$ in $e_3\times[-1,1]$. Similarly if a switch has branches $e_1,\dots, e_n$ on one side, and $e_{n+1}$ on the other, subdivide $v\times [-1,1]$ in $e_{n+1}\times [-1,1]$ into $n$ subintervals, and identify $v\times[-1,1]$ in $e_i\times[-1,1]$ to the $i$th subinterval, for $i=1, \dots, n$. A lamination is \emph{carried} by $\tau$ if it lies within the fibred neighbourhood and is transverse to all the fibres. 
Thus a closed curve $c$ carried by $\tau$ can be regarded as consisting of a cyclic sequence of oriented branches $\gamma_j$ and switches $v_j$ of $\tau$
\begin{equation}\label{Eqn:CyclicSequence}
v_0, \gamma_0, v_1, \gamma_1, v_2, \ldots, v_{n-1}, \gamma_{n-1}, v_n = v_0, \gamma_n = \gamma_0,
\end{equation}
such that each $\gamma_j$ ($j$ taken modulo $n$) is oriented from $v_j$ to $v_{j+1}$, and at each switch $v_j$, the two incident branches $\gamma_{j-1}$ and $\gamma_j$ lie on opposite sides of $v_j$. (Note the orientation on each $\gamma_j$ may agree or disagree with the orientation given to that branch on $\tau$.) Such a $c$ can be realised smoothly. However, it may run over an individual branch many times. We may further extend to curves that are not embedded. In such a case, we may arrange that self-intersections of the curve occur at switches.

A \emph{measured train track} consists of a train track with a non-negative real number $n_\gamma$ assigned to each branch $\gamma$, called a \emph{weight}. The weights must satisfy the following \emph{switch condition}. Let $v$ be a switch with incidence edges $\gamma_1$, $\gamma_2$ on one side, $\gamma_3$ on the other. Then
$n_{\gamma_1}+n_{\gamma_2}=n_{\gamma_3}$. More generally, for an oriented train track, we require:
\begin{equation}
\label{Eqn:3way_compatibility}
\epsilon_{\gamma_1} n_{\gamma_1} + \epsilon_2 n_{\gamma_2} + \epsilon_{\gamma_3} n_{\gamma_3} = 0
\end{equation}
Here, each $\epsilon_{\gamma_j} = 1$ if $\gamma_j$ is oriented away from $v$ and $-1$ if $\gamma_j$ is oriented towards $v$.
Equation \refeqn{3way_compatibility} expresses equality of incoming and outgoing components at $v$.

Define an \emph{abstract curve} to be an integer linear combination of branches of $\tau$, that is, $\sum_\gamma n_\gamma \gamma$ where each $n_\gamma$ is an integer, and the $n_\gamma$ satisfy the switch condition~\eqref{Eqn:3way_compatibility}. 

A closed curve $c$ on $\tau$ in the sense described by the cyclic sequence~\eqref{Eqn:CyclicSequence} above gives rise to an abstract curve on $\tau$ by assigning to each branch $\gamma$ of $\tau$ the \emph{signed} number of times $c$ passes through $\gamma$, denoted $n_\gamma$. Thus $n_\gamma$ given by $n_\gamma = n_\gamma^+ - n_\gamma^-$, where $n_\gamma^+$ (resp.\ $n_\gamma^-$) is the number of times $c$ passes through $\gamma$ in the direction agreeing (resp.\ disagreeing) with the orientation $\gamma$ is given in $\tau$. A multicurve (i.e.\ a collection of such curves $c$) also gives rise to an abstract curve.

\subsection{Signed intersection numbers}
There is a standard notion of signed intersection number of curves carried by a train track, which extends naturally to abstract curves. Each trivalent switch $s$ in an oriented train track has three branches, and we may draw the track around $s$ as in \reffig{Switch} (where the page has the orientation of the surface). One of these three branches $\gamma_0$ is tangent to $s$ in one direction, and the other two branches $\gamma_1, \gamma_2$ are tangent to $s$ in the opposite direction, so that $\gamma_0, \gamma_1, \gamma_2$ are in anticlockwise order around $s$ with respect to the orientation on the surface $S$. See \reffig{Switch}.

We define a \emph{local signed intersection number} at each switch. 
If $\gamma_1$ and $\gamma_2$ are oriented out of $s$, then we have a local signed intersection number defined as
\[ \gamma_1 \cdot\gamma_2 = 1; \quad \gamma_2 \cdot\gamma_1=-1 
\]
with other local intersection numbers at $s$ (including all those involving $\gamma_0$) being zero. If $\gamma_1$ is oriented into $s$ we flip the above signs; and if $\gamma_2$ is oriented into $s$ we flip them again. (Note this is just the right hand rule in all cases.)

This local intersection number thus agrees locally with the signed intersection number of curves as usually defined. 

\begin{definition}\label{Def:IntNumber_UsualTT}
For $\zeta=\sum_\gamma n_\gamma\gamma$, and $\zeta'=\sum_\gamma n'_\gamma\gamma$ abstract curves on $\tau$, define the \emph{abstract intersection number} $\zeta\cdot\zeta'$ by:
\begin{equation}
\label{Eqn:intersection_numbers}
  2 \zeta\cdot \zeta' = \left( \sum_\gamma n_\gamma \gamma \right) \cdot \left( \sum_\gamma n'_\gamma \gamma \right),
\end{equation}
where we extend linearly. 
\end{definition}

\begin{lemma}\label{Lem:Intersection}
  Let $\tau$ be a trivalent train track. Let $\zeta$, $\zeta'$ be curves carried by $\tau$ with associated abstract curves $\sum_\gamma n_\gamma\gamma$ and $\sum_\gamma n'_\gamma\gamma$. Then the algebraic intersection number of $\zeta$ and $\zeta'$ is equal to the abstract intersection number defined above. \qed
\end{lemma}

\begin{proof}
Each intersection between $\zeta$, $\zeta'$ is observed as $\zeta, \zeta'$ first converging at a switch to run along some branches of $\tau$ together, then diverging at another switch.
(This may occur with $\zeta, \zeta'$ converging along those branches in the same or opposite direction.)
Thus the linear sum of local intersection numbers is twice the algebraic intersection number of $\zeta \cdot \zeta'$, and so agrees with $2\zeta\cdot\zeta'$. 
\end{proof}

More generally, for $S$ and oriented surface, we say a \emph{$k$-switch} on $S$ is a switch $v$ with $k+1$ edges meeting $v$, where one edge $\gamma_0$ is tangent in one direction, and the remaining $\gamma_1, \dots, \gamma_k$ are tangent in the opposite direction, with notation such that $\gamma_0, \dots, \gamma_k$ are in anticlockwise order using the orientation on $S$. We allow $1$-switches. The compatibility equation \refeqn{3way_compatibility} then naturally generalises to
\begin{equation}\label{Eqn:kway_compatibility}
\epsilon_{\gamma_0} n_{\gamma_0} + \cdots + \epsilon_{\gamma_k} n_{\gamma_k} = 0, 
\end{equation}
where as above, at a vertex $v$, an incident oriented branch $\gamma$ has $\epsilon_\gamma = 1$ or $-1$ accordingly as $\gamma$ is oriented away from or towards $v$.

\begin{definition}\label{Def:LocalIntNumberKSwitch}
The \emph{local intersection number} at a $k$-switch, with incident branches $\gamma_0$ on one side and $\gamma_1, \ldots, \gamma_k$ on the other, in anticlockwise order, is defined as follows.
\[
\epsilon_{\gamma_i} \epsilon_{\gamma_j} \, \gamma_i \cdot \gamma_j = \left\{ \begin{array}{ll}
1 & 1 \leq i<j \leq n \\
-1 & 1 \leq j < i \leq n \\
0  & \text{otherwise.} \end{array} \right.
\]
\end{definition}

When there are no stations and only 2-switches, this definition reduces to the definition of local intersection number of curves on trivalent train tracks, and is twice the usual intersection number of curves, as in \reflem{Intersection}. A $k$-switch with $k \geq 2$ naturally degenerates into a series of $k-1$ adjacent 2-switches, with the corresponding local intersection numbers.

\subsection{Train tracks with stations}
\label{Sec:train_tracks_stations}
We enhance our train tracks by adding a new type of 4-valent vertex called a \emph{station}, with different compatibility properties. At a station $v$, the tangent space $T_v S$ is split into four quadrants by a \emph{horizontal} and \emph{vertical} direction, and each quadrant contains one edge incident to $v$. After removing the horizontal line from $T_v S$, it is split into two subsets which we call its two \emph{ends}. After removing the vertical line from $T_v S$, it is split into two subsets which we call its two \emph{sides}. The four quadrants are distinguished by their end and side. See \reffig{Station}.

\begin{figure}
\begingroup%
  \makeatletter%
  \providecommand\color[2][]{%
    \errmessage{(Inkscape) Color is used for the text in Inkscape, but the package 'color.sty' is not loaded}%
    \renewcommand\color[2][]{}%
  }%
  \providecommand\transparent[1]{%
    \errmessage{(Inkscape) Transparency is used (non-zero) for the text in Inkscape, but the package 'transparent.sty' is not loaded}%
    \renewcommand\transparent[1]{}%
  }%
  \providecommand\rotatebox[2]{#2}%
  \newcommand*\fsize{\dimexpr\f@size pt\relax}%
  \newcommand*\lineheight[1]{\fontsize{\fsize}{#1\fsize}\selectfont}%
  \ifx\svgwidth\undefined%
    \setlength{\unitlength}{144bp}%
    \ifx\svgscale\undefined%
      \relax%
    \else%
      \setlength{\unitlength}{\unitlength * \real{\svgscale}}%
    \fi%
  \else%
    \setlength{\unitlength}{\svgwidth}%
  \fi%
  \global\let\svgwidth\undefined%
  \global\let\svgscale\undefined%
  \makeatother%
  \begin{picture}(1,0.64999998)%
    \lineheight{1}%
    \setlength\tabcolsep{0pt}%
    \put(0,0){\includegraphics[width=\unitlength,page=1]{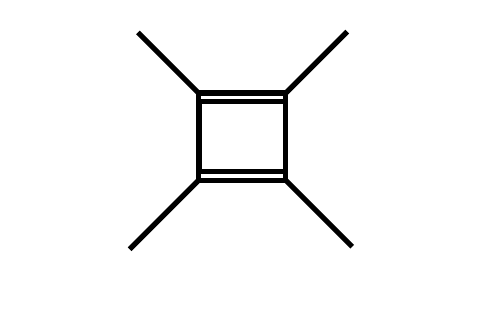}}%
    \put(0.0575192,0.59350546){\color[rgb]{0,0,0}\makebox(0,0)[lt]{\lineheight{1.25}\smash{\begin{tabular}[t]{l}$\hat{\overline{\gamma}}=\hat{\delta}$\end{tabular}}}}%
    \put(0.71448949,0.159827){\color[rgb]{0,0,0}\makebox(0,0)[lt]{\lineheight{1.25}\smash{\begin{tabular}[t]{l}$\gamma$\end{tabular}}}}%
    \put(0.19089047,0.159827){\color[rgb]{0,0,0}\makebox(0,0)[lt]{\lineheight{1.25}\smash{\begin{tabular}[t]{l}$\hat{\gamma}$\end{tabular}}}}%
    \put(0.70285827,0.59350547){\color[rgb]{0,0,0}\makebox(0,0)[lt]{\lineheight{1.25}\smash{\begin{tabular}[t]{l}$\overline{\gamma}=\delta$\end{tabular}}}}%
    \put(0.75896511,0.36586582){\color[rgb]{0,0,0}\makebox(0,0)[lt]{\lineheight{1.25}\smash{\begin{tabular}[t]{l}distinct ends\end{tabular}}}}%
    \put(0,0){\includegraphics[width=\unitlength,page=2]{Station.pdf}}%
    \put(0.31637184,0.02376174){\color[rgb]{0,0,0}\makebox(0,0)[lt]{\lineheight{1.25}\smash{\begin{tabular}[t]{l}distinct sides\end{tabular}}}}%
  \end{picture}%
\endgroup%

  \caption{Shows a station: $\gamma$, $\widehat{\gamma}$, $\overline{\gamma}$, $\widehat{\overline{\gamma}}$ lie in the four quadrants of $T_vS$. Sides are left and right, ends are top and bottom.}
  \label{Fig:Station}
\end{figure}

\begin{definition}
A \emph{station} is a 4-valent vertex $v$ in a graph smoothly embedded in a surface $S$, with distinguished horizontal and vertical directions as above, so that precisely one edge is incident to $v$ from each quadrant.
\end{definition}

\begin{definition}
An \emph{enhanced train track} on a surface $S$ is a graph $\tau$ smoothly embedded in $S$ such that each vertex is either a $k$-switch for some $k\geq 1$, or a station.
An \emph{orientation} on an enhanced train track is an orientation of each of its branches; note we allow arbitrary orientations on branches. 
\end{definition}
The enhanced train tracks we construct in this paper will lie on surfaces with corners.  Thus the embedding of the underlying graph may not be entirely smooth. However, smoothness only fails at points in the interior of branches, and this does not affect any of our arguments.

Given a branch $\gamma$ incident to a station $v$, we denote the other incident edges as follows. The edge approaching $v$ on the same side at the opposite end is denoted $\overline{\gamma}$;  the edge approaching $v$ at the same end but opposite side is denoted $\widehat{\gamma}$; the edge approaching $v$ at the opposite side and end is denoted $\widehat{\overline{\gamma}}$. Thus an overline swaps ends and a hat swaps sides.

A station has an unconventional compatibility condition. Let $v$ be a station, with adjacent edges $\gamma, \overline{\gamma}, \widehat{\gamma}, \widehat{\overline{\gamma}}$. For convenience, write $\delta = \overline{\gamma}$, so that the adjacent edges are $\gamma, \widehat{\gamma}, \delta, \widehat{\delta}$. If all the edges are oriented away from $v$, then we require the compatibility condition
\[
n_\gamma + n_{\widehat{\gamma}} = n_\delta + n_{\widehat{\delta}}.
\]
Note in particular that a curve on $\tau$, given by a sequence of branches passing through vertices as in the previous section, sometimes does but in general \emph{does not} satisfy this condition. In particular, a curve $c$ passing through $\gamma$ does satisfy the condition when it stays at one end of $v$, but if $c$ passes through $v$ from one end to the other, then the condition is violated.

However, if a curve $c$ \emph{changes orientation} as it passes through $v$ from one end to the other, then it \emph{does} satisfy the condition. Thus we will require curves to \emph{oscillate orientation} as they pass through stations from one end to the other; hence we call them \emph{oscillating curves} and make the following formal definitions.

As above, at a vertex $v$, an incident oriented branch $\gamma$ has $\epsilon_\gamma = 1$ or $-1$ accordingly as $\gamma$ is oriented away from or towards $v$.
\begin{definition}
\label{Def:abstract_oscillating_curve}
An \emph{abstract oscillating curve} on an oriented enhanced train track $\tau$ is a labelling of each branch $\gamma$ of $\tau$ by an integer $n_\gamma$, such that the following compatibility conditions are satisfied at each vertex $v$.
\begin{enumerate}
\item If $v$ is a $k$-switch with incident edges $\gamma_0, \gamma_1, \ldots, \gamma_k$, then
\[
\epsilon_{\gamma_0} n_{\gamma_0} + \cdots + \epsilon_{\gamma_k} n_{\gamma_k} = 0.
\]
\item If $v$ is a station with incident edges $\gamma, \widehat{\gamma}$ at one end and $\delta = \overline{\gamma}, \widehat{\delta}$ on the other, then
\[
\epsilon_\gamma n_\gamma + \epsilon_{\widehat{\gamma}} n_{\widehat{\gamma}} = \epsilon_\delta n_\delta + \epsilon_{\widehat{\delta}} n_{\widehat{\delta}}.
\]
\end{enumerate}
We denote an abstract oscillating curve by $\sum_\gamma n_\gamma \gamma$.
\end{definition}

A less abstract form of oscillating curve is defined by a cyclic sequence of oriented branches and vertices; however some of the orientations on the branches may be ``forwards" of ``backwards".
\begin{definition}\label{Def:OscillatingCurve}
An \emph{oscillating curve} of length $n$ on an enhanced train track is a cyclic sequence of oriented branches $\gamma_j$ and vertices $v_j$, with $j$ taken modulo $n$,
\[
v_0, \gamma_0, v_1, \gamma_1, v_2, \ldots, v_{n-1}, \gamma_{n-1}, v_n = v_0, \gamma_n = \gamma_0,
\]
such that the following conditions hold.
\begin{enumerate}
\item
Each branch $\gamma_j$ has one end at $v_j$ and one end at $v_{j+1}$. If $\gamma_j$ is oriented from $v_j$ to $v_{j+1}$ (resp. from $v_{j+1}$ to $v_j$) then $\gamma_j$ is \emph{forwards} (resp. \emph{backwards}).
\item
If $v_j$ is a $k$-switch, then $\gamma_{j-1}$ and $\gamma_j$ approach $v_j$ from opposite sides, and are both forwards or both backwards.
\item
If $v_j$ is a station, then:
\begin{enumerate}
\item if $\gamma_{j-1}$ and $\gamma_j$ are both at the same end of $v_j$, then they are both forwards or both backwards;
\item if $\gamma_{j-1}$ and $\gamma_j$ are at opposite ends of $v_j$, then one is forwards and one is backwards.
\end{enumerate}
\end{enumerate}
\end{definition}

(This definition assumes that the train track contains no loops and all incident edges at a station are distinct; this will be the case for all train tracks we construct.)

Thus, an oscillating curve is a closed curve along an enhanced train track $\tau$ composed of a series of arcs, each arc given orientation information; it must pass through an even number of stations from one end to the other, reversing orientation each time. It proceeds smoothly around $\tau$, except possibly at stations (and at corners of the surface, which in our construction will only lie in the interior of branches). It need not be simple.

We can alternatively regard an oscillating curve as a collection of oriented arcs along the train track, beginning and ending at stations. The arcs are ``created" or ``annihilated" in pairs at the stations. At creation, a pair of arcs departs from a station towards opposite ends; at annihilation, a pair of arcs arrives at a station from opposite ends.

An oscillating curve defines an abstract oscillating curve by counting the signed number of times it covers each branch. (As with ordinary train tracks, a finite collection of oscillating curves also defines an abstract oscillating curve by summation.) The conditions on orientations of oscillating curves at switches and stations ensure that the compatibility conditions of an abstract oscillating curve are satisfied. When there are no stations, an oscillating curve simply reduces to a closed curve which may be smoothly embedded on a train track, oriented forwards or backwards.

We now define local intersection numbers for enhanced train track branches and oscillating curves, generalising from the ordinary case. 
At stations $v$, local intersection numbers behave unconventionally. This behaviour is forced by the symplectic form; we will also give some topological motivation in \refsec{train_tracks_on_H}.
For the purposes of local intersection numbers, a station behaves like two 2-switches, one on each side, placed in an arrangement like \reffig{station_switches}.

\begin{figure}
\centering
\def\svgscale{0.5}
\begingroup%
  \makeatletter%
  \providecommand\color[2][]{%
    \errmessage{(Inkscape) Color is used for the text in Inkscape, but the package 'color.sty' is not loaded}%
    \renewcommand\color[2][]{}%
  }%
  \providecommand\transparent[1]{%
    \errmessage{(Inkscape) Transparency is used (non-zero) for the text in Inkscape, but the package 'transparent.sty' is not loaded}%
    \renewcommand\transparent[1]{}%
  }%
  \providecommand\rotatebox[2]{#2}%
  \newcommand*\fsize{\dimexpr\f@size pt\relax}%
  \newcommand*\lineheight[1]{\fontsize{\fsize}{#1\fsize}\selectfont}%
  \ifx\svgwidth\undefined%
    \setlength{\unitlength}{281.01258866bp}%
    \ifx\svgscale\undefined%
      \relax%
    \else%
      \setlength{\unitlength}{\unitlength * \real{\svgscale}}%
    \fi%
  \else%
    \setlength{\unitlength}{\svgwidth}%
  \fi%
  \global\let\svgwidth\undefined%
  \global\let\svgscale\undefined%
  \makeatother%
  \begin{picture}(1,0.59379866)%
    \lineheight{1}%
    \setlength\tabcolsep{0pt}%
    \put(0,0){\includegraphics[width=\unitlength,page=1]{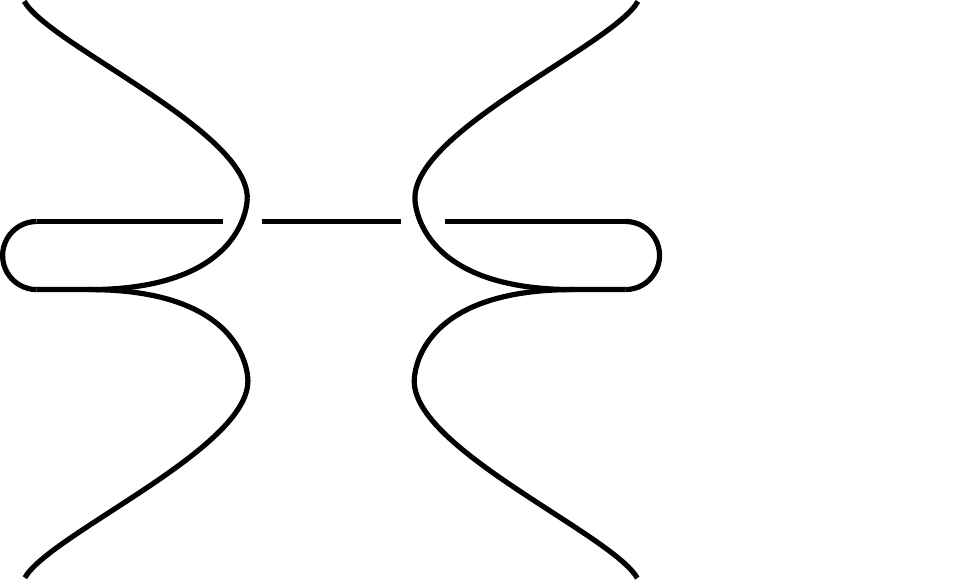}}%
    \put(0.57476049,0.08512666){\color[rgb]{0,0,0}\makebox(0,0)[lt]{\lineheight{1.25}\smash{\begin{tabular}[t]{l}$\gamma$\end{tabular}}}}%
    \put(0.55602836,0.45286169){\color[rgb]{0,0,0}\makebox(0,0)[lt]{\lineheight{1.25}\smash{\begin{tabular}[t]{l}$\overline{\gamma} = \delta$\end{tabular}}}}%
    \put(0.02727189,0.08938399){\color[rgb]{0,0,0}\makebox(0,0)[lt]{\lineheight{1.25}\smash{\begin{tabular}[t]{l}$\widehat{\gamma}$\end{tabular}}}}%
    \put(0.0207462,0.45286169){\color[rgb]{0,0,0}\makebox(0,0)[lt]{\lineheight{1.25}\smash{\begin{tabular}[t]{l}$\widehat{\delta}$\end{tabular}}}}%
  \end{picture}%
\endgroup%

  \caption{The intersection form on a station behaves unconventionally like two 2-switches.}
  \label{Fig:station_switches}
\end{figure}
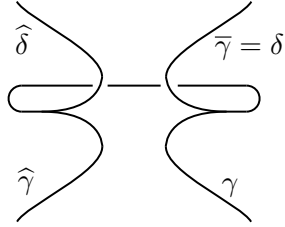

Denote the four incident branches by $\gamma$ and $\delta = \overline{\gamma}$ on one side, and $\widehat{\gamma}$ and $\widehat{\delta}$ on the other, and moreover so that in anticlockwise order they are $\gamma, \delta, \widehat{\delta}, \widehat{\gamma}$. See \reffig{Station}. If these four branches are all oriented away from $v$, define
\begin{equation}
\label{Eqn:station_intersection}
\gamma \cdot \delta = -1, \quad
\delta \cdot \gamma = 1, \quad
\widehat{\delta} \cdot \widehat{\gamma} = -1, \quad
\widehat{\gamma} \cdot \widehat{\delta} = 1.
\end{equation}

The same equations hold if all branches are oriented into $v$.
Adjusting the above by signs $\epsilon_\gamma = 1$ or $-1$ accordingly as they point away from or towards $v$ we obtain the following.

\begin{definition}
\label{Def:intersection_number_branches}
The \emph{local intersection number} of two branches meeting at a station in an oriented enhanced train track is defined as follows. If incident branches in anticlockwise order are $\gamma, \delta, \widehat{\delta}, \widehat{\gamma}$ as above, then
\begin{gather*}
\epsilon_\gamma \epsilon_\delta \; \gamma \cdot \delta = -1, \quad
\epsilon_\gamma \epsilon_\delta \; \delta \cdot \gamma = 1, \\
\epsilon_{\widehat{\gamma}} \epsilon_{\widehat{\delta}} \; \widehat{\delta} \cdot \widehat{\gamma} = -1, \quad
\epsilon_{\widehat{\gamma}} \epsilon_{\widehat{\delta}} \; \widehat{\gamma} \cdot \widehat{\delta} = 1,
\end{gather*}
and all other local intersection numbers between them are zero.

Local intersection numbers of an enhanced train track are then given by the above numbers at stations, and \refdef{LocalIntNumberKSwitch} on $k$-switches. 
\end{definition}

We can regard an oscillating curve as a curve on an ordinary train track, where $k$-switches have been degenerated into $2$-switches, and stations have been replaced with two 2-switches, one on each side. Then the same argument shows that the intersection number of the oscillating curves is equal to twice the usual intersection number of curves, although we must be careful to take into account the reversal of orientation in oscillating curves when proceeding through a station from one end to another (which can lead to unconventional results).
This leads to the following definition, generalising \refdef{IntNumber_UsualTT}.

\begin{definition}
\label{Def:intersection_number}
Let $\zeta = \sum_\gamma n_\gamma \gamma$ and $\zeta' = \sum_\gamma n'_\gamma \gamma'$ be abstract oscillating curves. Their \emph{intersection number} is given by
\[
  2 \zeta\cdot \zeta' = \left( \sum_\gamma n_\gamma \gamma \right) \cdot \left( \sum_\gamma n'_\gamma \gamma \right),
\]
where the right hand side is obtained by linear extension of \refdef{intersection_number_branches}.
\end{definition}

\section{Tetrahedra, polyhedra, and {H}eegaard splittings}\label{Sec:Tetrahedra}

Let $M$ be the interior of a compact manifold $\overline{M}$ with nonempty boundary consisting of $n_\mfc \geq 1$ components.

Let $\calT$ be an \emph{ideal} triangulation of $M$.
That is,
a decomposition of $M$ into tetrahedra with vertices removed; truncating the vertices yields a decomposition of $\overline{M}$ into truncated tetrahedra, such that the truncated vertices form a 2-dimensional triangulation of $\bdy\overline{M}$.
This applies regardless of whether the $\bdy \overline{M}$ consists of tori or higher genus surfaces.
We call the edges \emph{ideal} edges and the vertices before truncation \emph{ideal vertices}.

Such a triangulation is known to exist by work of Moise and Bing~\cite{Bing, Moise}. Concretely, it can often be obtained by computer~\cite{Regina, SnapPy}. As is common, we regard $\overline{M}$ as a deformation retract of $M$, obtained by removing an open neighbourhood of each end. We regard these end neighbourhoods as $\partial \overline{M} \times (0, \infty)$.

In this section, from the ideal triangulation $\calT$ we explicitly construct a Heegaard splitting of $M$ and a cell decomposition of the Heegaard handlebody.

\subsection{Heegaard decomposition}

Consider disjoint open tubular neighbourhoods of the ideal edges of $\calT$. These tubular neighbourhoods, together with the end neighbourhoods $\partial \overline{M} \times (0, \infty)$, form an open compression body. The complement of this compression body is a compact handlebody, given by $\overline{M}$ with tubular neighbourhoods of edges of $\calT$ removed. The handlebody may be regarded as having a 3-dimensional 0-cell for each tetrahedron of $\calT$, and a 1-handle for each pair of glued faces of $\calT$. The handlebody and compression body form a Heegaard splitting of $M$, and their common boundary is a Heegaard surface, which we denote by $H$. 

We regard $H$ as built from two subsurfaces: a surface $\bdy\overline{M}\cap H$ in general consisting of $n_\mfc$ components, of various genera, with punctures corresponding to intersections of ideal edges of 
$\calT$ with $\bdy\overline{M}$; and a collection of annuli corresponding to the boundary components of tubes about the ideal edges of $\calT$. 

\subsection{Cell decomposition}

We give a cell decomposition of $H$ and the handlebody bounded by $H$.

First, as $\overline{M}$ is obtained by truncating $M$ (removing cusp neighbourhoods), we may truncate each ideal tetrahedron $\tet$ of $\calT$ and regard $\overline{M}$ as obtained by gluing truncated tetrahedra. Each truncated tetrahedron $\overline{\tet}$ is obtained from an ideal tetrahedron $\tet$ by removing a neighbourhood of each ideal vertex.

\begin{figure}
  \includegraphics[scale=1.4]{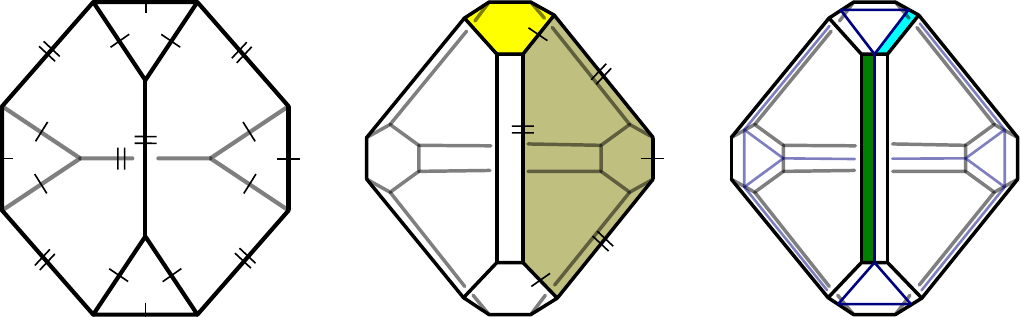}
  \caption{Left to right: A tetrahedron $\overline{\tet}$ with truncated vertices; short and long edges are shown with one and two ticks respectively. A polyhedron $\p$ obtained from $\overline{\tet}$ by truncating long edges; a short hexagon is shaded yellow and a long hexagon is shaded olive, with alternating long and short edges indicated. The cell decomposition on $\p$; a short rectangle is shaded aqua and long rectangle is shaded green.}
  \label{Fig:CellDecomposition}
\end{figure}

A truncated tetrahedron $\overline{\tet}$ is shown in \reffig{CellDecomposition} on the left. It has 12 vertices, 3 from each ideal vertex of $\tet$. It has 8 faces, consisting of 4 triangular faces, one from each ideal vertex of $\tet$, and $4$ hexagonal faces, one from each face of $\tet$. It has $16$ edges, consisting of $12$ sides of triangular faces, which we call \emph{short edges} (denoted with single ticks in \reffig{CellDecomposition} left); and $4$ edges which are truncated edges of $\tet$, which we call \emph{long edges} (denoted with double ticks in \reffig{CellDecomposition} left). Observe that short and long edges alternate around the boundary of each hexagonal face.

The triangular faces of the truncated tetrahedra form a triangulation of $\partial \overline{M}$, 
which we call the \emph{boundary triangulation}.
Each vertex of the boundary triangulation lies at an end of an ideal edge of an ideal tetrahedron $\tet$, or equivalently at an end of a long edge of a truncated tetrahedron $\overline{\tet}$. Each edge of the 
boundary triangulation is a short edge of a truncated tetrahedron $\overline{\tet}$, and lies at an end of a face of an ideal tetrahedron $\tet$. 

The handlebody of the Heegaard splitting is obtained from $\overline{M}$ by removing a tubular neighbourhood of each edge of $\calT$.  Further truncate our tetrahedra by removing a neighbourhood of each long edge. This turns the truncated tetrahedron $\overline{\tet}$ into a polyhedron $\p$, as in \reffig{CellDecomposition} middle. In the sequel, we use the notation $\overline{\tet}$ to refer to a \emph{truncated tetrahedron} and $\p$ to refer to a \emph{polyhedron} for ease of notation. Each triangular face of $\overline{\tet}$ has a neighbourhood of each vertex removed to become a hexagon in $\p$, which we call a \emph{short} hexagon; one is shaded yellow in \reffig{CellDecomposition} middle. Each hexagonal face of $\overline{\tet}$ becomes a smaller hexagonal face in $\p$, which we call a \emph{long hexagon} or just \emph{hexagon}. Each long hexagon can still be regarded as having alternating short and long edges around its boundary; one is shaded olive in \reffig{CellDecomposition} middle.

To obtain the cell decomposition of $\p$, each long edge of $\overline{\tet}$ is replaced by two rectangular faces along the length of the long edge, which we call \emph{long rectangles}, as shown in \reffig{CellDecomposition} right where one is shaded green.
For each short hexagon, we take a cell decomposition consisting of a triangle (a smaller version of triangular face from $\overline{\tet}$) and three rectangles, which we call \emph{short rectangles}; one is shaded aqua in \reffig{CellDecomposition} right. The remaining hexagonal faces of $\overline{\tet}$, which we call \emph{long hexagons} (one is shaded olive in \reffig{CellDecomposition} middle), remain hexagonal in $\p$
but become smaller, and still have long and short edges. We call such a face a hexagonal face of $\p$. Effectively, each long edge of $\overline{\tet}$ is replaced with two long rectangles, and each short edge of $\overline{\tet}$ is replaced with one short rectangle.

We may then regard each hexagonal face of $\p$ as having a 3-dimensional regular neighbourhood in $\p$ which intersects the boundary of $\p$ in six adjacent rectangles, alternating long and short. These regular neighbourhoods meet along edges parallel to the removed long edges of $\overline{\tet}$. Each rectangle (long or short) is then associated to a unique hexagonal face and forms part of its regular neighbourhood.

The effect on the boundary triangulation of replacing tetrahedra $\overline{\tet}$ with polyhedra $\p$ is to remove a neighbourhood of each vertex, and replace each triangle with a smaller triangle and three short rectangles. Effectively, each triangle has each of its sides replaced by a short rectangle, and these three short rectangles together bound the smaller triangle. A hexagonal face $\hexagon$ of $\p$ intersects the cusp triangulation in three edges, corresponding to sides of short rectangles.
See \reffig{BoundaryCellDecomposition}. 

The polyhedra $\p$ give a cell decomposition of the Heegaard handlebody of $\overline{M}$, and hence also a cell decomposition of the Heegaard surface $H$. The cell decomposition of $H \cap \bdy \overline{M}$ consists of triangles and short rectangles: one triangle and three short rectangles for each truncated ideal vertex of $\calT$. There is one short rectangle for each ideal vertex of each face of each tetrahedron of $\calT$. The cell decomposition of $\bdy H - \bdy \overline{M}$ consists of long rectangles, two for each edge of each tetrahedron of $\calT$. \reffig{BoundaryCellDecomposition} shows a cusp triangulation arising in a triangulation of the Whitehead link complement (discussed in \cite{HowieMathewsPurcell} and in section \refsec{Matrices} below), and the corresponding cell decomposition of a component of $H \cap \bdy \overline{M}$.

\begin{figure}
  \includegraphics[scale=1]{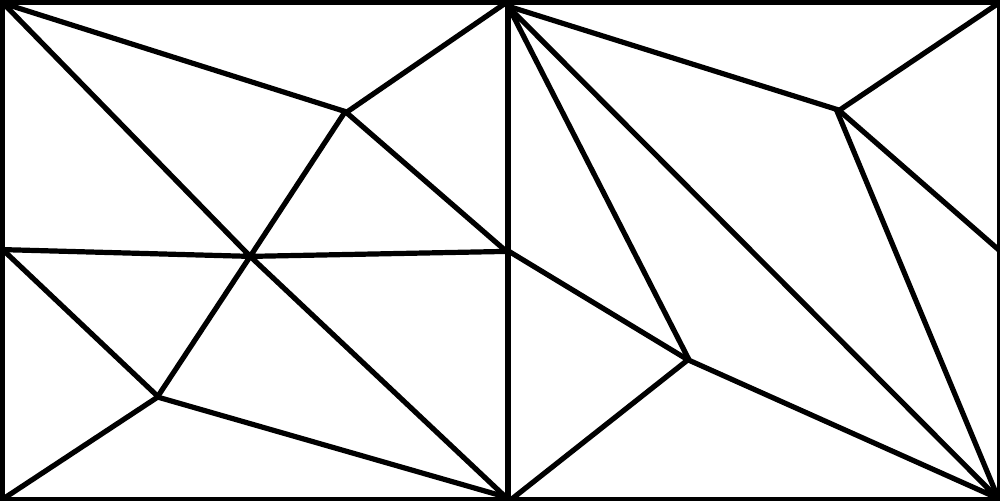}
	
	\bigskip
	
	\includegraphics[scale=1]{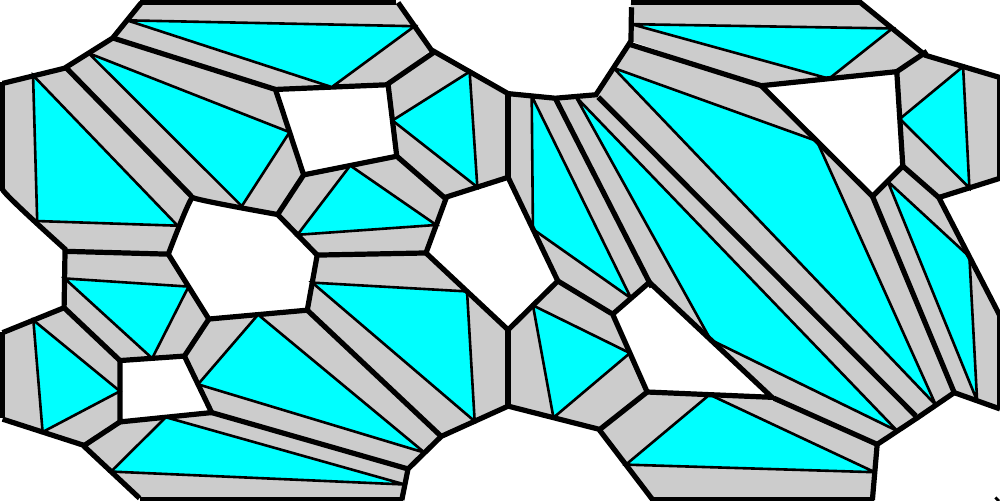}
  \caption{Top to bottom: Triangulation of a torus component of the boundary of the Whitehead link complement; opposite sides of the rectangle are identified. Cell decomposition of the corresponding component of $H \cap \bdy \overline{M}$; triangles in blue and short rectangles in grey.}
  \label{Fig:BoundaryCellDecomposition}
\end{figure}

\section{Polyhedra and their $abc$s}
\label{Sec:polyhedra_abcs}

Following a convention essentially going back at least to Neumann-Zagier \cite{NeumannZagier}, we label each edge of an ideal tetrahedron $\tet$ as an \emph{$a$-edge}, \emph{$b$-edge} or \emph{$c$-edge}, such that opposite edges have the same label, and around each vertex, as viewed from outside $\tet$, the three incident edges are $a$-, $b$- and $c$-edges in anticlockwise order. In the sequel, when referring to ``anticlockwise" or ``clockwise" order, we always mean as viewed from outside the relevant tetrahedron or polyhedron or manifold. 

The edge labelling of $\tet$ implies that around each face of $\tet$, the edges are labelled $a,b,c$ in clockwise order. See \reffig{abcs} left.

We can then label other parts of our truncated tetrahedron $\overline{\tet}$ (\reffig{abcs} right) and polyhedron $\p$ (\reffig{abcs2}) with various combinations of the letters $a,b,c$ as follows.

\begin{figure}
\centering
\def\svgscale{2}
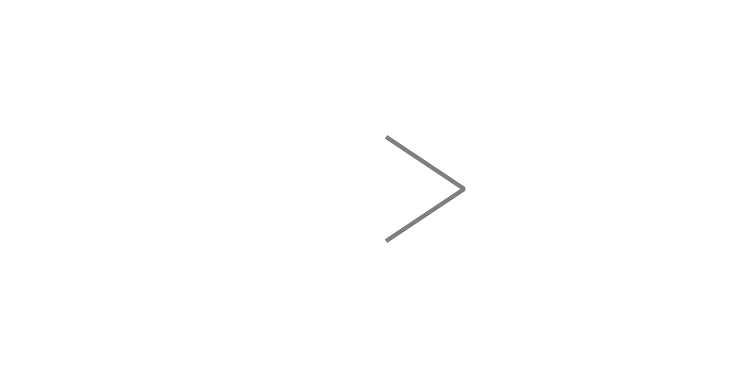
  \caption{Left to right: Tetrahedron $\tet$ with edge labellings. Truncated tetrahedron $\overline{\tet}$ with edge labellings.}
  \label{Fig:abcs}
\end{figure}

\begin{figure}
\centering
\def\svgscale{3.2}
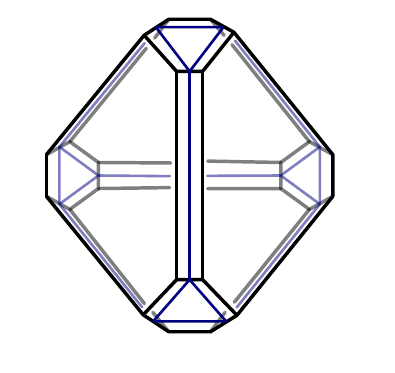
  \caption{Polyhedron $\p$ with labellings of rectangles, edges, and vertices of triangular faces.}
  \label{Fig:abcs2}
\end{figure}

Each long edge of $\overline{\tet}$ and long rectangle of $\p$ runs along an edge of $\tet$, and hence is labelled $a$, $b$ or $c$. Each triangular face of $\overline{\tet}$ (i.e. each triangle of the cusp triangulation) has vertices which are endpoints of long edges, and hence are labelled by $a,b,c$ in anticlockwise order around a triangular face (as always, as seen from outside the truncated tetrahedron). An edge of a triangular face of $\overline{\tet}$ has endpoints labelled $a$ and $b$, or $b$ and $c$, or $c$ and $a$. Accordingly, the edges of each triangular face of $\overline{\tet}$ are labelled $ab,bc,ca$, in anticlockwise order. Each triangular face of $\p$, being a shrunken triangular face of $\overline{\tet}$ then also has vertices labelled $a,b,c$ and edges labelled $ab,bc,ca$, in anticlockwise order. Each short rectangle of $\p$ runs along an edge of a triangular face, and hence obtains a labelling $ab$, $bc$ or $ca$. Similarly, each short edge of a hexagonal face of $\p$ is labelled $ab$, $bc$ or $ca$.

Each hexagonal face $\hexagon$ of $\p$ has three long edges alternating with three short edges: the three long edges are labelled $a,b,c$ in clockwise order as viewed from outside, and the three short edges are labelled $ab,bc,ca$ in clockwise order. The six edges are thus labelled $a,ab,b,bc,c,ca$ in clockwise order.

A notion of cyclic ordering will be useful in the sequel.
\begin{definition}[Cyclic ordering] \
\label{Def:cyclic_alphabetical_order}
\begin{enumerate}
\item
If $(k,l,m)$ is a cyclic permutation of $(a,b,c)$ or $(ab,bc,ca)$, then we say $k,l,m$ are in \emph{cyclic alphabetical order}.
\item
If $(k,l,m)$ is a cyclic permutation of $(c,b,a)$ or $(ab,ca,bc)$, then we say $k,l,m$ are in \emph{cyclic anti-alphabetical order}.
\item
If $(k,l)$ can be completed to $(k,l,m)$ in cyclic alphabetical order, then we say $k,l$ are in \emph{cyclic alphabetical order}.
\item
If $(k,l)$ can be completed to $(k,l,m)$ in cyclic anti-alphabetical order, then we say $k,l$ are in \emph{cyclic anti-alphabetical order}.
\end{enumerate}
\end{definition}

Thus, for instance, $k,l$ are in cyclic alphabetical order precisely when
\[ (k,l) = (a,b), (b,c), (c,a), (ab,bc), (bc,ca) \text{ or } (ca,ab). \]
And if $k,l,m$ are in cyclic alphabetical order then $k,l$ and $l,m$ and $m,k$ are in cyclic alphabetical order.

The following notation allows us to simplify some expressions.
\begin{definition}
\label{Def:epsilon}
\[
\epsilon_{k,l} =
\begin{cases}
1 & \text{if $k,l$ are in cyclic alphabetical order,} \\
-1 & \text{if $k,l$ are in cyclic anti-alphabetical order,} \\
0 & \text{otherwise.}
\end{cases}
\] 
\[
\epsilon_{k,l,m} =
\begin{cases}
1 & \text{if $k,l,m$ are in cyclic alphabetical order,} \\
-1 & \text{if $k,l,m$ are in cyclic anti-alphabetical order,} \\
0 & \text{otherwise.}
\end{cases}
\] 

\end{definition}

\subsection{Train tracks on a triangulation}
\label{Sec:train_tracks_on_H}

We now build an oriented enhanced train track on the boundary of each polyhedron $\p$ and on the Heegaard surface $H$ associated with a triangulation. 

As we have seen, each truncated vertex of an ideal tetrahedron $\tet$ of $\calT$ corresponds to four faces of $\p$: a triangle, and three short rectangles surrounding it. In each such set of four faces, we place an oriented enhanced train track as in \reffig{TrainTrack1}.  That is, inside a triangle $\Delta$, we place three 2-switches, one just inside the midpoint of each edge of the triangle. The switches are joined by branches running around the vertices of the original cusp triangulation, oriented anticlockwise around them. At each switch, place a further branch that runs transversely across the nearby edge of the triangle into the adjacent short rectangle, oriented into the short rectangle. 

The branches in $\Delta$ running around the vertices of the cusp triangulation labelled $a,b,c$ are denoted $\gamma_{\Delta(a)}$, $\gamma_{\Delta(b)}$, $\gamma_{\Delta(c)}$ respectively. The branches running out across the edges of the triangular face labelled $ab,bc,ca$ are denoted $\gamma_{\Delta(ab)}$, $\gamma_{\Delta(bc)}$, $\gamma_{\Delta(ca)}$ respectively. These branches run out of $\Delta$ into short rectangles, and each short rectangle is adjacent to a unique hexagonal face of $\p$. If $\Delta(ab)$ runs into a rectangular face adjacent to hexagon $\hexagon$, we also write $\gamma_{\hexagon(ab)!}$ for $\gamma_{\Delta(ab)}$. See \reffig{TrainTrack1}. Thus the branches labelled $\gamma_{\Delta(k)}$ over all triangles $\Delta$ and $k \in \{ab,bc,ca\}$ are precisely the branches labelled $\gamma_{\hexagon(k)!}$ over all hexagons $\hexagon$ and $k \in \{ab,bc,ca\}$. Their intersection numbers, as per \refdef{intersection_number_branches}, are then given by the following.

\begin{lemma}
\label{Lem:intersection_numbers_triangles}
In triangle $\Delta$, intersection numbers are given by
\begin{gather*}
\gamma_{\Delta(a)} \cdot \gamma_{\Delta(b)} 
= \gamma_{\Delta(b)} \cdot \gamma_{\Delta(c)} 
= \gamma_{\Delta(c)} \cdot \gamma_{\Delta(a)} = 1, \\
\gamma_{\Delta(b)} \cdot \gamma_{\Delta(a)}
= \gamma_{\Delta(c)} \cdot \gamma_{\Delta(b)}
= \gamma_{\Delta(a)} \cdot \gamma_{\Delta(c)} = -1,
\end{gather*}
and all other intersection numbers are $0$.

Thus by \refdef{epsilon}, the intersection numbers can be expressed as:
  \[
  \pushQED{\qed}
  \gamma_{\Delta(k)} \cdot \gamma_{\Delta(l)} = \epsilon_{k,l}.  \qedhere
  \popQED
\]
\end{lemma}

\begin{figure}
	\centering
	\def\svgscale{1.5}
  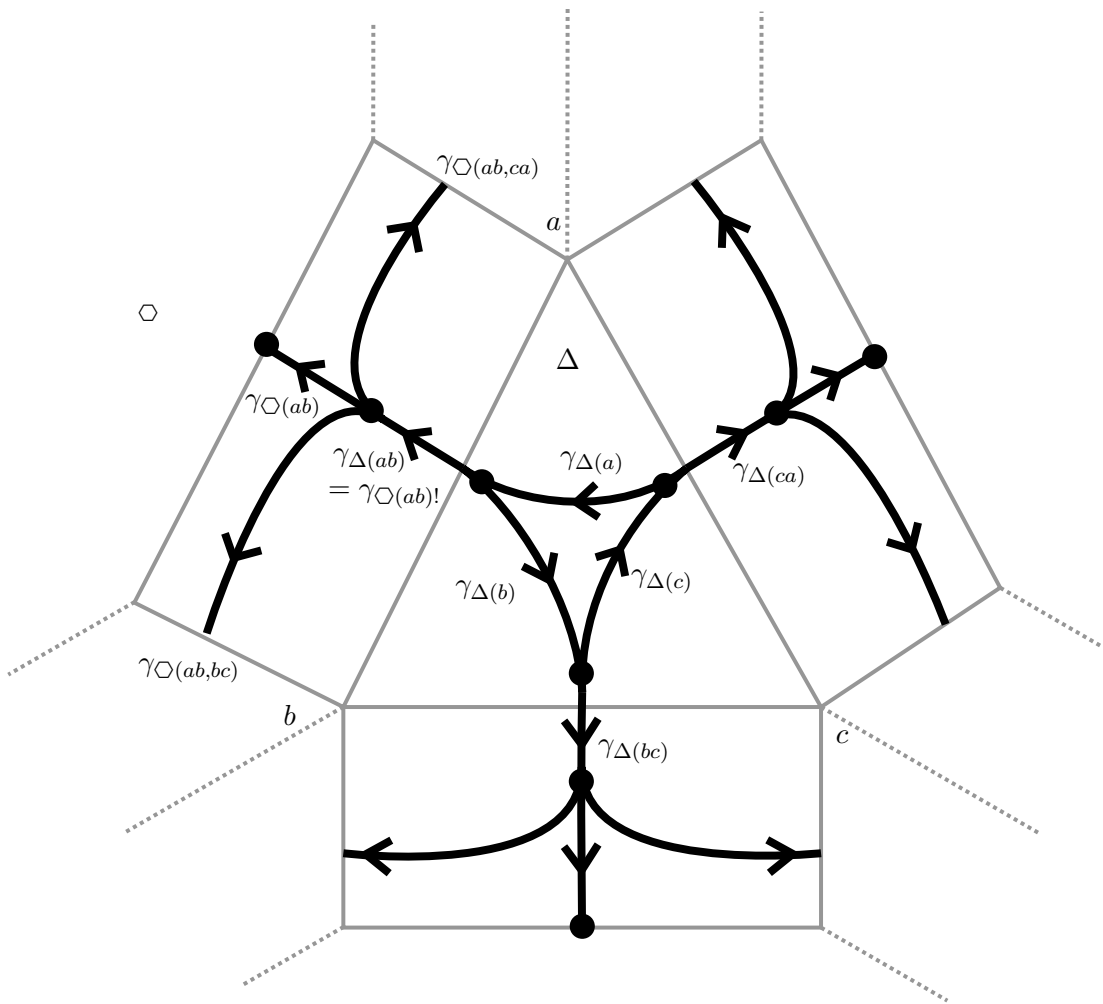 \\
  \caption{Train tracks on a triangular face $\Delta$ and adjacent short rectangles of the Heegaard surface $H$. Switches are denoted by large dots. The hexagonal face adjacent to the short rectangle labelled $ab$ is denoted $\hexagon$. (Note $\hexagon$ is not part of $H$.) Sides of adjacent long rectangles are dotted.}
  \label{Fig:TrainTrack1}
\end{figure}

\begin{figure}
	\def\svgscale{0.6}
  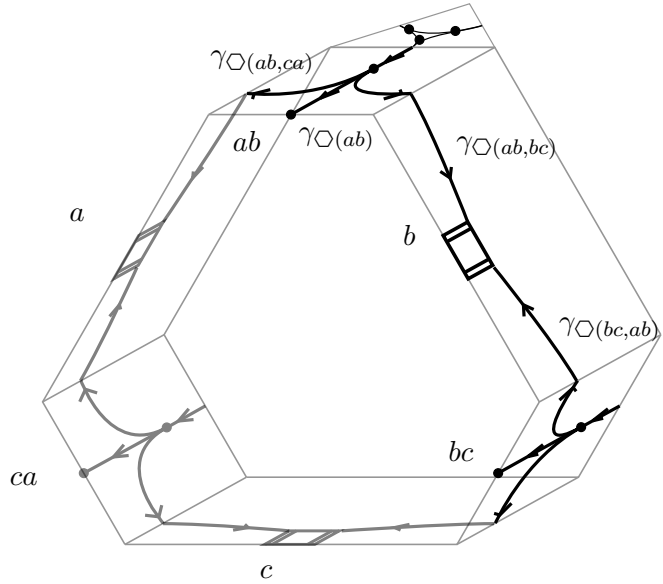 \\
  \caption{Train tracks in a neighbourhood of a hexagonal face $\hexagon$ in a polyhedron $\p$. Edges of $\hexagon$ are given $abc$-labels. Stations are shown along long rectangles. A triangular face adjacent to the $ab$-short rectangle is also shown.}
  \label{Fig:TrainTrack2}
\end{figure}

\begin{figure}
	\includegraphics[scale=0.6]{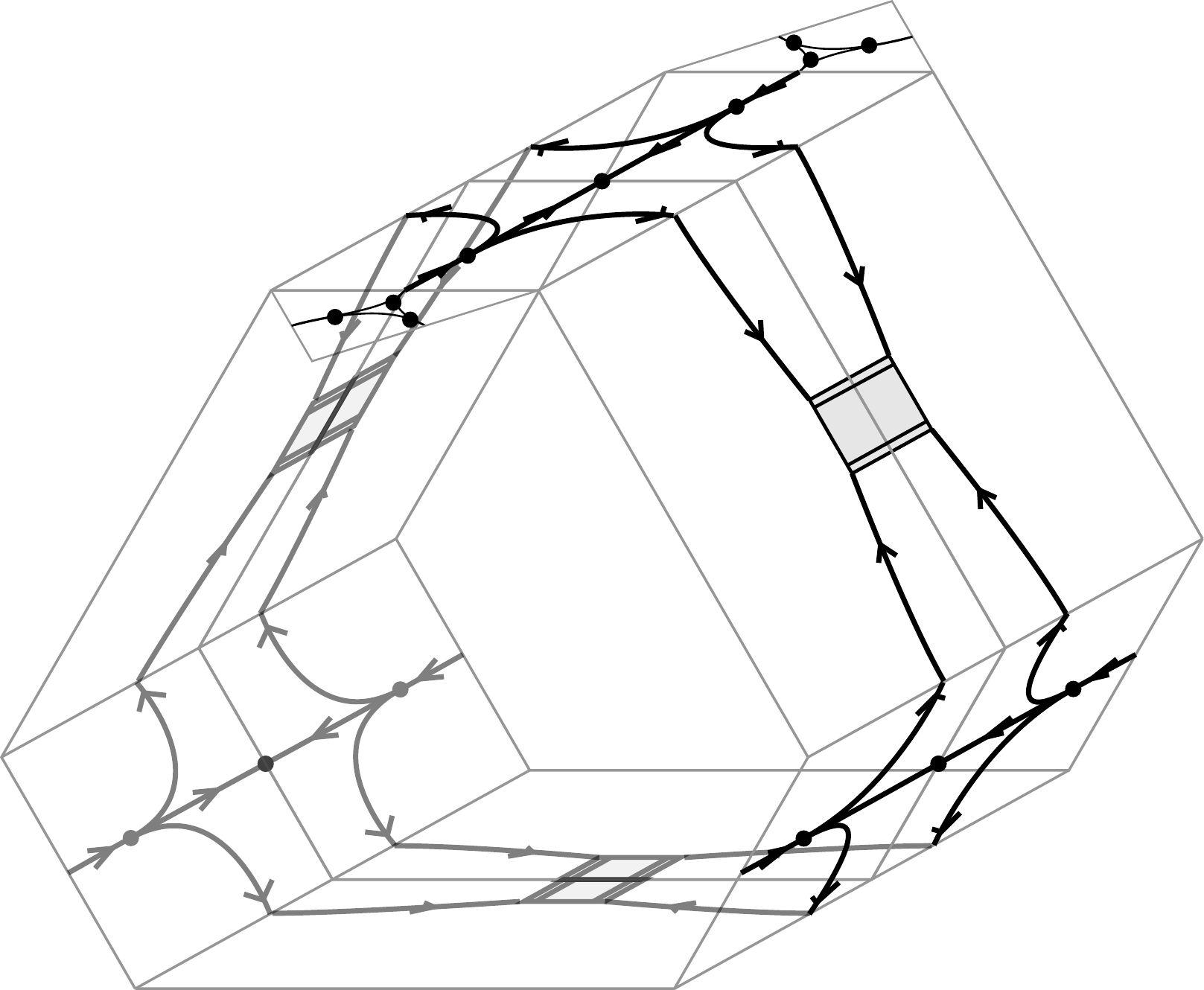}
  \caption{Train tracks near a glued pair of hexagonal faces. There are 1-switches and stations on the common boundary.}
  \label{Fig:hexagons_glued}
\end{figure}

In each short rectangle adjacent to $\Delta$, we place a 3-switch and a 1-switch as follows. As discussed above, each short rectangle lies in a regular neighbourhood of a hexagonal face of $\p$, which we denote $\hexagon$. The train track branch from $\Delta$ (which is $\gamma_{\Delta(ab)} = \gamma_{\hexagon(ab)!}$, $\gamma_{\Delta(bc)}= \gamma_{\hexagon(bc)!}$ or $\gamma_{\Delta(ca)} = \gamma_{\hexagon(ca)!}$) runs into the 3-switch on one side, oriented into the switch. 
Three branches run out of the other side, oriented away from the switch: a left, central and right branch. The left and right branches proceed into adjacent long rectangles, and the central branch proceeds to the opposite side of the short rectangle.

The left branch meets an edge of a long rectangle also adjacent to $\hexagon$, and proceeds into that long rectangle. If the long rectangle has type $a$ (resp.\ $b$,$c$), then the branch starts near the short edge of $\hexagon$ labeled $ca$ (resp.\ $ab$,$bc$) and proceeds towards the short edge of $\hexagon$ labeled $ab$ (resp.\ $bc$,$ca$), so we denote this branch $\gamma_{\hexagon(ca,ab)}$ (resp.\ $\gamma_{\hexagon(ab,bc)}$, $\gamma_{\hexagon(bc,ca)}$).

The right branch also meets an edge of a long rectangle adjacent to $\hexagon$ and proceeds into that long rectangle. If the long rectangle has type $a$ (resp.\ $b$, $c$), then the branch starts near the short side of $\hexagon$ labelled $ab$ (resp.\ $bc$, $ca$) and proceeds towards the short edge of $\hexagon$ labelled $ca$ (resp. $ab$, $bc$), so we denote this branch $\gamma_{\hexagon(ab,ca)}$ (resp.\ $\gamma_{\hexagon(bc,ab)}$, $\gamma_{\hexagon(ca,bc)}$).

The central branch runs straight out to a point on the final side of the short rectangle, which is also a short side of $\hexagon$. At this point we place a 1-switch. The short rectangle and short side of $\hexagon$ have a label $ab$, $bc$, or $ca$; we denote this central branch $\gamma_{\hexagon(ab)}$, $\gamma_{\hexagon(bc)}$ or $\gamma_{\hexagon(ca)}$ accordingly. Note that in the polyhedron $\p$, only one branch is incident with the 1-switch. However, when the polyhedra are glued together, the 1-switch also lies on the boundary of another polyhedron $\widehat{\p}$ glued to $\p$, and the corresponding train track branch on $\widehat{\p}$ is also incident with the 1-switch.
See \reffig{hexagons_glued}.

Applying \refdef{intersection_number_branches}, the local intersection numbers are then immediately given by the following lemma.
\begin{lemma}
\label{Lem:intersection_numbers_short_rectangles}
Inside each short rectangle of $\p$ with label $k \in \{ab,bc,ca\}$, adjacent to triangle $\Delta$ and hexagon $\hexagon$, intersection numbers at the 3-switch are given by
\begin{gather*}
\gamma_{\hexagon(k,m)} \cdot \gamma_{\hexagon(k)} 
= \gamma_{\hexagon(k,m)} \cdot \gamma_{\hexagon(k,l)}
= \gamma_{\hexagon(k)} \cdot \gamma_{\hexagon(k,l)} = 1, \\
\gamma_{\hexagon(k,l)} \cdot \gamma_{\hexagon(k)}
= \gamma_{\hexagon(k,l)} \cdot \gamma_{\hexagon(k,m)}
= \gamma_{\hexagon(k)} \cdot \gamma_{\hexagon(k,m)}= -1,
\end{gather*}
where $k,l,m$ are in cyclic alphabetical order. All other intersection numbers are $0$.

Thus by \refdef{epsilon}, the intersection numbers can be written as
  \[
  \pushQED{\qed}
\gamma_{\hexagon(k)} \cdot \gamma_{\hexagon(k,l)} = \epsilon_{k,l}, \quad
\gamma_{\hexagon(k,l)} \cdot \gamma_{\hexagon(k,m)} = -\epsilon_{l,m} \qedhere
\popQED
\]
\end{lemma}

Finally, we describe the train track on the long rectangles. We have already seen that a branch enters each long rectangle, oriented into it from each adjacent short rectangle. Thus two branches run into the long rectangle, one from each end, oriented into the rectangle. 
The long rectangle forms part of a regular neighbourhood of a hexagonal face $\hexagon$ of a polyhedron $\p$. The incoming edges from adjacent short rectangles are labelled $\gamma_{\hexagon(k,l)}$ and $\gamma_{\hexagon(l,k)}$ where $k$ and $l$ are distinct elements of $\{ab,bc,ca\}$. 
The train track on the long rectangle consists of these two branches running into a station placed on the boundary of the rectangle in such a way that two other branches in an adjacent long rectangle complete the 4-valent station. 
See \reffig{TrainTrack2} and \reffig{hexagons_glued}.

In particular, the face $\hexagon$ is glued to another hexagonal face $\widehat{\hexagon}$ of a polyhedron $\widehat{\p}$. In gluing these faces, an edge of the long rectangle is glued to an edge of another long rectangle in $\widehat{\p}$. We place the station on this glued edge so that the vertical direction coincides with the common edge (i.e.\ parallel to the original edge of the ideal tetrahedron $\tet$). The horizontal direction is perpendicular to the vertical direction. Then the two sides of the station correspond to the two rectangles, or the two polyhedra, on either side of the edge. The two ends of the station correspond to the two ends of the edge. All branches are oriented into the station.

The adjacent long rectangle in $\widehat{\p}$ is adjacent to the hexagonal face $\widehat{\hexagon}$, and contains two branches, labelled $\gamma_{\widehat{\hexagon}(\widehat{k},\widehat{l})}$ and $\gamma_{\widehat{\hexagon}(\widehat{l},\widehat{k})}$, with $\widehat{k}, \widehat{l}$ again distinct elements of $\{ab,bc,ca\}$, oriented into the station.
Reorder $\widehat{k}, \widehat{l}$ if necessary so that $\gamma_{\hexagon(k,l)}$ and $\gamma_{\widehat{\hexagon}(\widehat{k},\widehat{l})}$ lie at the same end but on opposite sides of the station. Then $\gamma_{\hexagon(l,k)}$ and $\gamma_{\widehat{\hexagon}(\widehat{l},\widehat{k})}$ lie at the other end of the station, on opposite sides. Further, $\gamma_{\hexagon(k,l)}$ and $\gamma_{\hexagon(l,k)}$ lie at opposite ends but on the same side of the station; similarly for $\gamma_{\widehat{\hexagon}(\widehat{k},\widehat{l})}$ and $\gamma_{\widehat{\hexagon}(\widehat{l},\widehat{k})}$. Thus in the notation of equation \refeqn{station_intersection}, overlines correspond to reversing the order of $k$ and $l$; and hats correspond to hats on hexagons, i.e.\ switching from one hexagonal face to the face paired with it under gluing.

\begin{figure}
  \def\svgscale{0.8}
\begingroup%
  \makeatletter%
  \providecommand\color[2][]{%
    \errmessage{(Inkscape) Color is used for the text in Inkscape, but the package 'color.sty' is not loaded}%
    \renewcommand\color[2][]{}%
  }%
  \providecommand\transparent[1]{%
    \errmessage{(Inkscape) Transparency is used (non-zero) for the text in Inkscape, but the package 'transparent.sty' is not loaded}%
    \renewcommand\transparent[1]{}%
  }%
  \providecommand\rotatebox[2]{#2}%
  \newcommand*\fsize{\dimexpr\f@size pt\relax}%
  \newcommand*\lineheight[1]{\fontsize{\fsize}{#1\fsize}\selectfont}%
  \ifx\svgwidth\undefined%
    \setlength{\unitlength}{428.11465204bp}%
    \ifx\svgscale\undefined%
      \relax%
    \else%
      \setlength{\unitlength}{\unitlength * \real{\svgscale}}%
    \fi%
  \else%
    \setlength{\unitlength}{\svgwidth}%
  \fi%
  \global\let\svgwidth\undefined%
  \global\let\svgscale\undefined%
  \makeatother%
  \begin{picture}(1,0.78151679)%
    \lineheight{1}%
    \setlength\tabcolsep{0pt}%
    \put(0,0){\includegraphics[width=\unitlength,page=1]{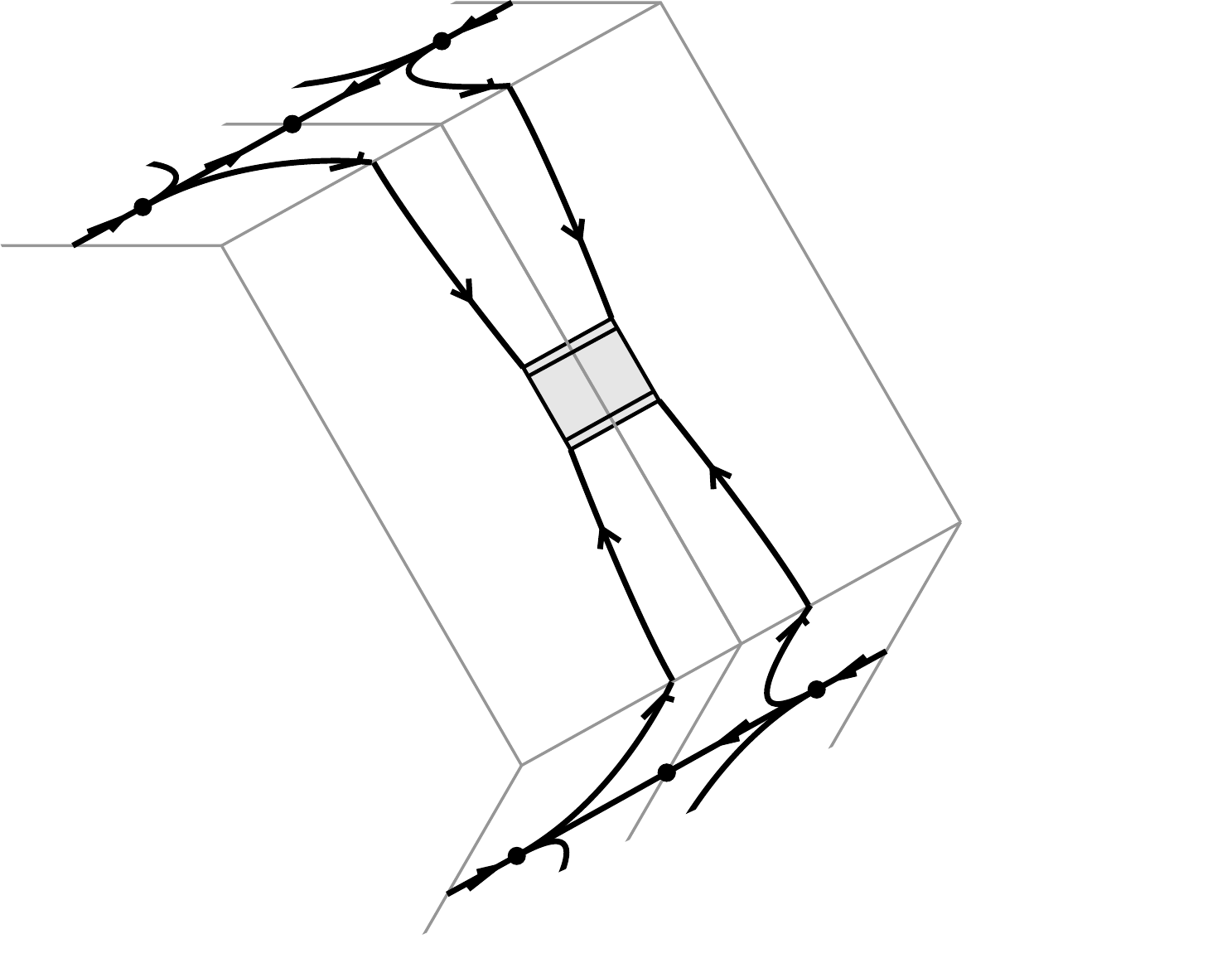}}%
    \put(0.46238317,0.63319724){\color[rgb]{0,0,0}\makebox(0,0)[lt]{\lineheight{1.25}\smash{\begin{tabular}[t]{l}$\gamma_{\hexagon(k,l)}$\end{tabular}}}}%
    \put(0.59095987,0.40297628){\color[rgb]{0,0,0}\makebox(0,0)[lt]{\lineheight{1.25}\smash{\begin{tabular}[t]{l}$\gamma_{\hexagon(l,k)}$\end{tabular}}}}%
    \put(0.27385936,0.55431351){\color[rgb]{0,0,0}\makebox(0,0)[lt]{\lineheight{1.25}\smash{\begin{tabular}[t]{l}$\gamma_{\widehat{\hexagon}(\widehat{k},\widehat{l})}$\end{tabular}}}}%
    \put(0.41707575,0.31007477){\color[rgb]{0,0,0}\makebox(0,0)[lt]{\lineheight{1.25}\smash{\begin{tabular}[t]{l}$\gamma_{\widehat{\hexagon}(\widehat{l},\widehat{k})}$\end{tabular}}}}%
  \end{picture}%
\endgroup%
 \\
  \caption{Train track branches in adjacent long rectangles.}
  \label{Fig:glued_long_rectangles}
\end{figure}
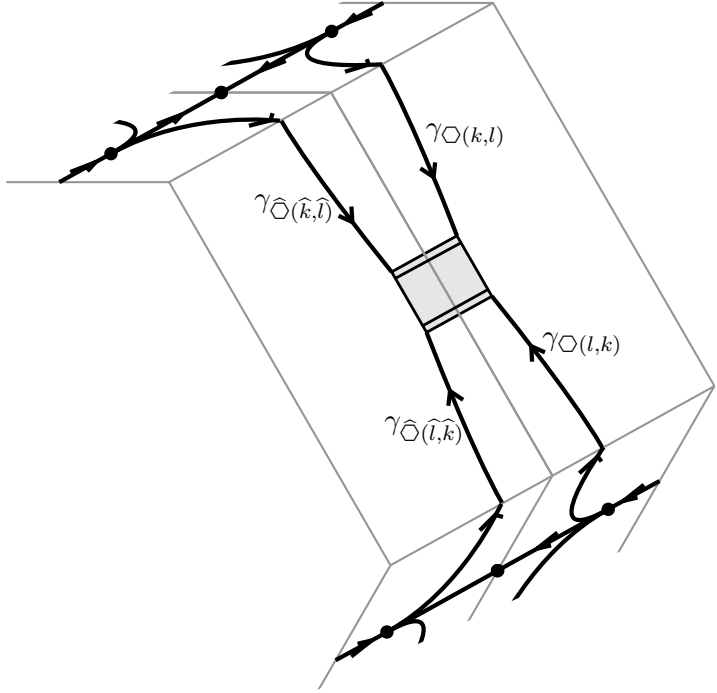

When polyhedra are glued, their hexagonal faces are glued in an orientation-reversing manner. Hence $k,l$ and $\widehat{k}, \widehat{l}$ have opposite cyclic-alphabetical orders. By switching labels if necessary, we may assume that $k,l$ are in cyclic alphabetical order and $(\widehat{k},\widehat{l})$ are in cyclic anti-alphabetical order. See \reffig{glued_long_rectangles}.

Then, in clockwise order around the station (as always, viewed from outside the glued polyhedra) we have $\gamma_{\hexagon(k,l)}$, $\gamma_{\hexagon(l,k)}$, $\gamma_{\widehat{\hexagon}(\widehat{l},\widehat{k})}$, $\gamma_{\widehat{\hexagon}(\widehat{k},\widehat{l})}$. 
Thus in the notation of \refsec{train_tracks_stations}
we may take 
$\gamma = \gamma_{\hexagon(l,k)}$ and $\delta = \overline{\gamma} = \gamma_{\hexagon(k,l)}$,
and intersection numbers are as follows.
\begin{lemma}
\label{Lem:intersection_numbers_long_rectangles}
At a station adjacent to hexagonal faces $\hexagon, \widehat{\hexagon}$ with incoming branches labelled $k,l$ in cyclic alphabetical order and $\widehat{k},\widehat{l}$ in cyclic anti-alphabetical order, 
\begin{gather*}
\gamma_{\hexagon(k,l)} \cdot \gamma_{\hexagon(l,k)} = 1, \quad
\gamma_{\hexagon(l,k)} \cdot \gamma_{\hexagon(k,l)} = -1, \\
\gamma_{\widehat{\hexagon}(\widehat{l},\widehat{k})} \cdot \gamma_{\widehat{\hexagon}(\widehat{k},\widehat{l})} = 1, \quad
\gamma_{\widehat{\hexagon}(\widehat{k},\widehat{l})} \cdot \gamma_{\widehat{\hexagon}(\widehat{l},\widehat{k})} = -1
\end{gather*}
and all other intersection numbers are $0$.
Using \refdef{epsilon}, this can be expressed as
  \[
  \pushQED{\qed}
  \gamma_{\hexagon(k,l)} \cdot \gamma_{\hexagon(l,k)} = \epsilon_{kl}. \qedhere
  \popQED
\]
\end{lemma}

\begin{figure}
	\def\svgscale{0.4}
\begingroup%
  \makeatletter%
  \providecommand\color[2][]{%
    \errmessage{(Inkscape) Color is used for the text in Inkscape, but the package 'color.sty' is not loaded}%
    \renewcommand\color[2][]{}%
  }%
  \providecommand\transparent[1]{%
    \errmessage{(Inkscape) Transparency is used (non-zero) for the text in Inkscape, but the package 'transparent.sty' is not loaded}%
    \renewcommand\transparent[1]{}%
  }%
  \providecommand\rotatebox[2]{#2}%
  \newcommand*\fsize{\dimexpr\f@size pt\relax}%
  \newcommand*\lineheight[1]{\fontsize{\fsize}{#1\fsize}\selectfont}%
  \ifx\svgwidth\undefined%
    \setlength{\unitlength}{672.3212513bp}%
    \ifx\svgscale\undefined%
      \relax%
    \else%
      \setlength{\unitlength}{\unitlength * \real{\svgscale}}%
    \fi%
  \else%
    \setlength{\unitlength}{\svgwidth}%
  \fi%
  \global\let\svgwidth\undefined%
  \global\let\svgscale\undefined%
  \makeatother%
  \begin{picture}(1,0.73836478)%
    \lineheight{1}%
    \setlength\tabcolsep{0pt}%
    \put(0,0){\includegraphics[width=\unitlength,page=1]{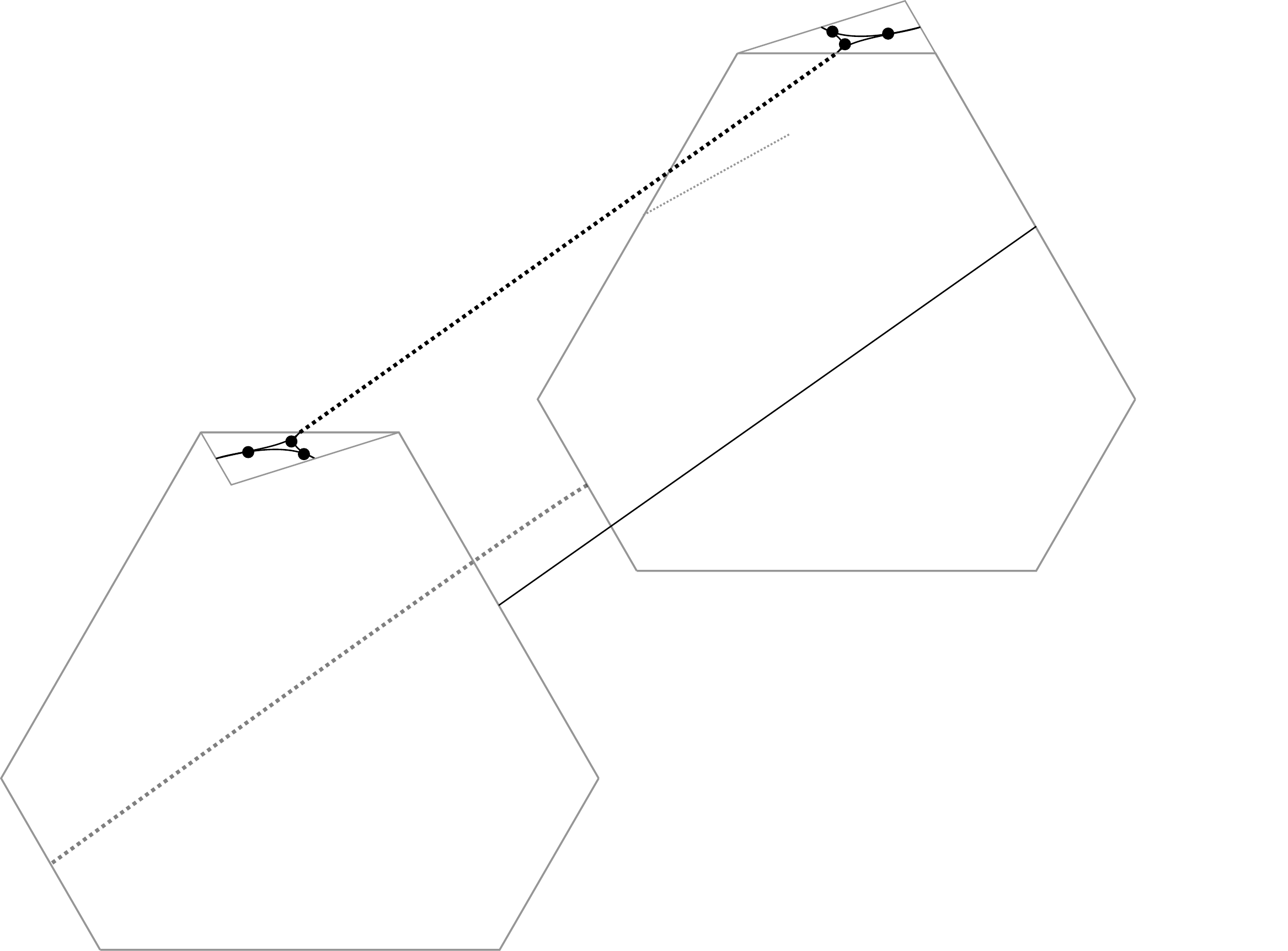}}%
    \put(0.20483862,0.28678387){\color[rgb]{0,0,0}\makebox(0,0)[lt]{\lineheight{1.25}\smash{\begin{tabular}[t]{l}$\gamma_{\widehat{\hexagon}(\widehat{k},\widehat{l})}$\end{tabular}}}}%
    \put(0,0){\includegraphics[width=\unitlength,page=2]{road_rules_motivation_2.pdf}}%
    \put(0.70694368,0.64415941){\color[rgb]{0,0,0}\makebox(0,0)[lt]{\lineheight{1.25}\smash{\begin{tabular}[t]{l}$\gamma_{\hexagon(k,l)}$\end{tabular}}}}%
    \put(0.83094962,0.46244774){\color[rgb]{0,0,0}\makebox(0,0)[lt]{\lineheight{1.25}\smash{\begin{tabular}[t]{l}$\gamma_{\hexagon(l,k)}$\end{tabular}}}}%
    \put(0.28051211,0.1297189){\color[rgb]{0,0,0}\makebox(0,0)[lt]{\lineheight{1.25}\smash{\begin{tabular}[t]{l}$\gamma_{\widehat{\hexagon}(\widehat{l},\widehat{k})}$\end{tabular}}}}%
    \put(0,0){\includegraphics[width=\unitlength,page=3]{road_rules_motivation_2.pdf}}%
  \end{picture}%
\endgroup%
 \\
  \caption{Topological motivation for the intersection form near stations.}
  \label{Fig:road_rules_motivation}
\end{figure}

\begin{remark}
We can now give some topological motivation for the unconventional intersection numbers at stations; see \reffig{road_rules_motivation}. Suppose we work with truncated tetrahedra $\overline{\tet}$ rather than polygons $\p$. Then we could draw curves similar to our train tracks, but rather more degenerate and singular, with short and long rectangles collapsing to sides of triangles, and curves $\gamma_{\hexagon(k,l)}, \gamma_{\hexagon(l,k)}$ actually drawn on the hexagon $\hexagon$, as in the top half of \reffig{road_rules_motivation}.  Similarly, $\gamma_{\widehat{\hexagon}(\widehat{l},\widehat{k})}$ and $\gamma_{\widehat{\hexagon}(\widehat{k},\widehat{l})}$ could be drawn on $\widehat{\hexagon}$. Then these curves naturally intersect  at the midpoints of long edges with intersection numbers (viewed on the 
boundary of $\overline{\tet}$) as we have defined them. When the hexagons $\hexagon, \widehat{\hexagon}$ are glued, the midpoints of long edges join (illustrated by a double line in \reffig{road_rules_motivation}) to produce a station. Such a description was originally used by the authors to state our results. Train tracks give us a more conventional, more easily described, less singular description of the curves involved, at the cost of introducing various intricacies along short and long rectangles, and the intersection form appearing unconventional.
\end{remark}

We have now completely described an oriented enhanced train track on each polyhedron $\p$. These fit together under gluing. When the polyhedra are glued into a Heegaard handlebody, the train tracks remain on the boundary of the handlebody, hence provide an oriented enhanced train track on the Heegaard surface $H$.

\subsection{Oscillating curves on the Heegaard surface}

Combining the construction of the oriented enhanced train track $\tau$ on $H$ above, with \refdef{abstract_oscillating_curve} of an abstract oscillating curve $\zeta$, we immediately obtain the following description of abstract oscillating curves on $\tau$.

First we simplify some notation. Recall that we write $\zeta = \sum_\gamma n_\gamma \gamma$, over the branches $\gamma$ of $\tau$. Each branch has a name such as $\gamma_{\Delta(a)}, \gamma_{\Delta(bc)}, \gamma_{\hexagon(ab)!}, \gamma_{\hexagon(bc)}, \gamma_{\hexagon(ab,ca)}, \gamma_{\hexagon(ca,bc)}$, etc. In each case, rather than writing $n_{\gamma_\bullet}$ we simply write $n_\bullet$. Thus, for instance, we write
\begin{gather}
\label{Eqn:n_simplify-1}
n_{\Delta(a)} = n_{\gamma_{\Delta(a)}}, \quad
n_{\Delta(bc)} = n_{\gamma_{\Delta(bc)}}, \quad
n_{\hexagon(ab)!} = n_{\gamma_{\hexagon(ab)!}}, \\
\label{Eqn:n_simplify-2}
n_{\hexagon(bc)} = n_{\gamma_{\hexagon(bc)}}, \quad
n_{\hexagon(ab,ca)} = n_{\gamma_{\hexagon(ab,ca)}}, \quad
\text{etc.}
\end{gather}

\begin{lemma}
\label{Lem:oscillating_conditions}
An abstract oscillating curve on $\tau$ is given by $\zeta = \sum_\gamma n_\gamma \gamma$ precisely when the $n_\gamma$ satisfy the following equations:
\begin{enumerate}
\item
For each triangle $\Delta$ of each polyhedron $\p$,
\[
n_{\Delta(ab)} = n_{\Delta(a)} - n_{\Delta(b)}, \quad
n_{\Delta(bc)} = n_{\Delta(b)} - n_{\Delta(c)}, \quad
n_{\Delta(ca)} = n_{\Delta(c)} - n_{\Delta(a)}.
\]
\item
For each short rectangle adjacent to a triangle $\Delta$ and hexagon $\hexagon$ in a polyhedron $\p$, with label $k$ and with $\{k,l,m\} = \{ab,bc,ca\}$,
\[
n_{\Delta(k)} = n_{\hexagon(k)} + n_{\hexagon(k,l)} + n_{\hexagon(k,m)}.
\]
\item
For each pair of glued short rectangles, adjacent to glued hexagonal faces $\hexagon, \widehat{\hexagon}$, and with labels $k, \widehat{k} \in \{ab,bc,ca\}$ respectively,
\[
n_{\hexagon(k)} + n_{\widehat{\hexagon}(\widehat{k})} = 0.
\]
\item
For each pair of glued long rectangles, adjacent to glued hexagonal faces $\hexagon, \widehat{\hexagon}$, respectively with labels $(k,l)$ and $(\widehat{k},\widehat{l})$ in clockwise and anticlockwise order respectively,
\[
n_{\hexagon(k,l)} + n_{\widehat{\hexagon}(\widehat{k},\widehat{l})} = n_{\hexagon(l,k)} + n_{\widehat{\hexagon}(\widehat{l},\widehat{k})}.
\]
\end{enumerate}
\end{lemma}

We note that (i) above implies that for each triangle $\Delta$,
\begin{equation}
\label{Eqn:triangle_sum_zero}
n_{\Delta(ab)} + n_{\Delta(bc)} + n_{\Delta(ca)} = 0.
\end{equation}
Also observe that (ii) may be rewritten as
\begin{equation}
\label{Eqn:avast_exclamation}
n_{\hexagon(k)!} = n_{\hexagon(k)} + n_{\hexagon(k,l)} + n_{\hexagon(k,m)}
\end{equation}

\begin{proof}
The equations in (i) are the compatibility conditions from \refdef{abstract_oscillating_curve} at the three 2-switches inside each triangle $\Delta$. The equations in (ii) are the compatibility conditions at the 3-switch inside each short rectangle. One such condition is $n_{\Delta(ab)} = n_{\hexagon(ab)} + n_{\hexagon(ab,bc)} + n_{\hexagon(ab,ca)}$ and the others are obtained by cyclically permuting $ab,bc,ca$, hence the conditions can be expressed as claimed. The equations in (iii) are the compatibility conditions at each 1-switch where pairs of short rectangles meet. And the equations in (iv) are the compatibility conditions at stations.
\end{proof}

\section{Symplectic vector space and combinatorial holonomy}
\label{Sec:CombinatorialHolonomy}

Let the polyhedra of the handlebody decomposition be $\p_1, \ldots, \p_n$; they are truncations of the ideal tetrahedra $\tet_1, \ldots, \tet_n$ of $\calT$. We have given labels using the letters $a,b,c$ to many aspects of these polyhedra.

We now define a symplectic vector space which again goes back at least to Neumann and Zagier \cite{NeumannZagier}, based on these $abc$s.

\begin{definition}
  Let $\V$ be a $2n$-dimensional real vector space generated by $3n$ elements
  \[ \a_1, \b_1, \cc_1, \ldots, \a_n, \b_n, \cc_n,\]
  subject to the $n$ relations
\[
\a_i + \b_i + \cc_i = \0 \quad \text{for $i = 1, \ldots, n$.}
\]
\end{definition}

There is a natural symplectic form $\omega$ on $\V$, given on generators by
\[
\omega ( \a_i, \b_j ) = 
\omega ( \b_i, \cc_j ) = 
\omega ( \cc_i, \a_j ) =
\delta_{i,j},
\]
\[
\omega ( \b_i, \a_j ) =
\omega ( \cc_i, \b_j ) =
\omega ( \a_i, \cc_j ) = - \delta_{i,j}.
\]
The generator $\a_i$ is associated to objects in $\tet_i$ or $\p_i$ of type $a$. Similarly, $\b_i$ and $\cc_i$ are associated to $b$- and $c$-type objects.

Given an abstract oscillating curve $\zeta$, we now define its combinatorial holonomy as an element of $\V$.

\begin{definition}\label{Def:Holonomy}
Let $\gamma$ be a branch of the enhanced train track $\tau$ on $H$. Its \emph{combinatorial holonomy} is
\[
h(\gamma) = \left\{ \begin{array}{ll} 
\a_i & \text{$\gamma = \gamma_{\Delta(a)}$ where $\Delta$ is a triangular face of $\p_i$} \\ 
\b_i & \text{$\gamma = \gamma_{\Delta(b)}$ where $\Delta$ is a triangular face of $\p_i$} \\
\cc_i & \text{$\gamma = \gamma_{\Delta(c)}$ where $\Delta$ is a triangular face of $\p_i$} \\
0 & \text{otherwise.} \end{array} \right.
\]
The combinatorial holonomy $h(\zeta)$ of an abstract oscillating curve $\zeta = \sum_\gamma n_\gamma \gamma$ is then defined by extending linearly:
\[
h(\zeta) = \sum_\gamma n_\gamma h(\gamma).
\]
\end{definition}
As each oscillating curve defines an abstract oscillating curve, it too has a combinatorial holonomy.

Note that the combinatorial holonomy effectively ignores any part of an abstract oscillating curve $\zeta$ which runs along short or long rectangles of polyhedra $\p_i$, including all those parts of $\zeta$ which proceed through the interior of $\overline{M}$. Moreover, in each polyhedron $\p_i$ there are four triangular faces, and each triangular face has a branch of type $\gamma_{\Delta(a)}$, $\gamma_{\Delta(b)}$, and $\gamma_{\Delta(c)}$. The distinction between the four branches of $\tau$ in $\p_i$ of type $\gamma_{\Delta(a)}$ is lost in the combinatorial holonomy, as is the distinction between the four branches of type $\gamma_{\Delta(b)}$ and $\gamma_{\Delta(c)}$. It is not at all immediately clear, then, that the combinatorial holonomy can be used to extract any meaningful geometric or topological information about an oscillating curve.

The following is the main theorem of the paper, and it asserts that the combinatorial holonomy is sufficient to extract intersection numbers, as defined in \refdef{intersection_number}. Indeed, the symplectic form $\omega$ is the intersection form.

\begin{thm}\label{Thm:MainThm}
For any abstract oscillating curves  $\zeta, \zeta'$ in $M$,
\[
\omega ( h(\zeta), h(\zeta')) = 2 \zeta \cdot \zeta'.
\]
\end{thm}

The next section will be devoted to the proof of \refthm{MainThm}.

\section{Evaluating $\omega$ on oscillating curves}
\label{Sec:omega_on_oscillating}

To prove \refthm{MainThm}, we express the intersection number as a sum over faces, in \refsec{IntOverFaces}, and do the same for $\omega$ in \refsec{OmegaOverFaces}. Then we show these sums are equal in \refsec{EvaluatingOmega}. 

We write our two abstract oscillating curves as $\zeta = \sum_\gamma n_\gamma \gamma$ and $\zeta' = \sum_\gamma n'_\gamma \gamma$, where the sums are over all the oriented branches $\gamma$ of the enhanced oriented train track $\tau$. Throughout, we write the coefficients $n_\gamma, n'_\gamma$ by simplifying subscripts as in \refeqn{n_simplify-1}--\refeqn{n_simplify-2}; we introduce further simplifications as we proceed.

\subsection{Intersection number as a sum over faces}
\label{Sec:IntOverFaces}

We express $2\zeta \cdot \zeta'$ as a sum over over triangular faces and glued hexagonal faces of the polyhedra $\p_i$; equivalently, over cusp triangles and faces of the triangulation $\TT$.

From \refdef{intersection_number}, $2 \zeta \cdot \zeta'$ is the sum of intersection numbers at switches of $\tau$, hence there are three types:
\begin{enumerate}
\item \emph{Triangular faces}: each triangular face $\Delta$ of a polyhedron $\p$ contains three 2-switches.
\item \emph{Short rectangles}: each short rectangle of a polyhedron $\p$ contains a 3-switch and a 1-switch. By \refdef{intersection_number_branches} the 1-switch does not contribute to the intersection number. Each polyhedron $\p$ contains twelve such rectangles: three per triangular face, or alternatively, three per hexagonal face. Alternatively again, there are six such rectangles per pair of glued hexagonal faces.
\item \emph{Long rectangles}: 
Each pair of glued long rectangles contains a station. Since each hexagonal face of $\p$ contains three long edges, there are three such stations per pair of glued faces.
\end{enumerate}
Thus $2\zeta \cdot \zeta'$ is a sum of three terms, arising from sums over the switches of each type; we denote them type (i), (ii), (iii) accordingly.

In case (i) the intersections occur in a specific triangle $\Delta$ of a polyhedron $\p_i$, and we express the sum of such intersections as a sum over triangles.

In cases (ii) and (iii) the intersections occur near a glued pair of hexagonal faces, denoted $\hexagon, \widehat{\hexagon}$. We express the sum of such intersections as a sum over such pairs $\hexagon, \widehat{\hexagon}$.

\subsubsection{Contributions from triangles.}
Here we follow Choi in part (b) of the final proof of \cite{Choi:NZ}. Recall that each triangular face $\Delta$ of a polyhedron $\p$ has vertices labelled $a,b,c$, and edges labelled $ab,bc,ca$ in anticlockwise order. Three branches $\gamma_{\Delta(a)}, \gamma_{\Delta(b)}, \gamma_{\Delta(c)}$ of $\tau$ run anticlockwise around the vertices labelled $a,b,c$. The edges labelled $ab,bc,ca$ respectively meet branches $\gamma_{\Delta(ab)}$, $\gamma_{\Delta(bc)}$, and $\gamma_{\Delta(ca)}$, as in \reffig{TrainTrack1}. Recall $n_\gamma, n_\gamma'$ denote the signed number of arcs of $\zeta, \zeta'$ respectively along the oriented branch $\gamma$. We apply the simplifications of \refeqn{n_simplify-1}--\refeqn{n_simplify-2} to the $n_\gamma'$ as for $n_\gamma$, adding a $'$ to each.

\begin{lem}\label{Lem:ContribTriangles}
The terms of type (i) in $2 \zeta \cdot \zeta'$ are given by
\[
\sum_{\Delta} n_{\Delta(ab)} n'_{\Delta(bc)} - n_{\Delta(bc)} n'_{\Delta(ab)}
\]
where the sum is over all triangular faces $\Delta$ of polyhedra $\p$.
\end{lem}
The summand here is Choi's $\kappa(AB^i) \kappa' (BC^i) - \kappa' (AB^i) \kappa (BC^i)$. 

\begin{proof}
In triangle $\Delta$, the intersection numbers between branches are given by \reflem{intersection_numbers_triangles}. Any nonzero intersection number comes from the arcs $\gamma_{\Delta(a)}, \gamma_{\Delta(b)}, \gamma_{\Delta(c)}$. So the contribution to $2\zeta \cdot \zeta'$ from the switches in $\Delta$ is given by
\begin{gather*}
\left( 
n_{\Delta(a)} \gamma_{\Delta(a)} + n_{\Delta(b)} \gamma_{\Delta(b)} + n_{\Delta(c)} \gamma_{\Delta(c)}
\right) \cdot  
\\
\hspace{1in}
\left(
{n'}_{\Delta(a)} \gamma_{\Delta(a)} + {n'}_{\Delta(b)} \gamma_{\Delta(b)} + {n'}_{\Delta(c)} \gamma_{\Delta(c)}
\right).
\end{gather*}
Expanding and using \reflem{intersection_numbers_triangles} yields
\[
n_{\Delta(a)} {n'}_{\Delta(b)}
- n_{\Delta(b)} {n'}_{\Delta(a)}
+ n_{\Delta(b)} {n'}_{\Delta(c)}
- n_{\Delta(c)} {n'}_{\Delta(b)}
+ n_{\Delta(c)} {n'}_{\Delta(a)}
- n_{\Delta(a)} {n'}_{\Delta(c)}.
\]
On the other hand, the claimed summand can be expanded using \reflem{oscillating_conditions}(i)
\begin{gather*}
  n_{\Delta(ab)} n'_{\Delta(bc)} - n_{\Delta(bc)} n'_{\Delta(ab)} = \\
\left( n_{\Delta(a)} - n_{\Delta(b)} \right)
\left( {n'}_{\Delta(b)} - {n'}_{\Delta(c)} \right)
-
\left( n_{\Delta(b)} - n_{\Delta(c)} \right)
\left( {n'}_{\Delta(a)} - {n'}_{\Delta(b)} \right),
\end{gather*}
which simplifies to precisely the same expression.
\end{proof}

\subsubsection{Contributions from hexagons} 
\label{sec:hexagon_contributions}
We now count intersections of type (ii) and (iii). As discussed above, each such intersection arises near a pair of glued hexagonal faces $\hexagon, \widehat{\hexagon}$ of polyhedra $\p,\widehat{\p}$. Recall each hexagon has long sides labelled $a,b,c$ in clockwise order and short sides labelled $ab,bc,ca$ in clockwise order. For each $k \in \{ab,bc,ca\}$ we write $\widehat{k}$ for the element of $\{ab,bc,ca\}$ such that the short side of $\hexagon$ labelled $k$ is glued to the short side of $\widehat{\hexagon}$ labelled $\widehat{k}$. Label the short sides of $\hexagon$ as $k,l,m$ in arbitrary clockwise order around the face, so $k,l,m$ are in cyclic alphabetical order and are unique up to cyclic permutation. Then the short sides of $\widehat{\hexagon}$ labelled $\widehat{k}, \widehat{l}, \widehat{m}$ are in anticlockwise order around the face, and $\widehat{k},\widehat{l},\widehat{m}$ are in cyclic anti-alphabetical order. We will take such a labelling of glued hexagons throughout the sequel.

Type (ii) terms come from short rectangles. Each glued pair of hexagons $\hexagon, \widehat{\hexagon}$ has a pair of adjacent short rectangles labelled $k$ in $\p$ and $\widehat{k}$ in $\widehat{\p}$; and another two pairs labelled $l, \widehat{l}$ and $m,\widehat{m}$ respectively.

We consider the pair of short rectangles labelled $k$ and $\widehat{k}$; the pairs labelled $l,\widehat{l}$ and $m,\widehat{m}$ are similar. There is a 3-switch in each short rectangle and a 1-switch along their common boundary; only the 3-switch contributes nontrivially to intersection numbers, by \refdef{intersection_number_branches}. 

The rectangle labelled $k$ has a branch labelled $\gamma_{\Delta(k)}$ on one side, where $\Delta$ is the adjacent triangle; this is also denoted $\gamma_{\hexagon(k)!}$. It has additional branches labelled $\gamma_{\hexagon(k,m)}$, $\gamma_{\hexagon(k)}$, $\gamma_{\hexagon(k,l)}$ on the other side, in anticlockwise order. When $\hexagon$ is understood, for ease of notation we rewrite these branches more simply as
\[
\gamma_{kl} = \gamma_{\hexagon(k,l)}, \quad
\gamma_k = \gamma_{\hexagon(k)}, \quad
\gamma_{km}= \gamma_{\hexagon(k,m)}.
\]
See \reffig{glued_short_rectangles}.
Then as per \reflem{intersection_numbers_short_rectangles}, since $k,l,m$ are in cyclic alphabetical order, the nonzero intersection numbers at this 3-switch are
\begin{gather}
\label{Eqn:k_intersection_numbers-1}
\gamma_{km} \cdot \gamma_k
= \gamma_{km} \cdot \gamma_{kl}
= \gamma_k \cdot \gamma_{kl} =1, \\
\label{Eqn:k_intersection_numbers-2}
\gamma_{kl} \cdot \gamma_k
= \gamma_{kl} \cdot \gamma_{km}
= \gamma_k \cdot \gamma_{km} = -1.
\end{gather}
Similarly, the 3-switch in the rectangle labelled $\widehat{k}$ has $\gamma_{\widehat{\Delta}(\widehat{k})}$ on one side, where $\widehat{\Delta}$ is the adjacent triangle; and branches $\gamma_{\widehat{\hexagon}(\widehat{k},\widehat{l})}, \gamma_{\widehat{\hexagon}(\widehat{k})}, \gamma_{\widehat{\hexagon}(\widehat{k},\widehat{m})}$ on the other, in anticlockwise order. Again for ease of notation we rewrite
\[
\widehat{\gamma}_{km} = \gamma_{\widehat{\hexagon}(\widehat{k},\widehat{m})}, \quad
\widehat{\gamma}_k = \gamma_{\widehat{\hexagon}(\widehat{k})}, \quad
\widehat{\gamma}_{kl} = \gamma_{\widehat{\hexagon}(\widehat{k},\widehat{l})}
\]
so the nonzero intersection numbers among them are
\begin{gather}
\label{Eqn:khat_intersection_numbers-1}
\widehat{\gamma}_{kl} \cdot \widehat{\gamma}_k
= \widehat{\gamma}_{kl} \cdot \widehat{\gamma}_{km}
= \widehat{\gamma}_k \cdot \widehat{\gamma}_{km} = 1, \\
\label{Eqn:khat_intersection_numbers-2}
\widehat{\gamma}_{km} \cdot \widehat{\gamma}_k
= \widehat{\gamma}_{km} \cdot \widehat{\gamma}_{kl}
= \widehat{\gamma}_k \cdot \widehat{\gamma}_{kl} = -1.
\end{gather}

\begin{figure}
  \def\svgscale{0.8}
  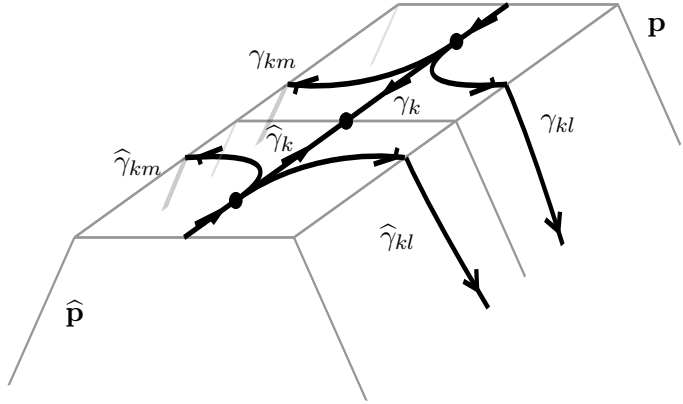 \\
  \caption{Train track branches in adjacent short rectangles.}
  \label{Fig:glued_short_rectangles}
\end{figure}

In the abstract oscillating curve $\zeta = \sum_\gamma n_\gamma \gamma$, we simplify subscripts on coefficients $n_\gamma$ as for the branches $\gamma$, further simplifying from \refeqn{n_simplify-1}--\refeqn{n_simplify-2}:
\begin{gather}
\label{Eqn:n_simplify-3}
n_{kl} = n_{\gamma_{\hexagon(k,l)}}, \quad
n_k = n_{\gamma_{\hexagon(k)}}, \quad
n_{km} = n_{\gamma_{\hexagon(k,l)}}, \\
\label{Eqn:n_simplify-4}
\widehat{n}_{km} = n_{\gamma_{\widehat{\hexagon}(\widehat{k},\widehat{m})}}, \quad
\widehat{n}_k = n_{\gamma_{\widehat{\hexagon}(\widehat{k})}}, \quad
\widehat{n}_{kl} = n_{\gamma_{\widehat{\hexagon}(\widehat{k},\widehat{l})}}.
\end{gather}

For another abstract oscillating curve $\zeta' = \sum_\gamma n'_\gamma \gamma$, we use the same notation, with the addition of a $'$.

We can now compute type (ii) terms.
\begin{lem}\label{Lem:Type2Switches}
Let $\hexagon$ and $\widehat{\hexagon}$ be glued faces in polyhedra $\p, \widehat{\p}$ (possibly $\p=\widehat{\p}$). Label the short sides of $\hexagon$ by $k,l,m$ in clockwise order, and $\widehat{k}, \widehat{l}, \widehat{m}$ in anticlockwise order, so that $k$ glues to $\widehat{k}$, $l$ to $\widehat{l}$ and $m$ to $\widehat{m}$. Then the contribution to $2\zeta\cdot \zeta'$ from switches in the short rectangles with labels $k,\widehat{k}$ is given by:
\begin{align*}
n_k(n'_{kl} - n'_{km} + \widehat{n}'_{kl} - \widehat{n}'_{km}) + 
n_k'(-n_{kl} + n_{km} - \widehat{n}_{kl} + \widehat{n}_{km}) + \\
-n_{kl}n'_{km} + n_{km}n'_{kl} - \widehat{n}_{km}\widehat{n}'_{kl} + \widehat{n}_{kl} \widehat{n}'_{km}.
\end{align*}
\end{lem}

\begin{proof}
The contribution to $2\zeta \cdot \zeta'$ from the 3-switch in the short rectangle labelled $k$ is given by
\[
\left( n_{km} \gamma_{km} + n_k \gamma_k + n_{kl} \gamma_{kl} \right)
\cdot
\left( n'_{km} \gamma_{km} + n'_k \gamma_k + n'_{kl} \gamma_{kl} \right).
\]
Expanding and using the intersection numbers \refeqn{k_intersection_numbers-1}--\refeqn{k_intersection_numbers-2} above yields
\[
n_{km} (n'_k + n'_{kl}) + n_k (-n'_{km} + n'_{kl}) + n_{kl} (-n'_{km} - n'_k).
\]
Similarly, the contribution from the 3-switch in the short rectangle labelled $\widehat{k}$ is given by
\[
\left( \widehat{n}_{kl} \widehat{\gamma}_{kl} + \widehat{n}_k \widehat{\gamma}_k + \widehat{n}_{km} \widehat{\gamma}_{km} \right) 
\cdot
\left( \widehat{n}'_{kl} \widehat{\gamma}_{kl} + \widehat{n}'_k \widehat{\gamma}_k + \widehat{n}'_{km} \widehat{\gamma}_{km} \right) 
\]
which expands with \refeqn{khat_intersection_numbers-1}--\refeqn{khat_intersection_numbers-2} as
\[
\widehat{n}_{kl} ( \widehat{n}'_k + \widehat{n}'_{km})
+ \widehat{n}_k ( - \widehat{n}'_{kl} + \widehat{n}'_{km} )
+ \widehat{n}_{km} ( - \widehat{n}'_{kl} - \widehat{n}'_{k} ).
\]
Summing these terms, noting that $n_k + \widehat{n}_k = 0$ and $n'_k + \widehat{n}'_k = 0$ by \reflem{oscillating_conditions}(iii), and rearranging gives the desired expression.
\end{proof}

We now consider type (iii) terms, arising from stations on long rectangles. Consider again two glued hexagonal faces $\hexagon, \widehat{\hexagon}$ of $\p, \widehat{\p}$ with labels $k,l,m$ and $\widehat{k},\widehat{l},\widehat{m}$ as above. Consider the pair of long rectangles glued along long edges of $\hexagon, \widehat{\hexagon}$, with incoming branches $\gamma_{\hexagon(k,l)} = \gamma_{kl}$, $\gamma_{\hexagon(l,k)} = \gamma_{lk}$ and $\gamma_{\widehat{\hexagon}(\widehat{l},\widehat{k})} = \widehat{\gamma}_{lk}$, $\gamma_{\widehat{\hexagon}(\widehat{k},\widehat{l})} = \widehat{\gamma}_{kl}$ in clockwise order around the station, as described in \refsec{train_tracks_on_H} and \reffig{glued_long_rectangles}. 

\reflem{intersection_numbers_long_rectangles} then gives the nonzero intersection numbers at the station, in simplified notation, as
\begin{equation}
\label{Eqn:simplified_station_intersection} 
\gamma_{kl} \cdot \gamma_{lk} = 1, \quad
\gamma_{lk} \cdot \gamma_{kl} = -1, \quad
\widehat{\gamma}_{lk} \cdot \widehat{\gamma}_{kl} = 1, \quad
\widehat{\gamma}_{kl} \cdot \widehat{\gamma}_{lk} = -1.
\end{equation}

\begin{lem}\label{Lem:Type3Switches}
The contribution to $2\zeta\cdot\zeta$ from the station in the long rectangle between short sides labelled $k,l$ and $\widehat{k},\widehat{l}$ as above is
\[ 
n_{kl} n'_{lk} - n_{lk} n'_{kl} - \widehat{n}_{kl} \widehat{n}'_{lk} + \widehat{n}_{lk} \widehat{n}'_{kl}.
\]
\end{lem}

\begin{proof}
The relevant terms of $\zeta$ and $\zeta'$ give 
\[
\left( n_{kl} \gamma_{kl} + n_{lk} \gamma_{lk} + \widehat{n}_{kl} \widehat{\gamma}_{kl} + \widehat{n}_{lk} \widehat{\gamma}_{lk} \right) \cdot
\left( n'_{kl} \gamma_{kl} + n'_{lk} \gamma_{lk} + \widehat{n}'_{kl} \widehat{\gamma}_{kl} + \widehat{n}'_{lk} \widehat{\gamma}_{lk} \right).
\]
Expanding and using \refeqn{simplified_station_intersection} gives the desired expression.
\end{proof}

Combining \reflem{ContribTriangles}, \reflem{Type2Switches} and \reflem{Type3Switches} immediately gives the following expression for $2\zeta \cdot \zeta'$.
\begin{prop}
\label{Prop:IntOverFaces}
Using notation as above, $2\zeta \cdot \zeta'$ is given by
\begin{align}
\label{Eqn:intersection_number-1}
&\sum_\Delta n_{\Delta(ab)} n'_{\Delta(bc)} - n_{\Delta(bc)} n'_{\Delta(ab)} \\ 
\label{Eqn:intersection_number-2}
\begin{split}
&+ \sum_{\hexagon, \widehat{\hexagon}} \sum_{(k,l,m)} n_k (n'_{kl} - n'_{km} + \widehat{n}'_{kl} - \widehat{n}'_{km}) \\
&\hspace{3 cm} 
+ n'_k ( -n_{kl} + n_{km} - \widehat{n}_{kl} + \widehat{n}_{km} )
\end{split} \\
\label{Eqn:intersection_number-3}
&+ \sum_{\hexagon, \widehat{\hexagon}} \sum_{(k,l,m)} -n_{kl} n'_{km} + n_{km} n'_{kl} - \widehat{n}_{km} \widehat{n}'_{kl} + \widehat{n}_{kl} \widehat{n}'_{km} \\
\label{Eqn:intersection_number-4}
&+ \sum_{\hexagon, \widehat{\hexagon}} \sum_{(k,l,m)} 
n_{kl} n'_{lk} - n_{lk} n'_{kl} - \widehat{n}_{kl} \widehat{n}'_{lk} + \widehat{n}_{lk} \widehat{n}'_{kl}.
\end{align}
\end{prop}
The meaning of the summations is as follows: $\sum_\Delta$ means a sum over all triangles of all polyhedra $\p$. A summation $\sum_{(\hexagon, \widehat{\hexagon})}$ is a sum over pairs of glued hexagons. For each glued pair, we only count it once in the summation, arbitrarily choosing one to be $\hexagon$ and one to be $\widehat{\hexagon}$. Given such a pair, we choose labels $(k,l,m)$ and $(\widehat{k},\widehat{l},\widehat{m})$ such that $(k,l,m)$ is in cyclic alphabetical order (i.e.\ a cyclic permutation of $(ab,bc,ca)$; definition \refdef{cyclic_alphabetical_order}), $(\widehat{k},\widehat{l},\widehat{m})$ is in cyclic anti-alphabetical order (i.e.\ a cyclic permutation of $(ab,ca,bc)$), and so that the short sides of $\hexagon$ labelled $k,l,m$ are glued to the short sides of $\widehat{\hexagon}$ labelled $\widehat{k},\widehat{l},\widehat{m}$ respectively. Then the second summation $\sum_{(k,l,m)}$ denotes the sum over cyclic permutations of $k,l,m$.

Here \refeqn{intersection_number-1} gives type (i) intersections (from triangles), \refeqn{intersection_number-2}$+$\refeqn{intersection_number-3} gives type (ii) intersections (from short rectangles), and \refeqn{intersection_number-4} gives type (iii) intersections (from long rectangles). We will match these sums with sums arising from the symplectic form $\omega$.

\subsection{Expressing $\omega$ as a sum over faces}
\label{Sec:OmegaOverFaces}
As in \refdef{Holonomy}, $h(\zeta) = \sum_{\gamma} n_\gamma h(\gamma)$, where only those branches in triangles contribute, i.e. those of the form $\gamma_{\Delta(a)}$, $\gamma_{\Delta(b)}$ of $\gamma_{\Delta(c)}$ for some triangular face $\Delta$. 
As in \refeqn{n_simplify-1}, we write their coefficients as $n_{\Delta(a)}, n_{\Delta(b)}, n_{\Delta(c)}$ respectively. Thus
\begin{equation}
  \label{Eqn:holonomy_of_normal_arc}
  h(\zeta) = \sum_{i=1}^n \sum_{\Delta\subset{\p_i}} n_{\Delta(a)} \a_i + n_{\Delta(b)} \b_i + n_{\Delta(c)} \cc_i
\end{equation}
where the first sum is over all polyhedra $\p_1, \ldots, \p_n$ (corresponding to tetrahedra of the original triangulation $\calT$), and the second sum is over the four triangular faces of each $\p_i$. 

We then have the following result, which appears for cusp curves in \cite[page~49]{Choi:NZ}. 
\begin{lem}
\label{Lem:OmegaOverTriangles}
\[
\omega \left( h(\zeta), h(\zeta') \right)
= \sum_{\p} \sum_{\Delta, \nabla \subset \p}
n_{\Delta(ab)} n'_{\nabla(bc)} - n_{\nabla(bc)} n'_{\Delta(ab)}
\]
In this sum, $\Delta$ and $\nabla$ are triangular faces of the polyhedron $\p$, which may or may not coincide.
\end{lem}

\begin{proof}
Expanding out expressions for $h(\zeta)$ and $h(\zeta')$ from \refeqn{holonomy_of_normal_arc}, and noting that terms from different tetrahedra are $\omega$-orthogonal, $\omega(h(\zeta), h(\zeta'))$ is given by
\[
\sum_{i=1}^n \sum_{\Delta, \nabla \subset \p_i}
\omega \left(
n_{\Delta(a)} \a_i + n_{\Delta(b)} \b_i + n_{\Delta(c)} \cc_i,
n'_{\nabla(a)} \a_i + n'_{\nabla(b)} \b_i + n_{\nabla(c)} \cc_i \right)
\]
Since $\omega(\a_i, \b_i) = \omega(\b_i, \cc_i) = \omega(\cc_i, \a_i) = 1$, the above summand is
\[
n_{\Delta(a)} n'_{\nabla(b)} + n_{\Delta(b)} n'_{\nabla(c)} + n_{\Delta(c)} n'_{\nabla(a)} - n_{\Delta(a)} n'_{\nabla(c)} - n_{\Delta(b)}n'_{\nabla(a)} - n_{\Delta(c)} n'_{\nabla(b)}.
\]
Using the compatibility conditions from 
\reflem{oscillating_conditions}(i) (e.g. $n_{\Delta(ab)} = n_{\Delta(a)} - n_{\Delta(b)}$, etc.), we rewrite the summand as 
\begin{align*}
  &n_{\Delta(a)}(n'_{\nabla(b)} - n'_{\nabla(c)}) +
  n_{\Delta(b)}(n'_{\nabla(c)} - n'_{\nabla(a)}) +
  n_{\Delta(c)}(n'_{\nabla(a)} - n'_{\nabla(b)}) \\
	&\quad =  (n_{\Delta(a)}-n_{\Delta(b)})(n'_{\nabla(b)}-n'_{\nabla(c)}) -
  (n_{\Delta(c)}-n_{\Delta(b)})(n'_{\nabla(b)}-n'_{\nabla(a)}) \\
  &\quad = n_{\Delta(ab)} n'_{\nabla(bc)} - n_{\Delta(bc)} n'_{\nabla(ab)}
\end{align*}
Thus
\[
\omega (h(\zeta), h(\zeta')) =
\sum_{\p}
\sum_{\Delta, \nabla \subset \p}
n_{\Delta(ab)} n'_{\nabla(bc)} - \sum_{\p} \sum_{\Delta, \nabla \subset \p} n_{\Delta(bc)} n'_{\nabla(ab)} .
\]
Now recasting the sum over pairs $(\Delta, \nabla)$ as a sum over pairs $(\nabla, \Delta)$ in the second sum gives the desired expression.
\end{proof}
The good thing about recasting the second summand is that in the formula obtained for $\omega(h(\zeta), h(\zeta'))$, both terms involve the same sides $\Delta(ab)$ and $\nabla(bc)$. Indeed, such sides can be described quite precisely. 

Following Choi, we split the sum for $\omega(h(\zeta), h(\zeta'))$ into two parts: one where the triangles $\Delta, \nabla$ are the same triangle, and one where they are distinct triangles. The following observation about such sides is used in Choi's proof.
\begin{lem}
\label{Lem:WhereSidesLie}
Let $\Delta, \nabla$ be two distinct triangular faces of a truncated tetrahedron $\tau$. Then precisely one of the following occurs:
\begin{enumerate}
\item The sides $\Delta(ab)$ and $\nabla(bc)$ lie on the boundary of a single hexagonal face of $\tau$.
\item The side $\Delta(ab)$ lies on the boundary of a hexagonal face $\hexagon$ adjacent to both $\Delta$ and $\nabla$, but $\nabla(bc)$ does not. However $\nabla(ca)$ lies on the boundary of $\hexagon$.
\item The side $\nabla(bc)$ lies on the boundary of a hexagonal face $\hexagon$ adjacent to both $\Delta$ and $\nabla$, but $\Delta(ab)$ does not. However $\Delta(ca)$ lies on the boundary of $\hexagon$.
\end{enumerate}
\end{lem}

This lemma follows from inspecting \reffig{abcs} (or \reffig{abcs2} for the corresponding picture on a polyhedron $\p$); alternatively we can prove it explicitly as follows.

\begin{proof}
The triangular faces $\Delta, \nabla$ are adjacent to two common hexagonal faces, which we denote $\hexagon$ and $\hexagon'$. Of the three sides of $\Delta$, one is a short side of $\hexagon$, one is a short side of $\hexagon'$, and one is not a side of $\hexagon$ or $\hexagon'$. The same is true for $\nabla$.

Observe that if the long edge joining $\Delta$ and $\nabla$ has type $b$, then $\Delta(ab)$ and $\nabla(bc)$ certainly both lie in $\hexagon$ or $\hexagon'$; and the cyclic arrangement of $a,b,c$ edges around a hexagon implies that $\Delta(ab)$ and $\nabla(bc)$ are adjacent to the same hexagon. Conversely, if $\Delta(ab), \nabla(bc)$ are adjacent to the same hexagon, then the type of the long edge connecting them must be $\{a,b\} \cap \{b,c\} = b$.

Now suppose that $\Delta(ab)$ and $\nabla(bc)$ do not lie on the boundary of a single hexagonal face of $\tau_i$; equivalently, that the long edge joining $\Delta$ and $\nabla$ has type $a$ or $c$.

The side of $\Delta$ not in $\hexagon$ or $\hexagon'$, and the side of $\nabla$ not in $\hexagon$ or $\hexagon'$, have the same type (i.e.\ $ab$, $bc$ or $ca$). Thus at least one of $\Delta(ab), \nabla(bc)$ lie in $\hexagon$ or $\hexagon'$. But these sides do not lie in the same hexagon by assumption. Moreover, the side of $\Delta$ in $\hexagon$ (resp.\ $\hexagon'$), and the side of $\nabla$ in $\hexagon'$ (resp.\ $\hexagon$) again have the same type. Thus precisely one of $\Delta(ab), \nabla(bc)$ lie in a hexagonal face $\hexagon$ adjacent to $\Delta$ and $\nabla$. If $\Delta(ab)$ lies in $\hexagon$ but $\nabla(bc)$ does not, then the long side joining $\Delta$ and $\nabla$ must have type $a$, and the side of $\nabla$ adjacent to $\hexagon$ must be $\nabla(ca)$. Similarly, if $\nabla(bc)$ lies in $\hexagon$ but $\Delta(ab)$ does not, then the side of $\Delta$ adjacent to $\hexagon$ must be $\Delta(ca)$.
\end{proof}

We can use these observations to rewrite the sum over $\Delta \neq \nabla$ as a sum over hexagons.

\begin{lem}
\label{Lem:TrianglesToHexagons}
Let $\p$ be a polyhedron. Then
\[
\sum_{\Delta\neq \nabla} n_{\Delta(ab)} n'_{\nabla(bc)} - n_{\nabla(bc)} n'_{\Delta(ab)}  
= \sum_{\hexagon \subset \p}
\sum_{(k,l)}
n_{\hexagon(k)!} n'_{\hexagon(l)!} - n_{\hexagon(l)!} n'_{\hexagon(k)!}.
\]
\end{lem}
Here the first sum is over ordered pairs of distinct triangular faces $\triangle, \nabla$ of $\p$, the second sum is over hexagonal faces $\hexagon$ of $\p$, and the third sum is over $(k,l)$ in cyclic alphabetical order. Recall from \refsec{train_tracks_on_H} that $\gamma_{\hexagon(k)!}$ is the branch which runs from a triangle out to the short rectangle labelled $k$ adjacent to $\hexagon$. 
For $k \in \{ab,bc,ca\}$, $n_{\hexagon(k)!}$ (resp.\ $n'_{\hexagon(k)!}$) denotes the coefficient of $\gamma_{\hexagon(k)!}$ in $\zeta$ (resp.\ $\zeta'$).

\begin{proof}
Following Choi again, using \reflem{WhereSidesLie}, we write the sum over $\Delta \neq \nabla$ as a sum over hexagons. For any pair $\Delta \neq \nabla$, precisely one of the three cases of \reflem{WhereSidesLie} applies, and identifies a hexagon $\hexagon$, adjacent to both $\Delta$ and $\nabla$, of which at least one of $\Delta(ab)$ or $\nabla(bc)$ is a side.

When case (i) applies, the sides $\Delta(ab), \nabla(bc)$ lie in a common hexagon $\hexagon$, and thus $\gamma_{\Delta(ab)} = \gamma_{\hexagon(ab)!}$ and $\gamma_{\nabla(bc)} = \gamma_{\hexagon(bc)!}$. Indeed all such pairs in hexagonal faces arise precisely once, and so the terms in this case form the sum
\[
\sum_{\hexagon \subset \p}
n_{\hexagon(ab)!} n'_{\hexagon(bc)!} - n_{\hexagon(bc)!} n'_{\hexagon(ab)!}.
\]

When case (ii) applies, $\Delta(ab)$ and $\nabla(ca)$ lie in a common hexagonal face $\hexagon$, but $\nabla(bc)$ does not. Using \refeqn{triangle_sum_zero} we may replace $n_{\nabla(bc)}$ with $-n_{\nabla(ab)} - n_{\nabla(ca)}$, and similarly for $n'$, and observe that
\begin{align*}
n_{\Delta(ab)} & n'_{\nabla(bc)} - n_{\nabla(bc)} n'_{\Delta(ab)} \\
&= n_{\Delta(ab)} (-n'_{\nabla(ab)} - n'_{\nabla(ca)} ) -
(-n_{\nabla(ab)} - n_{\nabla(ca)}) n'_{\Delta(ab)} \\
&= n_{\nabla(ca)} n'_{\Delta(ab)} - n_{\Delta(ab)} n'_{\nabla(ca)}
- n_{\Delta(ab)} n'_{\nabla(ab)} + n_{\nabla(ab)} n'_{\Delta(ab)}
\end{align*}

Summing over all triangular faces $\Delta\neq\nabla$ of $\p$, the very last two terms sum to zero, as recasting the sum over $(\Delta, \nabla)$ as $(\nabla, \Delta)$ converts one into the other, and they cancel.

The remaining terms both involve $\Delta(ab)$ and $\nabla(ca)$, which refer to short sides of a common hexagon. Thus $n_{\Delta(ab)} = n_{\hexagon(ab)!}$ and $n_{\nabla(ca)} = n_{\hexagon(ca)!}$; similarly for $n'$. Indeed all such pairs of sides of hexagonal faces arise in his way. So the terms in this case form the sum
\[
\sum_{\hexagon \subset \p}
n_{\hexagon(ca)!} n'_{\hexagon(ab)!} - n_{\hexagon(ab)!} n'_{\hexagon(ca)!}.
\]

Similarly, in case (iii), $\nabla(bc)$ and $\Delta(ca)$ lie in a common hexagonal face $\hexagon$ but $\Delta(ab)$ does not. Replacing $n_{\Delta(ab)}$ with $-n_{\Delta(ca)} - n_{\Delta(bc)}$, and similarly for $n'$, yields
\begin{align*}
n_{\Delta(ab)} & n'_{\nabla(bc)} - n_{\nabla(bc)} n'_{\Delta(ab)} \\
&= (-n_{\Delta(ca)} - n_{\Delta(bc)}) n'_{\nabla(bc)}
- n_{\nabla(bc)} (-n'_{\Delta(ca)} - n'_{\Delta(bc)} ) \\
&= -n_{\Delta(ca)} n'_{\nabla(bc)} + n_{\nabla(bc)} n'_{\Delta(ca)} 
-n_{\Delta(bc)} n'_{\nabla(bc)} +n_{\nabla(bc)} n'_{\Delta(bc)},
\end{align*}
where again the last two terms cancel upon summation. Then $n_{\Delta(ca)} = n_{\hexagon(ca)!}$ and $n_{\nabla(bc)} = n_{\hexagon(bc)!}$, similarly for $n'$, and the 
 remaining terms sum to
\[
\sum_{\hexagon \subset \p}
n_{\hexagon(bc)!} n'_{\hexagon(ca)!} - n_{\hexagon(ca)!} n'_{\hexagon(bc)!}.
\]
Thus upon summing all terms we obtain the desired result.
\end{proof}

Combining lemmas \ref{Lem:OmegaOverTriangles} and \ref{Lem:TrianglesToHexagons} then immediately gives 
\begin{align}
\label{Eqn:omega_intermediate_form}
\begin{split}
\omega \left( h(\zeta), h(\zeta') \right) &= 
\sum_{\Delta}
n_{\Delta(ab)} n'_{\Delta(bc)} - n_{\Delta(bc)} n'_{\Delta(ab)} \\
& \quad + 
\sum_{\hexagon}
\sum_{(k,l)}
n_{\hexagon(k)!} n'_{\hexagon(l)!} - n_{\hexagon(l)!} n'_{\hexagon(k)!}
\end{split}
\end{align}
where $\sum_\Delta$ denotes a sum over all triangles of all polyhedra, $\sum_{\hexagon}$ denotes a sum over all hexagons of all polyhedra, and $\sum_{(k,l)}$ is again a sum over $\{(ab,bc),(bc,ca),(ca,ab)\}$, i.e.\ cyclically alphabetical orderings (\refdef{cyclic_alphabetical_order}).

In the following lemma, we expand the terms in the second sum using \refeqn{avast_exclamation} and make some cancellations to rewrite as as a sum over hexagons only involving terms of the form $n_k, n_{kl}$ (as in \refeqn{n_simplify-3}) and their $'$ versions.
\begin{prop}
\label{Prop:OmegaOverFaces}
With notation as above,
\begin{align}
\nonumber
\omega \left( h(\zeta), h(\zeta') \right) &= 
\sum_{\Delta}
n_{\Delta(ab)} n'_{\Delta(bc)} - n_{\Delta(bc)} n'_{\Delta(ab)} \\
\label{Eqn:omega_sum-1}
&+ \sum_{\hexagon} \sum_{(k,l,m)} n_k ( n'_{lk} - n'_{mk} ) - n'_k (n_{lk} - n_{mk}) \\
\label{Eqn:omega_sum-2}
&+ \sum_{\hexagon} \sum_{(k,l,m)} (n_{kl} + n_{km})(n'_{lk} + n'_{lm}) - (n'_{kl}+n'_{km})(n_{lk}+n_{lm}).
\end{align}
\end{prop}
Here $\sum_{(k,l,m)}$ is a sum over cyclically alphabetically ordered $k,l,m$.

We observe that just as the first line consists of contributions from branches of the train track in triangles, the second and third lines consist of contributions from branches of various types in short and long rectangles.

\begin{proof}
Comparing \refeqn{omega_intermediate_form} with the claim, it suffices to show that the second line of \refeqn{omega_intermediate_form} is equal to the sum of \refeqn{omega_sum-1} and \refeqn{omega_sum-2}. Using the compatibility conditions at 3-switches in short rectangles from \refeqn{avast_exclamation}, we replace $n_{\hexagon(k)!}$ with $n_{\hexagon(k)} + n_{\hexagon(k,l)} + n_{\hexagon(k,m)}$, where $(k,l,m)$ are cyclically alphabetically ordered. Using the simplified notation of \refeqn{n_simplify-1}--\refeqn{n_simplify-2}, this is just $n_k + n_{kl} + n_{km}$. Cyclically permuting $(k,l,m)$, and adding $'$ allows us to expand the second line of \refeqn{omega_intermediate_form} as
\begin{align*}
  \sum_{(k,l,m)} & ( n_k + (n_{kl} + n_{km}) ) (n'_l + (n'_{lm} + n'_{lk}) ) -
  ( n_l + (n_{lm} + n_{lk}) ) ( n'_k + (n'_{kl} + n'_{km}) ) \\
&= \sum_{(k,l,m)} n_k (n'_{lk} + n'_{lm}) - n_l (n'_{kl} + n'_{km}) 
	+ n'_l (n_{kl} + n_{km}) - n'_k (n_{lm} + n_{lk}) \\
	& \quad \quad \quad \quad + (n_{kl} + n_{km})(n'_{lk} + n'_{lm}) - (n_{lk} + n_{lm})(n'_{kl} + n'_{km}) \\
	& \quad \quad \quad \quad + n_k n'_l - n_l n'_k.
\end{align*}
In the first line of the resulting expression, cyclically permuting $(k,l,m)$ we observe that the sum of $n_l (n'_{kl} + n'_{km})$ is equal to the sum of $n_k (n'_{mk} + n'_{ml})$; similarly we can replace $n'_l (n_{kl} + n_{km})$ with $n'_k (n_{mk} + n_{ml})$. In the third line, we similarly replace $n_l n'_k$ with $n_k n'_m$. This gives the above summation as
\begin{align*}
&\sum_{(k,l,m)} n_k (n'_{lk} + n'_{lm} - n'_{mk} - n'_{ml}) - n'_k ( n_{lk} + n_{lm} -n_{mk} - n_{ml} ) \\
& \quad \quad + (n_{kl} + n_{km})(n'_{lk} + n'_{lm}) - (n_{lk} + n_{lm})(n'_{kl} + n'_{km}) \\
& \quad \quad + n_k (n'_l - n'_m).
\end{align*}
This sum contains all terms of \refeqn{omega_sum-1} and \refeqn{omega_sum-2}, but with a few extra terms; it now suffices to prove that
\[
\sum_{\hexagon} \sum_{(k,l,m)} n_k (n'_{lm} - n'_{ml}) - n'_k (n_{lm} - n_{ml}) + n_k (n'_l - n'_m) = 0.
\]
Regroup the sum over pairs of glued hexagons $\hexagon, \widehat{\hexagon}$ with labellings $k,l,m$ and $\widehat{k},\widehat{l},\widehat{m}$ as previously, with $k,l,m$ in cyclic alphabetical order, and $\widehat{k}, \widehat{l}, \widehat{m}$ in cyclic anti-alphabetical order. Then, for instance, $n_k n'_{lm}$ becomes $\widehat{n}_k \widehat{n}'_{ml}$ when it corresponds to a hexagon with a hat; similarly for other terms. The above sum thus becomes
\begin{align*}
&\sum_{\hexagon, \widehat{\hexagon}} \sum_{(k,l,m)} n_k (n'_{lm} - n'_{ml}) - \widehat{n}_k ( \widehat{n}'_{lm} - \widehat{n}'_{ml}) - n'_k (n_{lm} - n_{ml}) + \widehat{n}'_k (\widehat{n}_{lm} - \widehat{n}_{ml}) \\
& \quad \quad \quad + n_k (n'_l - n'_m) - \widehat{n}_k (\widehat{n}'_l - \widehat{n}'_m ).
\end{align*}
Now from the compatibility conditions at 1-switches of \reflem{oscillating_conditions}(iii) we have $\widehat{n}_k = - n_k$, so we obtain
\begin{align*}
&\sum_{\hexagon, \widehat{\hexagon}} \sum_{(k,l,m)} n_k (n'_{lm} - n'_{ml} + \widehat{n}'_{lm} - \widehat{n}_{ml}) - n'_k (n_{lm} - n_{ml} + \widehat{n}_{lm} - \widehat{n}_{ml}) \\
& \quad \quad \quad + n_k (n'_l - n'_m - n'_l + n'_m ).
\end{align*}
Finally, the compatibility condition at stations of \reflem{oscillating_conditions}(iv) is $n_{kl} - n_{lk} + \widehat{n}_{kl} - \widehat{n}_{lk} = 0$, and cyclic permutations and $'$ versions thereof. The sum thus collapses to zero.
\end{proof}

\subsection{Proof of main theorem}
\label{Sec:EvaluatingOmega}

To prove \refthm{MainThm}, we must show that the symplectic form $\omega(h(\zeta),h(\zeta'))$ is equal to $2 \zeta \cdot \zeta'$. We have written the former as a sum over triangles and hexagons in \refprop{OmegaOverFaces}, and the latter as a sum over triangles and hexagons in \refprop{IntOverFaces}. It is clear that the first sum \refeqn{intersection_number-1} of $2\zeta \cdot \zeta'$ matches the first sum of $\omega(h(\zeta),h(\zeta')$. Thus it remains to show that the sum of \refeqn{intersection_number-2} through \refeqn{intersection_number-4} is equal to the sum of \refeqn{omega_sum-1} and \refeqn{omega_sum-2}

Indeed, in the following two propositions we show that \refeqn{intersection_number-2} $=$ \refeqn{omega_sum-1} 
and \refeqn{intersection_number-3}$+$\refeqn{intersection_number-4} $=$ \refeqn{omega_sum-2}. These various equalities show the equivalence of various sub-counts of intersection numbers, with various contributions to $\omega$.

\begin{prop}
\begin{align*}
&\sum_{\hexagon, \widehat{\hexagon}} \sum_{(k,l,m)} n_k (n'_{kl} - n'_{km} + \widehat{n}'_{kl} - \widehat{n}'_{km}) + n'_k ( -n_{kl} + n_{km} - \widehat{n}_{kl} + \widehat{n}_{km} )  \\
&=
\sum_{\hexagon} \sum_{(k,l,m)} n_k ( n'_{lk} - n'_{mk} ) - n'_k (n_{lk} - n_{mk})
\end{align*}
\end{prop}

\begin{proof}
We regroup the right hand sum as a sum over pairs of glued hexagons $\hexagon, \widehat{\hexagon}$, with the usual matching labellings $k,l,m$ and $\widehat{k},\widehat{l},\widehat{m}$ which are cyclically alphabetical and anti-alphabetical respectively. As in the proof of \refprop{OmegaOverFaces}, a term such as $n_k n'_{lk}$ becomes $\widehat{n}_k \widehat{n}'_{mk}$ when in a hexagon with a hat.  The right hand sum is then given by
\begin{align*}
\sum_{\hexagon, \widehat{\hexagon}} \sum_{(k,l,m)}
n_k (n'_{lk} - n'_{mk}) - n'_k (n_{lk} - n_{mk})
- \widehat{n}_k (\widehat{n}'_{lk} - \widehat{n}'_{mk})
+ \widehat{n}'_k (\widehat{n}_{lk} - \widehat{n}_{mk})
\end{align*}
Matching conditions at 1-switches (\reflem{oscillating_conditions}(iii)) yield $\widehat{n}_k = -n_k$, $\widehat{n}'_k = - n'_k$ and we obtain
\begin{align*}
\sum_{\hexagon, \widehat{\hexagon}} \sum_{(k,l,m)}
n_k (n'_{lk} - n'_{mk} + \widehat{n}'_{lk} - \widehat{n}'_{mk}) + n'_k (- n_{lk} + n_{mk} - \widehat{n}_{lk} + \widehat{n}_{mk}).
\end{align*}
Finally, matching conditions at stations (\reflem{oscillating_conditions}(iv)) yield
\[
n_{lk} + \widehat{n}_{lk} = n_{kl} + \widehat{n}_{kl}, \quad
n_{mk} + \widehat{n}_{mk} = n_{km} + \widehat{n}_{km}.
\]
Substituting these and their primed versions, we obtain the desired left hand sum.
\end{proof}

\begin{prop}
\label{Prop:finalprop}
\begin{align*}
&\sum_{\hexagon, \widehat{\hexagon}} \sum_{(k,l,m)} -n_{kl} n'_{km} + n_{km} n'_{kl} - \widehat{n}_{km} \widehat{n}'_{kl} + \widehat{n}_{kl} \widehat{n}'_{km} \\
&+ \sum_{\hexagon, \widehat{\hexagon}} \sum_{(k,l,m)} 
n_{kl} n'_{lk} - n_{lk} n'_{kl} - \widehat{n}_{kl} \widehat{n}'_{lk} + \widehat{n}_{lk} \widehat{n}'_{kl} \\
&=
\sum_{\hexagon} \sum_{(k,l,m)} (n_{kl} + n_{km})(n'_{lk} + n'_{lm}) - (n'_{kl}+n'_{km})(n_{lk}+n_{lm}).
\end{align*}
\end{prop}

Before proving this, we prove a lemma.
\begin{lem}
\label{Lem:clever_rearrange}
In a pair of glued hexagonal faces, with labellings $k,l,m$ and $\widehat{k},\widehat{l},\widehat{m}$ as above, we have
\begin{align*}
&(n_{kl} + n_{km})(n'_{lm} - n'_{ml}) + (\widehat{n}_{kl} + \widehat{n}_{km})(\widehat{n}'_{ml} - \widehat{n}'_{lm}) \\
&= (n_{lk}+n_{mk})(n'_{lm}-n'_{ml}) + (\widehat{n}_{lk}+\widehat{n}_{mk})(\widehat{n}'_{ml}-\widehat{n}'_{lm})
\end{align*}
\end{lem}

\begin{proof}
From the compatibility condition at stations (\reflem{oscillating_conditions}(iv)) we have $n'_{lm} - n'_{ml} = \widehat{n}'_{ml} - \widehat{n}'_{lm}$. Moreover by the same compatibility conditions we have
\[
n_{kl} + \widehat{n}_{kl} - n_{lk} - \widehat{n}_{lk} = 0
\quad \text{and} \quad
n_{km} + \widehat{n}_{km} - n_{mk} - \widehat{n}_{mk} = 0.
\]
Thus
\begin{align*}
0 &= (n_{kl} + \widehat{n}_{kl} - n_{lk} - \widehat{n}_{lk} + n_{km} + \widehat{n}_{km} - n_{mk} - \widehat{n}_{mk})(n'_{lm} - n'_{ml}) \\
&= (n_{kl} + n_{km} - n_{lk} - n_{mk})(n'_{lm} - n'_{ml}) \\ 
& \quad \quad + (\widehat{n}_{kl} + \widehat{n}_{km} - \widehat{n}_{lk} - \widehat{n}_{mk})(\widehat{n}'_{ml} - \widehat{n}'_{lm}) \\
&= (n_{kl} + n_{km})(n'_{lm} - n'_{ml}) - (n_{lk} + n_{mk})(n'_{lm} - n'_{ml}) \\
& \quad \quad + (\widehat{n}_{kl} + \widehat{n}_{km})(\widehat{n}'_{ml} - \widehat{n}'_{lm})
- (\widehat{n}_{lk} + \widehat{n}_{mk})(\widehat{n}'_{ml} - \widehat{n}'_{lm}) \qedhere
\end{align*}
\end{proof}

\begin{proof}[Proof of \refprop{finalprop}]
Cyclically permuting $(k,l,m) \mapsto (m,k,l)$ in the second term, both terms in the right hand sum have factors of $(n_{kl} + n_{km})$, which becomes
\[
\sum_{\hexagon} \sum_{(k,l,m)} (n_{kl} + n_{km})(n'_{lk} + n'_{lm} - n'_{mk} - n'_{ml}).
\]
Now regrouping the sum over pairs of glued hexagons, with the usual labellings $k,l,m$ and $\widehat{k},\widehat{l},\widehat{m}$ we obtain
\begin{align*}
&\sum_{\hexagon,\widehat{\hexagon}} \sum_{(k,l,m)} (n_{kl}+n_{km})(n'_{lk} + n'_{lm} - n'_{mk} - n'_{ml}) \\
& \quad + (\widehat{n}_{kl} + \widehat{n}_{km})(\widehat{n}_{mk}'+\widehat{n}'_{ml}-\widehat{n}'_{lk}-\widehat{n}'_{lm}).
\end{align*}
Now using \reflem{clever_rearrange} the above sum becomes
\begin{align*}
&\sum_{\hexagon,\widehat{\hexagon}} \sum_{(k,l,m)}
(n_{kl}+n_{km})(n'_{lk}-n'_{mk}) + (n_{lk}+n_{mk})(n'_{lm}-n'_{ml}) \\
&\quad + (\widehat{n}_{kl}+\widehat{n}_{km})(\widehat{n}'_{mk}-\widehat{n}'_{lk}) + (\widehat{n}_{lk}+\widehat{n}_{mk})(\widehat{n}'_{ml}-\widehat{n}'_{lm}).
\end{align*}
Applying a cyclic permutation, we may rewrite the term $n_{mk}(n'_{lm}-n'_{ml})$ as $n_{kl}(n'_{mk}-n'_{km})$, the term  $n_{km}(n'_{lk}-n'_{mk})$ as $n_{lk}(n'_{ml}-n'_{kl})$, the term $\widehat{n}_{km}(\widehat{n}'_{mk}-\widehat{n}'_{lk})$ as $\widehat{n}_{lk}(\widehat{n}'_{kl}-\widehat{n}'_{ml})$ and the term $\widehat{n}_{mk}(\widehat{n}'_{ml}-\widehat{n}'_{lm})$ and $\widehat{n}_{kl} (\widehat{n}'_{km}-\widehat{n}'_{mk})$. The sum then becomes 
\begin{align*}
\sum_{\hexagon,\widehat{\hexagon}} \sum_{(k,l,m)}
n_{kl}(n'_{lk} - n'_{km}) + n_{lk}(n'_{lm}-n'_{kl})
+ \widehat{n}_{kl}(\widehat{n}'_{km}-\widehat{n}'_{lk}) + \widehat{n}_{lk}(\widehat{n}'_{kl}-\widehat{n}'_{lm}).
\end{align*}
Expanding and applying a final cyclic permutation to rewrite $n_{lk} n'_{lm}$ as $n_{km} n'_{kl}$ and $\widehat{n}_{lk} \widehat{n}'_{lm}$ as $\widehat{n}_{km} \widehat{n}'_{kl}$, we obtain the desired sum on the left.
\end{proof}

\section{A symplectic basis of oscillating curves}\label{Sec:SymplecticBasis}
Given a manifold $M$ with $n_\mfc \geq 1$ ends
and an ideal triangulation $\calT$, we have constructed a Heegaard surface $H$ with an oriented enhanced train track $\tau$. We now construct a collection of oscillating curves whose combinatorial holonomies form a standard symplectic basis for $\V$. These curves extend the rows of the standard Neumann-Zagier matrix for $M$ \cite{NeumannZagier}.

Throughout, we refer to the triangulation of a boundary component of $\overline{M}$ as the \emph{boundary triangulation}, and the cell decomposition on the intersection of a boundary component with the Heegaard surface as the \emph{Heegaard boundary decomposition}.

We will illustrate the process with the figure-8 knot complement. Figure~\ref{Fig:Fig8CuspOnly} shows the triangulation and boundary
triangulation of this knot complement. \reffig{Fig8Cusp_train_tracks} shows the Heegaard boundary decomposition, with train tracks.

\begin{figure}
\def\svgscale{1}
  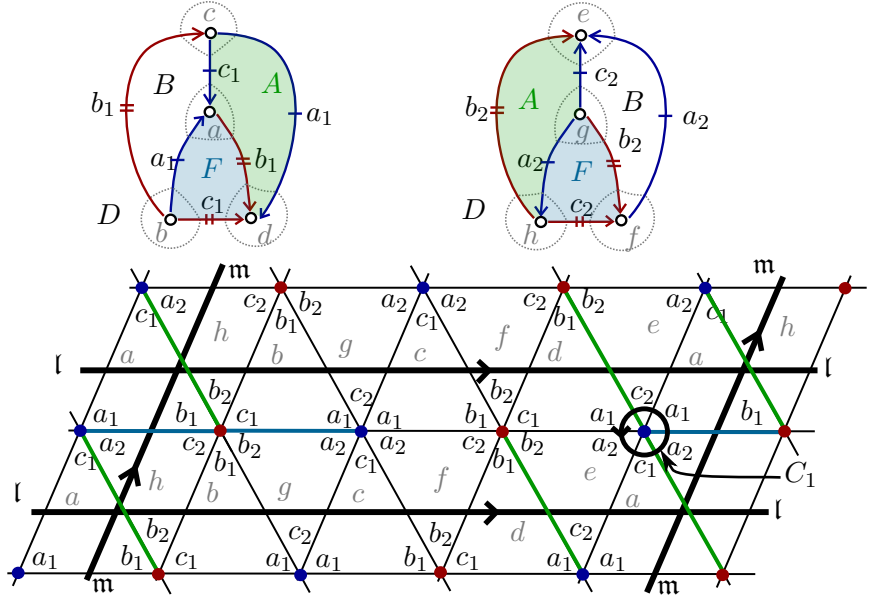
  \caption{The figure-8 knot complement is made of two tetrahedra, with cusp triangulation shown. A choice of boundary curves $\mfl$ and $\mfm$ are shown, and one of the edge curves $C_1$.}
  \label{Fig:Fig8CuspOnly}
\end{figure}

\begin{figure}
\includegraphics[width=\textwidth]{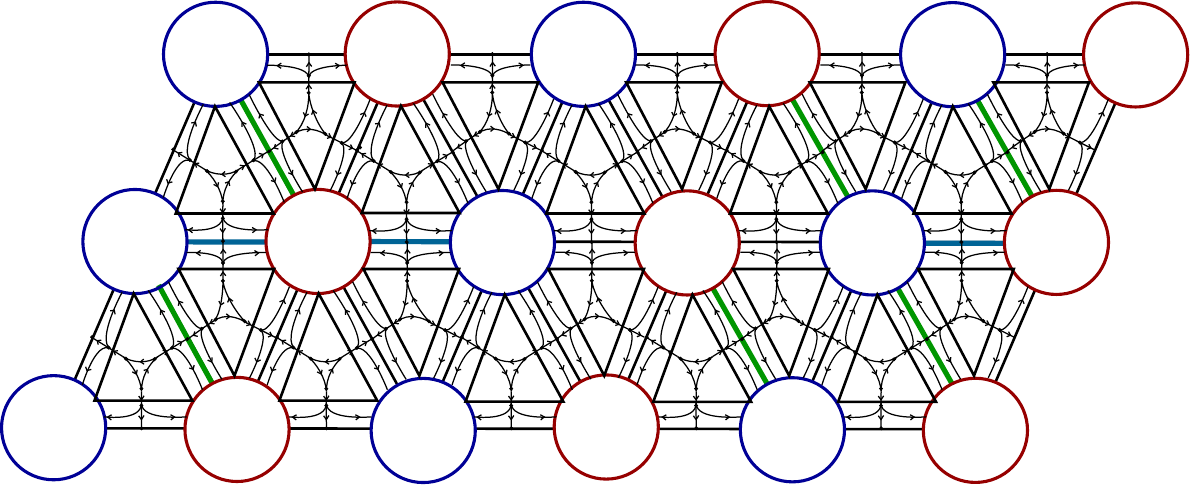}
  \caption{The same cusp as \reffig{Fig8CuspOnly}, after removing tubes around edges to give the handlebody decomposed into polyhedra. Cell decomposition along the cusp and train tracks are shown.}
  \label{Fig:Fig8Cusp_train_tracks}
\end{figure}

Each of the curves will be isotoped to lie on the train track $\tau$ constructed above. 

There are three types of curves in our basis. The first two are well known (at least when all ends are torus cusps) from the standard Neumann-Zagier matrix, as follows.

\smallskip

\emph{Boundary curves.} Denote the $n_\mfc$ boundary components of $\overline{M}$ by $S_1, \ldots, S_{n_\mfc}$ and the genus of $S_i$ by $g_i$.
On each boundary component $S_i$ we take $g_i$ pairs of oriented simple closed curves forming a symplectic basis for $H_1 (S_i)$ with respect to the usual intersection form.
We may realise these closed curves as smooth oriented curves along $\tau$ (hence oscillating curves), lying completely in the portion of $\tau$ in the 
end corresponding to $S_i$,
disjoint from branches running down tubes around ideal edges of the triangulation. 
We denote these oscillating curves 
$\mfm_1, \mfl_1, \ldots, \mfm_{g_i}, \mfl_{g_i}$ so that $\mfl_j \cdot \mfm_j = 1$ for $j=1, \ldots, g_i$, and all other intersection numbers are zero.

Note that cutting 
$S_i$ along the union of the $\mfl_j$ and $\mfm_j$ for $j=1, \ldots, g_i$ yields a $g_i$-punctured sphere,
which is connected, so any two points in this punctured sphere can be joined.
Thus any two points in $S_i$ can be joined by a curve which does not cross any $\mfl_j$ or $\mfm_j$.
Moreover, such curves can be taken to avoid vertices of the boundary triangulation.
For the figure-8 knot complement, $n_\mfc = 1$ and $g_1 = 1$, i.e.\ the boundary consists of a single torus; the boundary curves $\mfl, \mfm$ are shown in \reffig{Fig8CuspOnly}.
\smallskip

\emph{Edge curves}. These are the simplest curves. For each edge $E_i$, 
we simply take a curve which encircles this edge. Concretely, pick an endpoint of $E_i$ arbitrarily, which is a vertex in the boundary triangulation, and consider a circle around that vertex. It is realised by an oscillating curve of $\tau$ with zero intersection number with any $\mfm_j$ of $\mfl_j$, and it is disjoint from branches of $\tau$ that run down tubes around ideal edges of the triangulation. We denote this oscillating curve by $C_i$. For the figure-8 knot complement, the curve $C_1$ is shown, corresponding to the edge with one dash, in \reffig{Fig8CuspOnly}.

In fact, we will not keep all these curves: Neumann and Zagier showed that only a subset are linearly independent. We need to choose a collection of curves dual to those that form a linearly independent set. An advantage of our method below is that it picks out a linearly independent collection while selecting dual oscillating curves. 

\smallskip

\emph{Dual edge curves}. These curves are new to this paper. While they are not difficult to define, it takes a bit of work to prove that they have the required properties. In particular, they will be mutually disjoint, the $i$-th will intersect $C_i$ in exactly one point, and they will be disjoint from each of the boundary curves. Such a curve dual to the edge curve $C_i$ will run along branches of $\tau$ in the neighbourhood of the edge $E_i$. To turn this into a closed oscillating curve, we will need a set of edges through which the tube can return to the initial boundary component.
We define these now. 

\subsection{A tree-like set of edges}
Let $n$ be the number of tetrahedra in the ideal triangulation, and $n_E$ the number of edges. Recall that $n_\mfc$ denotes the number of ends.
Define the following graph.
\begin{definition}
The \emph{end (multi) graph} $G$ of $\calT$ is defined as follows.
\begin{enumerate}
\item The vertices of $G$ are the ends of $M$.
\item $G$ has one edge for each edge $E$ of $\calT$, with endpoints in $G$ joining vertices corresponding to the same ends meeting $E$.
\end{enumerate}
\end{definition}
In fact, $G$ is a quotient of the Heegaard surface $H$ by a map which collapses each boundary component of $\overline{M}$
to a point and collapses each annulus of $H$ around an edge of $\calT$ to an edge. For example, the end graph associated with the figure-8 knot complement consists of one vertex and two edges, each with both endpoints on that vertex. 

The graph $G$ is connected, as $M$ is connected.
Thus it has a spanning tree, and we have the following.

\begin{lem}
There exists a collection $\calE_\mfc'$ of $n_\mfc -1$ edges of $\calT$, forming a spanning tree for the end graph, which connect all ends of $M$.
\qed
\end{lem}
In other words, one can go from any end of $M$ to any other by travelling only along the chosen $n_\mfc - 1$ edges. As they form a tree, there is a unique sequence of edges connecting any two given ends (without backtracking).
For the figure-8 knot complement, the collection $\calE_\mfc'$ is empty, as a spanning tree consists only of the lone vertex. 

As all trees are bipartite, we may colour each end of $M$ black or white, so that each edge in $\calE_\mfc'$ connects ends of different colours.

\begin{lem}
There exists an edge $E_0$ of $\calT$ which connects ends of the same colour.
\end{lem}

\begin{proof}
If not, then $G$ is bipartite, hence contains no odd-length cycles. But $G$ contains many odd-length cycles, for instance a 3-cycle for each face of each tetrahedron of $\calT$.
\end{proof}

Thus $E_0$, and the edges of $\calE_\mfc'$ joining the endpoints of $E_0$, form an odd length cycle in $G$. Denote the set of edges $\calE_\mfc'\cup \{E_0\}$ by $\calE_\mfc$.

For the figure-8 knot complement, we may choose either edge to be $E_0$. We will choose it to be the edge with two dashes shown in \reffig{Fig8CuspOnly}. 

Aside from the edges of $\calE_\mfc$, there are there are $n_E -n_\mfc$ edges of $\calT$, which we denote $E_1, \ldots, E_{n_E-n_\mfc}$.

For the figure-8 knot, this is just $E_1$, consisting of the edge with one dash.

For each ideal edge of the triangulation, choose once and for all a long rectangle, with two branches of $\tau$ running through the long rectangle. We will require our curves to run over these branches of this long rectangle. 

\begin{lemma}\label{Lem:LongRectangle}
  Let $e$ be an ideal edge of a triangulation, and suppose a curve carried by $\tau$ runs through a fixed long rectangle $R$ adjacent to $e$, through both branches on $R$. The long rectangle $R$ lies in a tetrahedron $\tet$,
and is adjacent to a unique face $F$ of $\tet$. Then the curve must run through two triangles of the boundary triangulation adjacent to $e$ in $\tet$, entering the triangles along the side corresponding to the face $F$, and running across this side by way of a short rectangle containing one of the branches of $\tau$. 
\end{lemma}

\begin{proof}
  The branches of $\tau$ in the long rectangle $R$ each enter a short rectangle adjacent to a triangle in the truncated polyhedron $\p$ associated to $\tet$. 
The train track only permits them to run into the triangle, along the side adjacent to $F$. See \reffig{cusp_from_long_rect} left.
\end{proof}

\begin{figure}
\def\svgscale{1.2}
\begingroup%
  \makeatletter%
  \providecommand\color[2][]{%
    \errmessage{(Inkscape) Color is used for the text in Inkscape, but the package 'color.sty' is not loaded}%
    \renewcommand\color[2][]{}%
  }%
  \providecommand\transparent[1]{%
    \errmessage{(Inkscape) Transparency is used (non-zero) for the text in Inkscape, but the package 'transparent.sty' is not loaded}%
    \renewcommand\transparent[1]{}%
  }%
  \providecommand\rotatebox[2]{#2}%
  \newcommand*\fsize{\dimexpr\f@size pt\relax}%
  \newcommand*\lineheight[1]{\fontsize{\fsize}{#1\fsize}\selectfont}%
  \ifx\svgwidth\undefined%
    \setlength{\unitlength}{267.26876765bp}%
    \ifx\svgscale\undefined%
      \relax%
    \else%
      \setlength{\unitlength}{\unitlength * \real{\svgscale}}%
    \fi%
  \else%
    \setlength{\unitlength}{\svgwidth}%
  \fi%
  \global\let\svgwidth\undefined%
  \global\let\svgscale\undefined%
  \makeatother%
  \begin{picture}(1,0.45754287)%
    \lineheight{1}%
    \setlength\tabcolsep{0pt}%
    \put(0,0){\includegraphics[width=\unitlength,page=1]{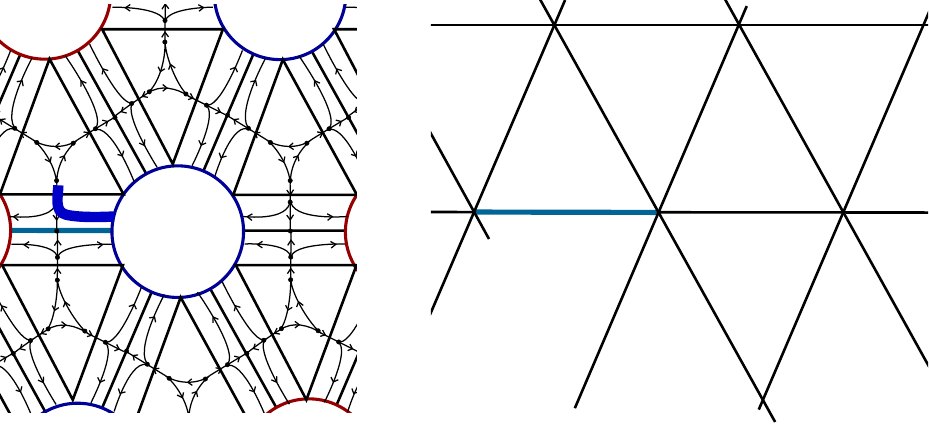}}%
    \put(0.60563156,0.19248352){\color[rgb]{0,0.39215686,0.58823529}\makebox(0,0)[lt]{\lineheight{0}\smash{\begin{tabular}[t]{l}$F$\end{tabular}}}}%
    \put(0.05650867,0.17143766){\color[rgb]{0,0.39215686,0.58823529}\makebox(0,0)[lt]{\lineheight{0}\smash{\begin{tabular}[t]{l}$F$\end{tabular}}}}%
    \put(0,0){\includegraphics[width=\unitlength,page=2]{cusp_from_long_rectangle.pdf}}%
    \put(0.1265985,0.21297273){\color[rgb]{0,0,0}\makebox(0,0)[lt]{\lineheight{1.25}\smash{\begin{tabular}[t]{l}$R$\end{tabular}}}}%
    \put(0,0){\includegraphics[width=\unitlength,page=3]{cusp_from_long_rectangle.pdf}}%
    \put(0.03620908,0.16614231){\color[rgb]{0,0.39215686,0.58823529}\makebox(0,0)[lt]{\lineheight{0}\smash{\begin{tabular}[t]{l}$F$\end{tabular}}}}%
    \put(0,0){\includegraphics[width=\unitlength,page=4]{cusp_from_long_rectangle.pdf}}%
  \end{picture}%
\endgroup%

  \caption{Entering the boundary triangulation from a long rectangle.}
  \label{Fig:cusp_from_long_rect}
\end{figure}

We therefore indicate a long rectangle by an arrow in the boundary triangulation that runs from one of the endpoints of the ideal edge $e$, across the side of a triangle that corresponds to the face $F$, and into the specified boundary triangle. See \reffig{cusp_from_long_rect} right.

For the figure-8 knot complement, and the edge with one dash $E_1$, select a long rectangle in the tetrahedron on the left of \reffig{Fig8CuspOnly} adjacent to the face labelled $F$ in that tetrahedron. Observe there is only one edge with one dash adjacent to the face $F$. In \reflem{LongRectangle}, this picks out the cusp triangles labeled $a$ and $b$, at either end of that edge. The corresponding vertex and side of the triangles is shown in \reffig{Fig8LongRect} in blue.

\begin{figure}
  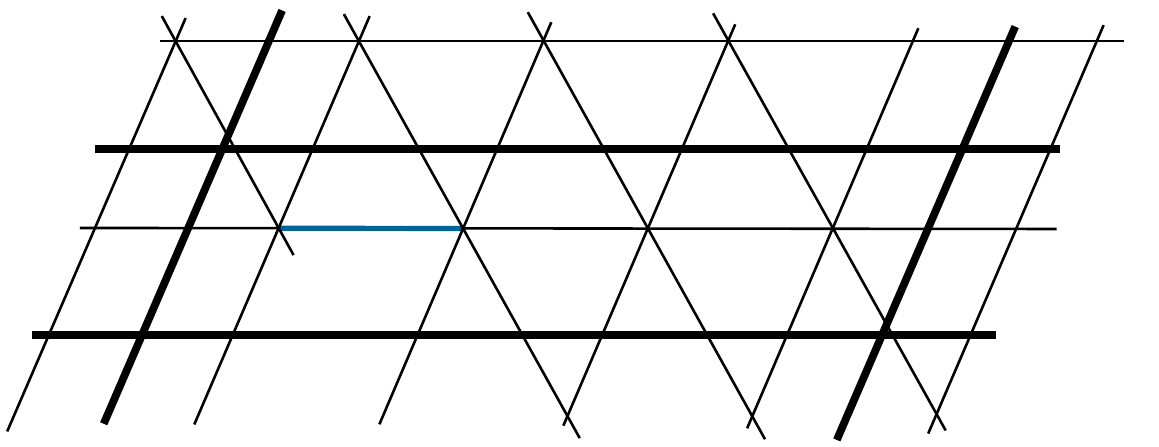
  \caption{Choosing long rectangles. For the edge with one dash $E_1$, choose a long rectangle in the tetrahedron on the left adjacent to the face $F$. This picks out cusp triangles $a$ and $b$, the side of the triangle corresponding to $F$, and the vertex of the triangle corresponding to $E_1$, shown by the thick blue arrows within these triangles. For the edge $E_0$, pick the face $A$ in the same tetrahedron. The cusp triangles, with vertices and sides are shown by the thick red arrows.}
  \label{Fig:Fig8LongRect}
\end{figure}

For the edge with two dashes $E_0$, select a long rectangle adjacent to the face labelled $A$ in the tetrahedron on the left of \reffig{Fig8CuspOnly}. This selects cusp triangles $a$ and $d$. The associated vertices and sides of these triangles are indicated by the thick red arrow in \reffig{Fig8LongRect}. 

Now we may construct the dual edge curves, for each of the edges $E_1, \ldots, E_{n_E-n_\mfc}$.

\begin{lemma}\label{Lem:CollectionGamma}
  There exists a collection of disjoint curves $\Gamma_1, \dots, \Gamma_{n_E-n_\mfc}$, each disjoint from all $\mfm_i$ and $\mfl_i$, with $\Gamma_i$ meeting $C_i$ exactly once, and disjoint from $C_j$ when $i\neq j$. Moreover, each such curve runs through an even number of long rectangles along branches in that long rectangle, and runs between distinct sides of each cusp triangle that it meets. 
\end{lemma}

\begin{proof}
We begin with $E_1$ and proceed inductively, forming a collection of disjoint closed curves, each disjoint from the boundary curves $\mfm_i$ and $\mfl_i$.

The edge $E_1$ connects two cusps, which may have the same or different colours. If the colours are different, then a collection of edges in 
$\calE_\mfc'$ (the spanning tree of $G$), along with $E_1$, forms a cycle of even length. If the colours are the same, then $E_1$, $E_0$, and a collection of edges in the spanning tree of $G$ (i.e.\ $E_1$ together with edges in $\calE_\mfc$) forms a cycle of even length. In either case, we will form a closed curve from these edges. In particular, the curve will run through the chosen long rectangle adjacent to $E_1$, across boundary cells avoiding $\mfm_i$ and $\mfl_i$, to the chosen long rectangles adjacent to selected edges of the $\calE_\mfc$.

By \reflem{LongRectangle}, a path through the chosen long rectangles adjacent to $E_1$ must run into a fixed triangle $\Delta_1$ of the boundary triangulation at a fixed side. Let $E\in\calE_\mfc$ be the next edge from $E_1$ in the cycle of edges. Again the choice of long rectangle for $E$ determines a fixed side of a fixed cusp triangle $\Delta$ adjacent to $E$. There exists an embedded path from the side of $\Delta_1$ adjacent to $E_1$,
to the side of $\Delta$ adjacent to $E$, that does not meet any $\mfm_i$, $\mfl_i$, or any vertices of the boundary triangulation, since the boundary triangulation cut along the $\mfm_i$, $\mfl_i$, and the vertices of the triangulation is path connected (indeed homeomorphic to a disc with some finite number of punctures).

We claim we may choose the path so that each arc of the path lying in a cusp triangle enters and leaves the triangle in distinct sides. For suppose not. If the arc is not the first or last between $\Delta$ and $\Delta_1$, then an innermost such arc cuts off a disk with the triangle side, which can be used to isotope the curve through, reducing the number of intersections with sides of the cusp triangulation. Repeating finitely many times removes all such arcs. It remains to consider the first arc or the last arc, with one endpoint on a fixed side of the fixed triangle $\Delta_1$, say. We may replace this arc by one that runs from the fixed side of $\Delta_1$ to the other side adjacent to the vertex corresponding to $E_1$, then around this vertex to the triangle in the opposite side. Because the path runs from $E_1$ into this triangle, we may choose this path in a neighbourhood of $E_1$, and therefore keep it disjoint from all other arcs in the path. From there, we may reconnect to the rest of the path, removing other bigons as needed. Similarly for the last arc, we may reroute the path around the vertex corresponding to $E$. In the case that $E$, $E_1$ lie on the same side of the same triangle, the effect will be to run the path once around that side, meeting distinct sides of triangles along the way.

Now the curve will run through the long rectangle adjacent to $E$, exiting the long rectangle at the other end of the annulus enclosing $E$, on a (not necessarily distinct) boundary component, in a branch of $\tau$ uniquely determining a side of a boundary triangle in this boundary component.
Repeat as above, connecting this side of the triangle to the next, choosing an arc disjoint from the $\mfm_i$, $\mfl_i$, any vertices of the boundary triangulation, and any previously selected arcs; again this is possible since the complement of the $\mfm_i$, $\mfl_i$, the vertices, and previous paths is a path connected space. Again we may choose the path to run through distinct sides of each triangle it enters. 

Repeat, continuing until we have run through each edge of the cycle, connecting back to the opposite end of the original ideal edge $E_1$, thus forming a closed curve consisting of arcs in long rectangles, and arcs in boundary triangles that run between distinct sides of each boundary triangle. Because the cycle has even length, it runs through an even number of long rectangles. Call the curve $\Gamma_1$. 

Now suppose disjoint curves $\Gamma_1$, $\dots$, $\Gamma_{i-1}$ are selected for edges $E_1$, $\dots$, $E_{i-1}$, respectively. We build a curve associated with the edge $E_i$ as above, by choosing arcs between boundary triangles determined by long rectangles. 

It may be the case that the curve we construct runs over a long rectangle corresponding to an edge $E$ of $\calE_\mfc$ that we have used in constructing a previous curve $\Gamma_j$. In that case, we choose an arc through the long rectangle to run parallel to $\Gamma_j$, but disjoint from it. Note this also fixes endpoints on the boundary triangulation that are disjoint from $\Gamma_j$, but adjacent to it. However, this does not affect our ability to choose disjoint arcs as above. In this manner we obtain disjoint curves $\Gamma_1, \dots, \Gamma_{n_E-n_{\mfc}}$.
\end{proof}

A choice of such curves for the figure-8 knot complement is shown in \reffig{Fig8OscCurve}. 

\begin{figure}
  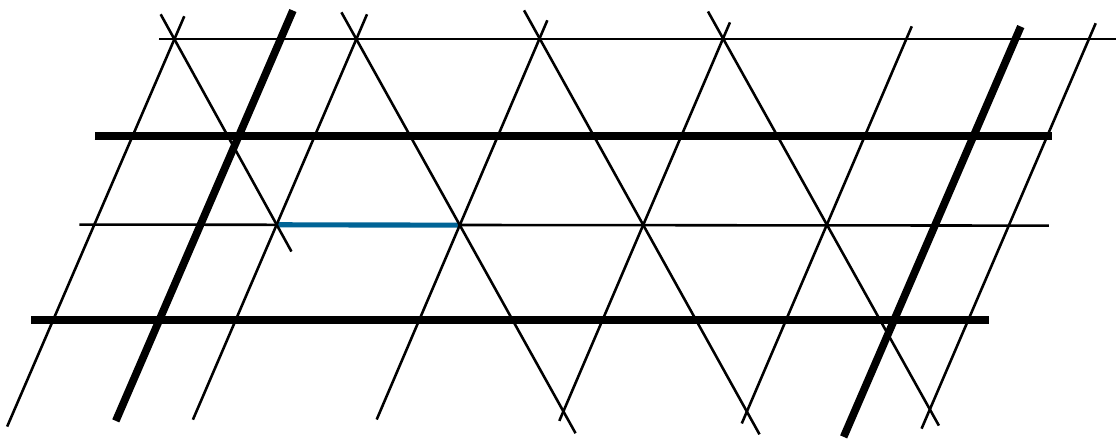
  \caption{The arcs of the dual oscillating curve $\Gamma$ for the figure-8 knot.}
  \label{Fig:Fig8OscCurve}
\end{figure}

\begin{lemma}
The curves $\Gamma_1, \dots, \Gamma_{n_E-n_\mfc}$ may be isotoped on each boundary component of $\bdy\overline{M}$ to lie in the train track $\tau$, and oriented to become oscillating curves. Moreover, this can be done in such a way that the oriented intersection numbers satisfy $\Gamma_i\cdot C_j = \delta_{ij}$, $\Gamma_i\cdot\mfm_j=\Gamma_i\cdot\mfl_j=0$. 
\end{lemma}

\begin{proof}
Each curve runs through boundary components and long rectangles. Within the long rectangles, we isotope the curves onto the two branches in the long rectangle. After exiting a long rectangle, the curves must run into a short rectangle and then into a boundary triangle. Because we have chosen arcs in the boundary triangulation to run between distinct sides of the triangles, each such arc can be isotoped into the corresponding triangle of the Heegaard boundary composition, into the unique branch running between the corresponding sides. Then between boundary triangles, it runs through the unique branches of the short rectangles that take it to the next cusp triangle. Finally, at each station in a long rectangle, we change the orientation of the curve. Because it meets an even number of stations, this defines the required orientation on an oscillating curve.

To verify intersection numbers, observe that rather than isotoping each curve onto a branch of the train track, we may instead isotope them into a neighbourhood of the branch, so that they run in parallel in this neighbourhood without affecting their geometric intersection numbers. Similarly for the curves $C_i$, $\mfm_j$, $\mfl_j$. The intersection number of oscillating curves then agrees with algebraic intersection number, being careful to take into account the reversal of orientation when proceeding through a station from one end of a long rectangle to another.
\end{proof}

Recall that we define $g = \sum_{i=1}^{n_\mfc} g_i$ to be the sum of genera of boundary components.
We can now prove the following main theorem.
\begin{thm}\label{Thm:OscAlgorithm}
There exists a constructive algorithm to determine oscillating curves $C_1, \dots, C_{n_E-n_\mfc}$ associated with edge gluings, and dual oscillating curves $\Gamma_1, \dots, \Gamma_{n_E-n_\mfc}$, so that these curves along with a standard basis for $H_1 (\partial M)$, 
$\mfm_1$, $\mfl_1$, $\ldots$, $\mfm_{g}$, $\mfl_{g}$ form a symplectic basis, in the sense that
\begin{align*}
\omega(h(C_i), h(G_j)) &= 2 C_i \cdot G_j = 2 \delta_{ij} \\
\omega(h(\mfm_i), h(\mfl_j)) &= 2 \mfm_i \cdot \mfl_j = 2 \delta_{ij}
\end{align*}
and $\omega$, evaluated on the combinatorial holonomy of any other pair of curves, is zero.
\end{thm}

\begin{proof}
By construction, the boundary curves form a symplectic basis for $H_1 (\partial M)$. By construction, the intersection numbers are as claimed; by \refthm{MainThm}, $\omega$ evaluates as claimed. We thus have $2(n_E - n_\mfc) + 2g$ oscillating curves satisfying the desired symplectic properties; to see that this is a basis then it suffices to show that this number equals $2n$, the dimension of $\V$.

This follows from an Euler characteristic count. 
On the one hand, $\chi(\partial M)$ is given by
\[
\sum_{i=1}^{n_\mfc} \chi(S_i) = \sum_{i=1}^{n_\mfc} (2-2g_i) = 2 n_\mfc - 2g.
\]
On the other hand, the boundary triangulation has $2 n_E$ vertices (one at each end of each ideal edge), $6n$ edges (there are $4n$ faces of tetrahedra, which glue into $2n$ faces of the ideal triangulation, each contributing $3$ edges to the boundary triangulation) and $4n$ triangles ($4$ from each tetrahedron). Hence
\[
\sum_{i=1}^{n_\mfc} \chi(S_i) = 2n_E - 6n + 4n = 2 n_E - 2n.
\]
Equating these two expressions gives $2n = 2 n_E - 2 n_\mfc + 2g$ as desired.
\end{proof}

\begin{proof}[Proof of \refthm{OscAlgorithmTori}]
This is simply the special case of \refthm{OscAlgorithm} when the boundary consists only of tori. We have all $g_i = 1$ so that $g=n_\mfc$, and the equation at the end of the previous proof reduces to $n=n_E$.
\end{proof}

\section{Matrices associated with oscillating curves}\label{Sec:Matrices}

We have constructed curves and their duals. In this section we recall the definitions of the incidence matrix, Neumann--Zagier matrix, and use our new tools to construct an (almost, up to factors of 2) symplectic matrix associated with a triangulation of a 3-manifold with torus boundary component(s). 

For a given ideal triangulation, the basis $\a_i$, $\b_i$, $\cc_i$ are encoded by vectors:
\[
\a_1=(1, 0, 0, 0, \dots, 0), \b_1=(0,1,0,0,\dots, 0), \cc_1=(0,0,1,0,\dots, 0),
\]
and so on. Using this basis, the oscillating curves of \refdef{abstract_oscillating_curve} can be written as vectors. Indeed, for the curves $C_i$ running once around an edge, and cusp curves $\mfm_i$, $\mfl_i$, these vectors are typically listed in a matrix $\In$, called the \emph{incidence matrix} associated to the gluing data. For example, the incidence matrix can be obtained in the computer software SnapPy~\cite{SnapPy}.

Using the fact that $\a_i+\b_i+\cc_i=0$, Neumann and Zagier adjusted the basis to consist of $\a_i-\cc_i$ and $\b_i-\cc_i$ for all $i$. The Neumann-Zagier matrix is obtained from the incidence matrix by subtracting the terms associated with $\cc_i$ from those associated with $\a_i$ and $\b_i$.

Theorem~\ref{Thm:OscAlgorithm} implies that we may complete the Neuman-Zagier matrix to a symplectic matrix, by augmenting it with rows associated with the curves $\Gamma_i$.

\begin{example}
  For the figure-8 knot complement, and the curves identified in the previous section, we obtain the following incidence vectors associated to the curves:
\begin{align*}
  \mfm: &  (0, -1, 0, 1, 0, 0) \\
  \mfl: &  ( 0, -2, -2, 0, 2, 2) \\
  C: &     ( 2 ,0, 1, 2, 0, 1) \\
  \Gamma:& ( -1, -1, -1, -1, 1, 1) 
\end{align*}

The rows of the Neumann-Zagier matrix are obtained by subtracting the entry corresponding to $\cc_i$ from those corresponding to $\a_i$ and $\b_i$, for each $i$. We make this modification for each of the vectors above. Then by our main theorem, \refthm{MainThm}, the resulting matrix is symplectic (with extra factor of 2). 
\[
\SY = \kbordermatrix{
     & \Delta_1 & & \Delta_2 &\\
\mfm & 0 & -1 & 1&  0 \\
\mfl & -2 & 0 & 2 & 0 \\
\Gamma& 0 & 0 & -2 & 0 \\
C     & 1 & -1 & 1 & -1
}
\]
\end{example}

Recall that a complete hyperbolic structure associated with a triangulation satisfies edge gluing equations and cusp equations. In logarithmic form, this is encoded by the incidence matrix, and the Neumann--Zagier matrix. In particular, a hyperbolic structure is encoded by a vector
\[ Z' = (Z_1, Z_1', Z_1'', \dots, Z_{n-n_\mfc}, Z_{n-n_\mfc}', Z_{n-n_\mfc}'') \]
satisfying
\[
\In \cdot Z' = i\pi C,
\]
where $C$ is the vector $(2i\pi, \dots, 2i\pi, 0, \dots, 0)^T$, with entries $2i\pi$ for each row associated with an edge gluing equation, or $C_i$, and $0$ for each row associated with a cusp, and satisfying $Z_j+Z_j'+Z_j''=i \pi$ for all $j$, and
$\exp(Z_i) + \exp(Z_i^{-1}) -1 = 0$.

This can be rewritten using the Neumann--Zagier matrix:
\[ \NZ \cdot Z = i\pi C^\flat \]
where $Z$ consists only of terms $Z_j, Z_j'$ from $Z'$, and $C^\flat$ is obtained from $C$ by subtracting from each row the sums of the integers in every third column from the same row in the incidence matrix (the integers corresponding to the $\cc_j$).

\begin{theorem}\label{Thm:SolvingNZ}
Let $\SY$ be the $2n\times 2n$ matrix whose rows correspond to gluing equations and their symplectic duals arising from oscillating curves, and curves forming a symplectic basis for $H_1(\partial M)$.
Then a solution $Z$ to the gluing and cusp equation satisfies
\[ 2Z = i\pi (-J(\SY)^T J)\overline{C} \]
Here $\overline{C}$, $\SY$, $J$ are all integral matrices, so our solutions are integer or half integer multiples of $i\pi$. 
\end{theorem}

\begin{proof}
Adjust the matrix $\NZ$ to be invertible by augmenting it with the symplectic duals to gluing equations arising from oscillating curves. Add $0$'s to each corresponding row of $C^\flat$ to obtain $\overline{C}$. Then the logarithmic gluing and cusp equations take the form
\[ \SY\cdot Z = i\pi \overline{C}. \]
By \refthm{MainThm}, $\SY$ is a symplectic matrix up to a factor of 2, satisfying:
\[ (-J(\SY)^T J) \cdot \SY = 2\Id. \]
Therefore multiplying each side of the equation by $(-J(\SY)^T J)$ gives the result.
\end{proof}

We conclude with one additional example.

\subsection{The Whitehead link}

The Whitehead link complement is well known to have a decomposition into a single regular ideal octahedron; this is described in Thurston's 1979 lecture notes, for example~\cite{thurston}. 

Here, we will work through an example of a less familiar decomposition of the Whitehead link complement, obtained in \cite{HowieMathewsPurcell}. This is a decomposition into five tetrahedra and five ideal edges, with the property that one of the cusps only has one ideal edge running into it, with the other endpoint of that edge on the other cusp. This edge, call it $E_1$, must lie in the collection $\calE_\mfc$. Three additional edges of the triangulation border tetrahedra meeting the edge $E_1$. These edges are associated to slopes in the Farey graph in \cite{HowieMathewsPurcell}; we therefore refer to them by the same notation as that paper, labeling them $E_\infty$, $E_{2/1}$ and $E_{3/1}$. The final edge has both endpoints on the second cusp. We call this edge $E_0$ and add it to $\calE_\mfc$.

Choose meridian and longitude for the cusps as in \cite{HowieMathewsPurcell}. We have edge curves $C_\infty$, $C_{2/1}$, $C_{3/1}$ associated with edges $E_\infty$, $E_{2/1}$ and $E_{3/1}$, respectively. We now find oscillating curves $\Gamma_\infty$, $\Gamma_{2/1}$, $\Gamma_{3/1}$.

Begin by selecting a long rectangle for $E_0$. This is equivalent to choosing a face of a tetrahedron adjacent to $E_0$, and a direction, and connecting it to both endpoints of $E_0$. We select the face $0(123)$ in Regina notation, which lies between tetrahedra $0$ and $1$ and is indicated in \reffig{whitehead_cusp_basis}.

\begin{figure}
  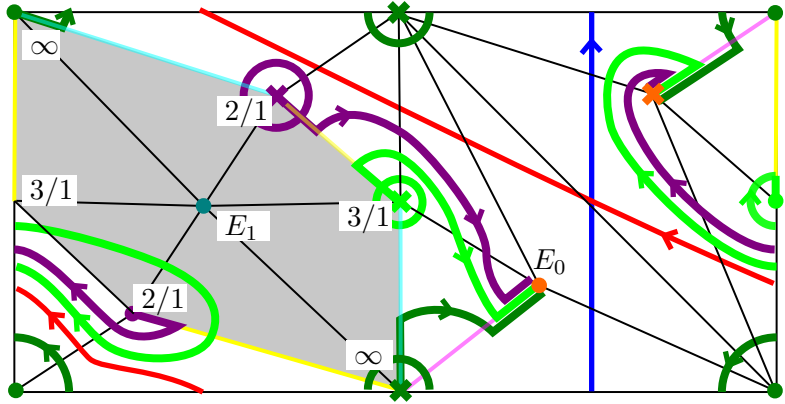
  \caption{The cusp triangulation and symplectic basis of oscillating curves for the more interesting cusp of the 5-tetrahedron decomposition of the Whitehead link. Shown are curves $\mfm_1$ (blue), $\mfl_1$ (red), $C_\infty$ (dark green), $C_{2/1}$ (purple), $C_{3/1}$ (light green), $\Gamma_{\infty}$ (dark green), $\Gamma_{2/1}$ (purple), and $\Gamma_{3/1}$ (light green). Note that $\Gamma_\infty$, $\Gamma_{2/1}$ and $\Gamma_{3/1}$ all travel along $E_0$ but do not intersect there; they emerge from both ends of $E_0$ in the same order.}
  \label{Fig:whitehead_cusp_basis}
\end{figure}

For each of $E_\infty$, $E_{2/1}$ and $E_{3/1}$, select a face and an orientation on the face, and connect that face to the associated edge in the cusp diagram.

Now form $\Gamma_\infty$ by drawing paths in the cusp from one of our endpoints in associated with $E_\infty$ to one associated with $E_0$, disjoint from $\mfm_1$ and $\mfl_1$. This is shown in dark green in \reffig{whitehead_cusp_basis}. Similarly form $\Gamma_{2/1}$, shown in medium green, and $\Gamma_{3/1}$. Observe that we have made several choices when picking these curves, and there are likely choices that would have led to simpler curves. However, the ones shown will suffice.

The associated vectors in the incidence matrix are as follows:
\begin{align*}
  \Gamma_\infty: & (0,1,1, 1,1,0, 2,1,0, 0,0,0,0,0,0) \\
  C_\infty: & (0,1,2, 1,2,0, 2,1,0, 0,0,1, 0,0,1) \\
  \Gamma_{2/1}: & (0,0,1, -1,0,-1, 1,-1,-1,0,0,0,0,0,0) \\
  C_{2/1}: & (1,0,0, 0,0,1, 0,0,0, 0,1,0, 0,1,0) \\
  \Gamma_{3/1}: & (0,1,1, 0,0,-1, 1,0,-1, 0,-1,0, 0,-1,0) \\
  C_{3/1}: & (0,1,0, 1,0,0, 0,1,0, 1,0,0, 1,0,0) \\
  \mfm_1: & (0,0,1, 0,-1,0, 0,0,0, 0,0,0, 0,0,0) \\
  \mfl_1: & (0,-1,1, 1,-1,0, 0,0,0, 0,0,0, 0,0,0) \\
\end{align*}
Additionally, there are vectors $\mfm_0$ and $\mfl_0$ as follows, not shown here but shown in \cite{HowieMathewsPurcell}:
\begin{align*}
  \mfm_0: & (0,0,0, 0,0,0, 0,0,0, 1,0,0, -1,0,0) \\
  \mfl_0: & (0,0,0, 0,0,0, 0,0,0, 0,1,0, 0,-1,0)
\end{align*}

Then the Neumann-Zagier matrix can be extended to the following almost symplectic matrix.
\[
\kbordermatrix{
     & \Delta_0 & & \Delta_1 & & \Delta_2 && \Delta_3 && \Delta_4\\
\mfm_0 & 0&0 & 0&0 & 0&0 & 1&0 & -1&0 \\
\mfl_0 & 0&0 & 0&0 & 0&0 & 0&1 & 0&-1 \\
\mfm_1 & -1&-1 & 0&-1& 0&0& 0&0 & 0&0 \\
\mfl_1 & -1&-2 & 1&-1 & 0&0 & 0&0 & 0&0 \\
\Gamma_\infty& -1&0 & 1&1 & 2&1 & 0&0 & 0&0 \\
C_\infty     & -2&-1& 1&2 & 2&1 & -1&-1 & -1&-1 \\
\Gamma_{2/1} & -1&-1& 0&1 & 2&0 & 0&0 & 0&0 \\
C_{2/1} & 1&0 & -1&-1 & 0&0 & 0&1 & 0&1 \\
\Gamma_{3/1} & -1&0 & 1&1 & 2&1 & 0&-1 & 0&-1 \\
C_{3/1} & 0&1 & 1&0 & 0&1 & 1&0 & 1&0
}
\]

\small

\bibliography{biblio}
\bibliographystyle{amsplain}

\end{document}